\documentclass[12pt]{article}

\usepackage[T1]{fontenc}
\usepackage[latin1]{inputenc}
\usepackage[english]{babel}
\usepackage{lmodern}

\usepackage{amsmath,amsthm,amssymb,mathrsfs}
\usepackage{stmaryrd} 
\usepackage{accents}  
\usepackage{bm}

\usepackage[normalem]{ulem} 
\usepackage{xcolor}         

\usepackage{indentfirst}
\usepackage{enumitem}
\usepackage[font=small,labelfont=bf]{caption}

\usepackage[numbers,sort&compress]{natbib}    
\usepackage{breakcites} 

\usepackage{tikz}
\usetikzlibrary{decorations.pathreplacing}

\newtheorem{thm}{Theorem}[section]
\newtheorem{defi}[thm]{Definition}
\newtheorem{lma}[thm]{Lemma}
\newtheorem{cor}[thm]{Corollary}
\newtheorem{re}[thm]{Remark}

\newcommand{\diam}{\operatorname{diam}}
\newcommand{\law}{\mathcal{L}}

\newcommand{\pr}{\mathbb{P}}
\newcommand{\prob}{\pr}
\newcommand{\E}{\mathbb{E}}
\newcommand{\mean}{\E}

\newcommand{\real}{\mathbb{R}}
\newcommand{\N}{\mathbb{N}}
\newcommand{\Z}{\mathbb{Z}}

\newcommand{\Var}{\operatorname{Var}}
\newcommand{\indep}{\perp\!\!\!\perp}
\newcommand{\ubar}[1]{\underaccent{\bar}{#1}}

\newcommand{\bone}{\mathbf{1}}

\newcommand{\cF}{\mathcal{F}}

\DeclareMathOperator*{\argmin}{arg\,min}

\def\g{{\gamma}}
\def\iid{{i.i.d.}}

\newcommand{\G}{\Gamma}

\def\sxi{{\xi((x,M_x),\hP,\Gamma_{\lambda})}}

\newcommand{\Vol}{\operatorname{Vol}}
\def\card{{\operatorname{card}}}

\renewcommand{\P}{\mathcal{P}}
\def\scrB{{\mathcal{B}}}
\def\hP{{\hat{\P}}}
\def\scrP{{\mathcal{P}}}

\def\[{\left[}
\def\]{\right]}
\def\({\left(}
\def\){\right)}

\newcommand{\vertiii}[1]{{
		\left\vert\kern-0.25ex
		\left\vert\kern-0.25ex
		\left\vert #1
		\right\vert\kern-0.25ex
		\right\vert\kern-0.25ex
		\right\vert}}

\newcommand{\cov}{\operatorname{Cov}}
\def\var{{\operatorname{Var}}}

\newcommand{\yellow}[1]{\textcolor{yellow}{#1}}

\newcommand{\ignore}[1]{}

\def\qed{\hfill\hbox{${\vcenter{\vbox{
					\hrule height 0.4pt
					\hbox{\vrule width 0.4pt height 6pt
						\kern5pt
						\vrule width 0.4pt}
					\hrule height 0.4pt}}}$}}

\newcounter{con}
\stepcounter{con}
\newcommand{\qcon}[1]{\addtocounter{con}{1}}

\newcounter{as}
\stepcounter{as}
\newcommand{\qas}[1]{\addtocounter{as}{1}}

\newcommand{\R}{\mathbb{R}}
\newcommand{\M}{\mathbb{M}}
\def\0{{\mathbf{0}}}
\newcommand{\hatX}{\hat{\mathcal{X}}}

	\setlist[description]{topsep=0pt, partopsep=0pt}
	
	\numberwithin{equation}{section}
	
	\ignore{
		\makeatletter
		\renewcommand\section{\@startsection {section}{1}{\z@}%
			{-3.5ex \@plus -1ex \@minus -.2ex}%
			{1.3ex \@plus.2ex}%
			{\center\small\sc\MakeTextUppercase}}
		\def\subsection#1{\@startsection {subsection}{2}{0pt}%
			{-3.5ex \@plus -1ex \@minus -.2ex}%
			{1ex \@plus.2ex}%
			{\bf\mathversion{bold}}{#1}}
		\def\subsubsection#1{\@startsection{subsubsection}{3}{0pt}%
			{\medskipamount}%
			{-10pt}%
			{\normalsize\itshape}{\kern-2.2ex. #1.}}
		\makeatother
	}
	
	\allowdisplaybreaks[4]

\begin{document}

\title{\sc\bf\large  Quantitative CLT{\small s} for Geometric Statistics of Dependent Marked Point Processes
}

\author{
	Tianshu Cong\footnote{School of Mathematics, Jilin University, JL 130012, China. Email: congtianshu@jlu.edu.cn.  Work supported by National Natural Science Foundation of China (Grant No. 12401184), and Jilin Province Postdoctoral Research Start-up Fund.}\\ {\it Jilin University}\\
	Aihua Xia\footnote{School of Mathematics and Statistics,
		The University of Melbourne,
		VIC 3010, Australia. E-mail: aihuaxia@unimelb.edu.au. Work supported in part by Australian Research Council Grant No DP220102666.}\\ {\it University of Melbourne}
	\\ 
	J.  E. Yukich\footnote{Department of Mathematics, Lehigh University, Bethlehem, PA 18015, USA. E-mail: joseph.yukich@lehigh.edu.  Work supported in part by a Simons Research grant.
	} \\ {\it Lehigh University}
}

%
%

\date{\today}
\maketitle
\vskip-1cm
\begin{abstract} 
	Given a geometric statistic  expressible as a sum of scores which depend on local data, \citet{BYY19} established central limit theorems for centered and normalized versions of these statistics, subject to the underlying point process having fast decay of correlations and also subject to a variance growth condition.   Building on this, \citet{CX23} derived rates of normal approximation, as measured by the Wasserstein distance, for statistics of point processes exhibiting fast decay of dependence. Here we go further and  establish weak mixing conditions yielding rates of normal convergence for geometric statistics of marked point processes. The  mixing conditions, which are in terms of verifiable geometric criteria,  provide rates of normal approximation in the Wasserstein distance for statistics of a large class of point processes with dependent marks. Examples include statistics of  determinantal point processes, unevenly spaced time series, continuum percolation, interacting diffusions on spatial random graphs, as well as local $U$-statistics.
\end{abstract}

\vskip8pt \noindent\textit{Key words and phrases:} Wasserstein distance; Stein's method; fast decay of dependence; determinantal point process; unevenly spaced time series; stabilization; spatial random graphs; interacting diffusions; U-statistics.
	
\vskip8pt\noindent\textit{AMS 2020 Subject Classification:} Primary 60F05, Secondary 60D05; 60G55; 62E20; 05C80.

\section{Introduction and main results}

Consider a random spatial model, possibly time-dependent, and  driven by an underlying ground point process  $\P$ in $\R^d$, which we always take to be stationary. 
The  point process $\P$  may have spatial dependencies and it may carry marks, which may  depend on each other and  on the ground process; we shall write  $\hP$ to denote such marked point processes.   We are interested in the normal approximation of statistics of  spatial random models driven by $\hP$, including  basic counting statistics for $\hP$, statistics of time series driven by  $\hP$,  as well as  local $U$-statistics  and  statistics of interacting diffusions on random graphs defined on $\hP$. 

In all cases we assume that the statistics take the form  
$$
 H_{\lambda} := \sum_{x\in {{\scrP} \cap\Gamma_{\lambda}}}\sxi,$$ where
$\Gamma_{\lambda} \subset \R^d$, $\lambda\in (0,\infty),$  are observation windows,
the marks $M_x$ are random elements in a mark space, and the \emph{score} $\xi((x,M_x),\hP,\Gamma_{\lambda})$ describes  the interaction between a marked point $(x,M_x) \in  \hP$  and the process  $\hP$ itself,  with possible additional dependence on
$\Gamma_{\lambda}$. The scores  $\xi((x,M_x),\hP,\Gamma_{\lambda}),\  x \in \P$, are in general spatially dependent.
The statistics  $H_{\lambda}$  typically depend on the geometry of the point process  $\P$ and the underlying model and hence are  called  \emph{geometric statistics}. If   $\xi  \equiv 1$  then $H_{\lambda}$ 
reduces to the counting statistic over the window $\Gamma_{\lambda}$.

The study of geometric statistics has a long history and has roots in the works of \citet{BHH59}
 and \citet{Steele81}.  \citet{BB} were among the first to establish a central limit theorem for a geometric statistic, namely they showed that the standardized and centered sum of edge lengths of the $k$-nearest neighbor graph on an i.i.d.  random sample converges to a normal random variable as the sample size tends to infinity. In these seminal papers, as well as in the others that followed, the central limit theory for $(H_{\lambda})_{\lambda\in(0,\infty)}$ focused on statistics consisting of sums of score functions on Poisson or binomial input,  possibly with independent marks, and such that the score functions $\xi$ depended on local data in the sense that they satisfy stopping set stabilization criteria.  
In this paper we do not restrict to Poisson or binomial input, nor do we restrict to independent marks or to score functions with stopping sets.

When  $\P$ has fast decay of correlations and when $\xi$ satisfies stopping set exponential stabilization, then $((H_{\lambda} - \E H_{\lambda})/\sqrt{\Var H_{\lambda}})_{\lambda \in(0,\infty)}$ satisfies a qualitative central limit theorem as in \cite{BYY19}.  For  $\P$ satisfying an exponential mixing condition in the total variation distance (known as EDD; see Section~\ref{section1.4}) then convergence rates  for  $((H_{\lambda} - \E H_{\lambda})/\sqrt{\Var H_{\lambda}})_{\lambda \in(0,\infty)} $ to the standard normal under the Wasserstein distance  were obtained in \cite{CX23}.  The class of point processes was further enlarged in \citet{BYY25}, allowing qualitative central limit theorems for $((H_{\lambda} - \E H_{\lambda})/\sqrt{\Var H_{\lambda}})_{\lambda \in(0,\infty)}$ with $\P$ replaced by a marked point process $\hP$  {\em in which points and marks may be dependent},  and with $\xi$  required to satisfy a localization condition weaker than stabilization and 
allowing for long-range interactions, namely with $\xi$ required to be {\em distributionally close}  to a family of short-range scores.  Short-range scores are those depending only on data within a fixed deterministic radius.   

The purpose of this paper is to establish {\em quantitative} central limit theorems in this more general set-up, namely for statistics $H_\lambda$ driven by a dependently marked point process $\hat\P$, where $\hat\P$ has fast decay of dependence, its ground process $\P$ is stationary, and the score functions are closely approximated by a family of short-range scores.  Our approach extends \citet{SY13}, \citet{XY15}, and \citet{HOS25} in  three ways:  we do not restrict to input which is a rarefied Gibbsian point process, we allow  the input to have dependent marks,  and we require that the score function  satisfy localization criteria which are weaker than the  stabilization criteria required in these papers.  
We extend  \citet{CX24} by allowing for dependent marks and by relaxing the stabilization conditions on the score function in that paper. Finally, our work shows that some of the qualitative central limit theorems in \cite{BYY25} can be upgraded to quantitative central limit theorems when the ground point process $\P$ is EDD. In the EDD setting for $\P$, we treat a broad class of score functions with fewer Palm requirements than those considered  in \cite{BYY25}. The moment conditions on scores are less stringent and  require neither the existence of   high-order Palm distributions nor integrability of the score with respect to these distributions. On the other hand, the score functions considered here are required to satisfy stabilization criteria with respect to a local estimator, 
either in a probabilistic sense (cf.  Definition~\ref{prdistance}) or in an $L^p$ sense (cf. Definition \ref{lpstab}). 
We refer to Section~\ref{section1.4} for further remarks comparing our conditions with the existing ones in the literature.

We illustrate the versatility of our approach by obtaining quantitative central limit theorems for
a range of functionals, some of which fall outside the scope of  statistics encountered in classical stochastic geometric models: (i) counting statistics of determinantal point processes with kernels decaying exponentially fast, (ii) statistics of unevenly spaced time series,  (iii) statistics of continuum percolation , (iv) statistics of interacting diffusions on spatial random graphs with EDD input, and (v) certain local U-statistics of EDD marked point processes.

We deduce these quantitative central limit theorems from our main finding, which is that when the sums $H_{\lambda}$  satisfy a mixing condition on well-separated disjoint sets, then one obtains  rates of  normal approximation for  $(H_{\lambda} - \E H_{\lambda})/\sqrt{\Var H_{\lambda}} $, where the rate is with respect to the Wasserstein distance.   The mixing condition stipulates  that the covariance of functions of sums of scores on disjoint sets decays rapidly with respect to the distance between the sets, uniformly over a class of test functions and test sets. Mixing is shown to hold whenever  $\P$ satisfies EDD and whenever the score $\xi$ is well enough approximated by a family of short-range scores,  a geometric condition which is fulfilled whenever $\xi$ satisfies the standard stopping set stabilization criterion or even a weaker criterion whereby $\xi$ is well approximated by a family of  short-range scores either in the probabilistic or $L^p$ sense.

\subsection{Terminology}
	
We consider a ground point process $\P$ on $\real^d$, $d \geq 1$, and take it to be a simple stationary point process with intensity $\mu \in (0,\infty)$.  We let $\hP = \sum_{x \in \P} \delta_{ (x,M_x) }$ be a {\em simple, marked point process} on $\R^d \times \mathbb{M}$,  where $\mathbb{M}$ is a measurable space, $\delta$ is the Dirac measure, $ (x, M_x)$  are random elements in  $\R^d \times \mathbb{M}$, and for any $x \in \R^d$, the set $\{x\} \times \mathbb{M}$ contains at most one element of $\hP$. Elements $x \in \scrP$ are  referred to as {\em points}, while $M_x$ and $\mathbb{M}$ are the  {\em marks} and the {\em mark space}, respectively.  We do not  require independence of marks. The space $\mathbb{M}$ is general and, for example, could  be the space of time-evolving functions, which would be the case when each  $M_x, x \in \P,$ is a stochastic process representing  the state of an interacting particle system or the sample path of an interacting diffusion. The projection of $\hP$ on $\R^d$ is the ground process $\P$. The marked point process $\hP$ is a random element in the space ${\cal N}_{\R^d \times  \mathbb{M}}$ of locally finite subsets of $\R^d \times  \mathbb{M}$  endowed with the evaluation $\sigma$-algebra. Let $\Gamma_{\lambda}:=[ -0.5\lambda^{1/d}, 0.5\lambda^{1/d}]^d,\ \lambda\in(0,\infty)$, denote observation windows. 

A score  is a measurable function $\xi:  (\R^d \times \mathbb{M}) \times  {\cal N}_{\R^d \times  \mathbb{M}} \times {\cal B}(\real^d)  \to \R$ where ${\cal B}(\real^d)$ denotes the Borel $\sigma$-algebra on  $\real^d$. The value of  $\xi((x, m_x), \hat{\cal X}, A),$ 
where $\hat{\cal X} \in   {\cal N}_{\R^d \times  \mathbb{M}}$
and $A\in {\cal B}(\real^d)$, 
is relevant only for $(x, m_x) \in  \hat{\cal X}\cap (A\times \mathbb{M})$ and by convention we set it to be  $0$ when $(x,m_x) \not \in  \hat{\cal X} \cap (A\times \mathbb{M})$.
We put
\begin{equation*}
	H_{\lambda} := \sum_{x\in {{\scrP} \cap\Gamma_{\lambda}}}\sxi,  \   \sigma_{\lambda}^2:=\Var H_{\lambda}, \ Z_{\lambda} := \frac{H_{\lambda}-  \mathbb{E}H_{\lambda}} {\sigma_{\lambda}}.
\end{equation*}
We wish to establish rates of convergence of $(Z_{\lambda})_{\lambda\in(0,\infty)}$ to the normal.  
When $\xi((x, m_x), \hat{\cal X}, A)$ does not depend on $A\in {\cal B}(\real^d)$, then
the analysis of $(H_{\lambda})_{\lambda\in(0,\infty)}$ is easier because it is unaffected by the  position of $x$ relative to the boundary of $\Gamma_{\lambda}$.  We include the third parameter because, in some applications, the region $\Gamma_{\lambda}$  may  influence the value of the score function as would be the case for  statistics of the Voronoi tessellation  \cite[Section~$3.2$]{CX24}. While many of  our examples hold for the canonical choices of $\xi$ given by
$\sxi$ and  $\xi((x, M_x), {\hP}):=\xi((x,M_x),\hP,\mathbb{R}^d)$, we shorten our presentation and typically state results for only one of the two scores.  When marks are not present, all statements remain valid by viewing the ground process as a marked point process with degenerate marks. In this case, we write the score  as $\xi(x,\P,\Gamma_{\lambda})$ and all other notation remains unchanged.

For any real-valued random variable $X$ and $p \in [1, \infty)$  we put  $\|X\|_p=(\mean(|X|^p))^{1/p}$.  For any full-dimensional $B\in {\scrB} (\real^d)$, let $\Vol(B)$ denote its volume and for
${\cal X} \in {\cal N}_{\R^d}$, let  $\card({\cal X})$ be its cardinality.   

\begin{defi}\label{defimomp1} ($p_1$-moment on the ground  process)
	The ground point process $\scrP$  has a $p_1$-moment, $p_1 \in [1, \infty)$, if
	\begin{equation} \label{momground}
		b_1:= b_1(p_1):=   \|\card({\scrP} \cap \Gamma_1)\|_{p_1}<\infty.
	\end{equation}
\end{defi}

\begin{re}\label{remark1.1}  
The moment condition \eqref{momground}  holds for {all} $p_1 \in [1, \infty)$ for homogeneous Poisson point processes and for $\alpha-$determinantal point processes whose  kernel decays exponentially fast, provided  $-1/\alpha\in \mathbb{N}$, as established in Section~2.2.2 and Section~2.1  (Remark~(i)) of \cite{BYY19}. For a Gibbs point process with Papangelou intensity bounded above by a locally integrable function $h : \R^d \to \R$, iterating the GNZ equation gives the $n$-th factorial moment $\mathbb{E}[(\P(B))^{(n)}]\le(\int_Bh(x) \,\mathrm{d}x)^n$ for every bounded Borel set $B$; see, e.g., \citet[Section~2]{D19}, hence
$\card(\P\cap\Gamma_1)$ has finite moments of all orders. This includes many commonly used Gibbs point processes; for example, those whose interaction range is smaller than the critical value of continuum percolation of an associated Poisson point process \cite[Section~4.1]{HOS25}.
\end{re}

The ground point process $\scrP$ is of a general nature, and one may  require  conditions on the higher-order Palm distributions of $\scrP$ when proving limit theorems.  However, the class of score functions satisfying conditions based on higher-order Palm distributions is narrower than the class based on first-order Palm distributions (see the examples in Remark~\ref{re1.5}~(iv) and Remark~\ref{diff1stPalm2ndPalm}). While the first-order Palm distributions of ground point processes $\scrP$ are often well-studied, their higher-order Palm distributions are less well-known and typically cumbersome. Except for Poisson and related point processes, obtaining closed forms or analyzing the properties of these distributions is often more challenging than estimating higher-order moments of the score function $\xi$ with respect to the first-order Palm distributions of $\scrP$. Thus, 
we circumvent the use of higher-order Palm distributions by instead assuming only higher moment conditions on the ground point process $\scrP$ and on the score function $\xi$ with respect to the first-order Palm distributions;  see \cite[Lemma~5.5]{CX23}.

For any fixed distinct $x_1,...,x_q \in \R^d$, we denote by $\mathbb{P}_{x_1,...,x_q}$ the Palm probability given points 
$x_1,...,x_q$;  likewise  $\E_{x_1,...,x_q}$ and  ${\cal L}_{x_1,...,x_q}$ denote the Palm expectation and Palm distribution, respectively. 

\begin{defi}\label{defimomp2} ($p_2$-moment on the score)
	The score $\xi$ has a $p_2$-moment, $p_2 \in [1, \infty)$, if
	\begin{equation} \label{momscore}
		b_2 := b_2(p_2):= \sup_{\lambda\in(0,\infty)} \sup_{x\in\Gamma_{\lambda}} \(\E_x|\xi((x,M_x), \hP, \Gamma_{\lambda})|^{p_2}\)^{1/p_2} <\infty.
	\end{equation}
\end{defi} 

Given  $A\in {\cal B}(\real^d)$,  the sum of the score functions on the point set $\scrP \cap A$ is  
\begin{equation} \label{HlaA}  
	H_{\lambda,A}:=\sum_{x\in\scrP\cap A}\sxi,\ \ \ H_{\lambda}:=H_{\lambda,\G_{\lambda}}
\end{equation}
and its centered version is 
\begin{equation} \label{SlaA}
	S_A:=S_{\lambda,A}:=H_{\lambda,A}-\mathbb{E} H_{\lambda,A}.
\end{equation}
For $A,B\in {\cal B}(\real^d)$, let $\diam(A)=\sup\{d(x,y):\ x,y\in A\}$ be the diameter of $A$, $d(A,B)=\inf\{d(x,y):\ x\in A,y\in B\}$, $d(x,B)=d(B,x):=d(\{x\},B)$, where $d$ is the Euclidean distance on $\real^d$.  Given  $\rho\in (0,\infty)$ let  $B_\rho(A)=\{y:\ d(y,A)<\rho \}$ be the $\rho$-enlargement of $A$. Define the class of test functions
$$ \cF_b:=\{g: \mathbb{R} \rightarrow \mathbb{R};~ \|g\|_\infty\vee \|g'\|_\infty:=\sup_{x\in\real}|g(x)|\vee \sup_{x\in\real}|g'(x)| \le 1\}. $$

The following assumption, similar to the exponential decay of dependence in \cite[(2.1)]{CX23}, specifies a weak mixing condition on the {score} sums $\{S_A\}_{ A \in {\cal B}(\real^d)}$,  which, together with moment assumptions on the ground point process and the scores, yields  proximity bounds between $(H_{\lambda} - \E H_{\lambda})/\sqrt{\Var H_{\lambda}}$ and the standard normal in terms of the Wasserstein distance.  The mixing condition requires that the covariance of $S_A$ and $f(S_B)$, as well as the covariance of $S_A^2$ and $f(S_B)$, decays fast with respect to the distance between $A$ and $B$, uniformly for $f$ in a class of test functions and uniformly for $A, B$ belonging to a class of test sets.

\begin{defi} \label{WSM}  (mixing of score sums) 
Let ${\cal H}_1$ be the class of open hyperrectangles in $\real^d$ with sides at least unit length and with  faces parallel to coordinate planes in $\real^d$, and let ${\cal H}_2$ be the class of unions of two incident hyperrectangles in ${\cal H}_1$.  Let $ {\cal F}_1$  denote  the test function class  $\cF_b$ together with the identity $g(x) = x$ and let   $ {\cal F}_2 =  \cF_b$. The  score  sums satisfy  exponential mixing if there exists $\lambda_0 \in(0,\infty)$ and positive constants  $  \{ \g_i \}_{i = 0}^{3}$  such that for  $\lambda\ge \lambda_0$, $j \in \{1,2\}$ and   all $g\in {\cal F}_j$, $A\in  {\cal H}_j$, and all sets $B\in{\cal B}(\Gamma_\lambda)$ that are the finite union of disjoint hyperrectangles in ${\cal H}_1$
satisfying the separation condition
 $d(A,B)\ge \g_3 [\log(\diam(A)\vee \diam(B))\vee 1]$, we have 
	\begin{align}
		&\left|\mathbb{E}(S_A^jg(S_B)) - \mathbb{E}(S_A^j) \mathbb{E}(g(S_B)) \right|=\left|\cov(S_A^j,g(S_B))\right| \nonumber\\
		&\le \g_1(\diam(A)^{\g_0}\vee 1)(\diam(B)^{\g_0}\vee 1)\mathrm{e}^{-\g_2d(A,B)}=:U({\bm\g},A,B),
		\label{mixing-score}
		\end{align}
	where ${\bm\g} :=  \{ \g_i \}_{i = 0}^3$. The score sums satisfy polynomial mixing if  \eqref{mixing-score}  holds with $U({\bm\g},A,B)$ replaced by
	\begin{equation}\label{mixing-score'}
		U'({\bm\g},A,B) :=  \g_1(\diam(A)^{\g_0}\vee 1)(\diam(B)^{\g_0}\vee 1)d(A,B)^{-\g_2}, 
	\end{equation}
	where the parameters ${\bm\g} $ in \eqref{mixing-score} and \eqref{mixing-score'} are different in general.
\end{defi}

In Section 1.3 we  determine conditions on the ground point process  $\scrP$ and on the scores $\xi$ which imply the mixing bounds 
\eqref{mixing-score}  and \eqref{mixing-score'}.

\begin{re}\label{remark1.5a}
	\begin{description}
		\item{(i)} (the covariance conditions  \eqref{mixing-score} and \eqref{mixing-score'}) 
		The mixing of score sums implies covariance conditions of a more general form for a larger class of sets as in  Lemma~\ref{mixing-score-a1}, which plays an essential role in proving  Theorem~\ref{mainthm1}, our main result.
If either  \eqref{mixing-score} or \eqref{mixing-score'} hold for all pairs of disjoint sets $A$ and $B$, then there is no need to specify the parameter $\gamma_3$ controlling the separation condition.
		\item{(ii)} (comparison with mixing of score sums without centering) The mixing properties \eqref{mixing-score} and \eqref{mixing-score'} are equivalent, up to increasing the exponent $\gamma_0$ by an additive term $d$, to the corresponding properties obtained by replacing
		$S_A$ and $S_B$ with $H_{\lambda,A}$ and $H_{\lambda,B}$, respectively. See Lemma~\ref{mixing-uncentered}.
		\item{(iii)} (comparison with the $(BL,\psi,\theta)$ dependence condition)  
		The $(BL,\psi,\theta)$ dependence condition is an assumption imposed on covariances involving a specified class of test functions satisfying certain bounded Lipschitz conditions as  introduced in the context of random field theory;  cf. \citet[p.~93--95]{BA07} and the references therein. This assumption is used to obtain a rate of convergence for the normal approximation error of the partial sum of a centered $(BL,\psi,\theta)$-dependent random field under the Kolmogorov distance via Stein's method, see \cite[p.~182--188]{BA07} for details. Our framework  requires verifying a mixing property for a smaller class of test sets and a different class of test functions.
	\end{description} 
\end{re}

\subsection{A general quantitative CLT under  mixing of score sums}
We use the Wasserstein distance to quantify the difference between two real-valued random variables $X_1$ and $X_2$, where
\begin{equation*}\ignore{\label{Wass01}}
	d_W(X_1,X_2):=\sup_{h\in \cF_{\rm Lip}}|\mathbb{E}h(X_1)- \mathbb{E} h(X_2)|,
\end{equation*} 
where $\cF_{\rm Lip}:=\left\{f: ~\real\rightarrow\real, \quad  |f(x)-f(y)|\le |x-y|,\ \mbox{for all }x,y\in\real\right\}$.
We quantify the $d_W$ distance between $(H_{\lambda} - \E H_{\lambda})/\sqrt{\Var H_{\lambda}}$ and the standard normal $N(0,1)$, here denoted by $Z$.  The rates of normal approximation in the following general result as well as in the corollaries and applications assume that  $\Var H_{\lambda} =\Omega(\lambda^\nu)$ for some $\nu>2/3$.  We regard the verification of this condition as a separate issue and do not systematically address this.

\begin{thm}\label{mainthm1}(mixing of score sums implies a quantitative CLT)
	Assume that the point process $\scrP$ and score $\xi$ respectively satisfy the moment assumptions  \eqref{momground}  and \eqref{momscore}  for  $p_1$ and $p_2$ satisfying 
	$$\frac{p_1p_2}{p_1+p_2-1} \ge 3.$$
	Assume that $\Var H_{\lambda} =\Omega(\lambda^\nu)$ for some $\nu>2/3$. 
	\begin{description}
		\item{(a)} If the score sums satisfy exponential mixing \eqref{mixing-score}, then \begin{equation}\label{mainthm1-01}
			d_W\( \frac{H_{\lambda} - \E H_{\lambda}} {\sqrt{\Var H_{\lambda}}}, Z\) = O\( (\log \lambda)^{3d} \lambda^{-1.5\nu+1}\). 
		\end{equation}
		\item{(b)} If  polynomial mixing  \eqref{mixing-score'} holds and 
		\begin{equation}\label{relationship}
			\gamma_2 > \max\left\{\frac{3(-\nu d+2d+2\gamma_0)}{3\nu-2}, \frac{6(-\nu d+d+\gamma_0)}{3\nu-2}+\g_0\right\},
		\end{equation}
		or equivalently
		\begin{equation}\label{taudef}
			\tau(\g_0,\g_2):=\frac{3}{2}\(\frac{2\nu d+2\g_0}{3d+\g_2}\vee\frac{\nu d+2\g_0}{3d+\g_2-\g_0}\)< \frac32\nu-1,
			\end{equation} 
		then 
		\begin{equation}\label{mainthm1-02}
			d_W\( \frac{H_{\lambda} - \E H_{\lambda}} {\sqrt{\Var H_{\lambda}}}, Z\) = O\(\lambda^{-1.5\nu+1+\tau(\g_0,\g_2)}\)=o(1).
		\end{equation}
	\end{description}
\end{thm}

\begin{re}\label{re1.3}
	\begin{description}
		\item{(i)} (the optimality of rates under exponential mixing) 
		It is not uncommon that $\nu=1$, cf. \cite{BYY19}.  In this case  the $O((\log \lambda)^{3d}\lambda^{-0.5})$ bound  \eqref{mainthm1-01}  is nearly optimal, save for the  logarithmic factor.  In general one cannot do better than establishing a rate  $O(\lambda^{-0.5})$ as seen in Section 2.3 of~\citet{SchulteY}, which treats Poisson input.
		\item{(ii)} (the rate of convergence under polynomial mixing) In the case of polynomial mixing, i.e., under condition \eqref{mixing-score'}, as $\g_2\to\infty$, the bound \eqref{mainthm1-02} approaches $O\(\lambda^{-1.5\nu+1}\)$, which, in general, is unimprovable  when $\nu=1$. 
		\item{(iii)} (moment conditions for the ground point process and  for the scores) 
		If the ground point process, respectively the score, has a $p= (3 + \delta)$-moment for some positive $\delta$, then the score, respectively ground point process, must have  a  $3(p-1)/(p-3)$ moment. For the point processes discussed in Remark~\ref{remark1.1}, one can take $p_1$  in Definition 1.1  to be arbitrarily large. Consequently, it suffices that the score satisfies \eqref{momscore}  for some $p_2>3$.  Likewise, if a score has moments of all orders then $\P$ need only have a moment of order $p_1>3$.  When compared with  \cite{CX23}   this gives  normal approximation results holding under  slightly more flexible moment assumptions on  the point process and the score.
		\item{(iv)} (ground point process with infinite second moment) 
		One may easily  construct a simple stationary point process $\P$ on $\real^d$ such that  $\|\card({\scrP} \cap \Gamma_1)\|_2=\infty$. For example, start with a unit intensity Poisson point process $\Psi$  on $\real^d$ representing locations of parents. Next, independently of other parents, each parent produces $\eta$ children with distribution
		$$\mathbb{P}(\eta=i)=\frac{4}{i(i+1)(i+2)},\ i \in \N.$$
		Assuming the locations of the children of a parent at $x$ are independent and identically distributed within the open ball 
centered at $x$ with radius $1$, and independently of the children of other parents, then the point process $\scrP$ formed by the locations of the children of $\Psi$ is a simple stationary point process on $\real^d$ with intensity $2$ and satisfies $\mathbb{E}(\card({\scrP} \cap \Gamma_1)^2)=\infty$.
	\end{description}
\end{re}

\subsection{Corollaries} \label{section1.3}
We give conditions on the marked point process $\hat{{\cal P}}$ and on the score $\xi$ which together ensure mixing of score sums as in Definition \ref{WSM}.
Our first condition is adapted from \cite{CX23}.
\begin{defi}\label{EDD}(EDD marked point processes) 
	The marked point process $\hat{{\cal P}}$ exhibits exponential decay of dependence (EDD) if there exist constants $\theta_0\in [0,\infty)$, $\theta_i\in (0,\infty)$ for $1\le i\le 3$, such that for any  $A$, $B\in {\cal B}(\mathbb{R}^d)$ with $d(A,B)\ge \theta_3[\log(\diam(A)\vee \diam(B))\vee 1]$,  we have
	\begin{equation}\label{mixing}
		\beta_{A,B}:=d_{TV}\(\law(\hat{{\cal P}}|_{A\cup B}),\law(\hat{{\cal P}}|_{A}\cup\tilde{\hat{{\cal P}}}|_{B})\)\le \theta_1(\diam(A)^{\theta_0}\vee 1)(\diam(B)^{\theta_0}\vee 1)\mathrm{e}^{-\theta_2 d(A,B)},
	\end{equation}
	where $\tilde{\hat{{\cal P}}}$ is an independent copy of $\hat{{\cal P}}$. A point process ${\cal P}$ is EDD if it satisfies \eqref{mixing} with degenerate marks.
\end{defi}

\begin{re}\label{reEDD} 
	By definition, if a marked point process $\hP$ is EDD, then so is the ground process ${\cal P}$.   EDD point processes include Gibbs point processes  whose  interaction range is smaller than the critical value of continuum percolation of an associated Poisson point process as well as determinantal point processes whose kernel satisfies 
	\begin{equation}\label{reEDDdoubleexp}
		|K(x,y)|\le C_1\mathrm{e}^{-C_2\mathrm{e}^{C_3|x-y|}}, \quad  x,y \in \R^d,  
	\end{equation} 
	for some positive constants  $C_1,C_2$ and $C_3$, see \cite[Appendix~A]{HOS25} and \cite[Section~3.2]{CX23}. Disagreement-coupling  methods related to those in \cite{HOS25} have recently been extended to marked Euclidean spaces and to certain
	repulsive pair-potential models in \citet{HOS26}. The strong spatial mixing theorem of \citet{MP22} also provides a route to verifying EDD for finite-range repulsive pair-potential Gibbs processes with activity below $e/\Delta_\phi$.  On the determinantal side, if $\alpha$ satisfies $-1/\alpha\in \mathbb{N}$ and the kernel $K$ is Hermitian with $0\le K\le -1/\alpha$, then $\alpha$-determinantal point processes  exist, as shown in \citet[Theorem~1.2 and the proof of Lemma~3.3]{ST03}, and can be interpreted as the superposition of $-1/\alpha$ $\iid$
	determinantal point processes with kernel $-\alpha K$. If the kernel satisfies \eqref{reEDDdoubleexp}, the resulting 
	$\alpha$-determinantal point process is still EDD, since it is a finite superposition of i.i.d. EDD point processes \cite[Lemma~3.5]{CX23}. 
\end{re}

If the score function is trivial, i.e., $\xi\equiv 1$, and the EDD point process has a  $p_1$-moment for some $p_1>3$, then $H_{\lambda}$ is the counting statistic and fulfils the exponential mixing condition \eqref{mixing-score} with
$$\gamma_0=3d+\theta_0\frac{p_1-3}{p_1},\qquad \gamma_2=\theta_2\frac{p_1-3}{p_1},\qquad \gamma_3=\theta_3,$$
and with $\gamma_1$ a finite constant depending only on the moment and EDD constants. This follows by  applying Lemmas~\ref{moments-ground} and  \ref{mixing-uncentered} together with H\"older's inequality.

Next we consider conditions on general score functions.  One possibility is to require that $\xi$ satisfies  exponential stopping set stabilization as in \citet{LSY19} and \cite{SchulteY}, but this condition is sometimes restrictive.    

Instead, for each $r \in (0, \infty)$ we construct a short-range score function $\hat{\xi}^{[r]}$  based solely on the point pattern within radius $r$, and such that $\hat{\xi}^{[r]}$
approximates $\xi$ with sufficient accuracy,  and then use this approximation  for quantifying the error in approximating $H_{\lambda}$ by the corresponding sum of short-range scores. Given a score $\xi$ and  $r \in(0,\infty)$, 
{\it a short-range estimator} 
$$ \hat{\xi}^{[r]}: (\R^d \times \mathbb{M}) \times {\cal N}_{\R^d \times  \mathbb{M}}\times {\cal B}(\real^d) \to \R$$
is defined to be  a measurable function of the marked point process that depends solely on the input within an open ball of radius $r$ around the location of interest, i.e., for almost all realizations $\hatX$ of the marked point process $\hat{\cal P}$,  for all locally finite sets $\hat{{\cal Y}}\subset (\mathbb{R}^d\backslash B(x,r))\times \M$, and for all $(x,m_x) \in \hatX$ the following holds:
\begin{equation}\label{restr}
	\hat{\xi}^{[r]} \left(\left(x,m_x\right), \left[\hatX\cap\left(B(x,r)\times \M\right)\right]\cup \hat{{\cal Y}},\Gamma_{\lambda} \right)=\hat{\xi}^{[r]}\left(\left(x,m_x\right),\hatX\cap\left(B(x,r)\times \M \right) ,\Gamma_{\lambda}\right).
\end{equation}

In other words  {\em $\hat{\xi}^{[r]}$ has a stopping set which is an open ball of radius $r$.} The restricted score 
\begin{equation} \label{restriction}
	\hat{\xi}^{[r]}\((x,m),\hatX,\Gamma_{\lambda}\):=\xi\((x,m),\hatX\cap\(B(x,r)\times\M\),\Gamma_{\lambda}\),  \ r\in (0,\infty),
\end{equation}
is a special case of the short-range estimator. We emphasize that while we could choose $\hat{\xi}^{[r]}$  to be the restriction of $\xi$ to the radius $r$ ball as in \eqref{restriction}, it is advantageous to allow for other choices of $\hat{\xi}^{[r]}$, as  when considering statistics of 	continuum percolation  models  as in \citet{LX24} and statistics of Laguerre tessellations as in  \citet{TY}. Short-range estimators $\hat{\xi}^{[r]}$   are used in  \cite{BYY25} where  $\hat{\xi}^{[r]}$ denotes the restricted score and it plays a central role in establishing optimal rates of normal approximation on the Poisson space in \cite{TY}.  The benefits are also discussed in Section~\ref{ex.srg}.

\begin{defi}\label{prdistance} (probabilistic stabilization) 
The score $\xi$ satisfies exponential probabilistic stabilization if there is a family of short-range scores $(\hat{\xi}^{[r]})_{r\in(0,\infty)}$ and positive constants $C_1$ and $C_2$ such that
	\begin{align} \label{localization}
		\varphi^{(Pr)}(r)&:= \sup_{\lambda\in(0,\infty)} \sup_{x\in\Gamma_{\lambda}}\mathbb{P}_x\left(  {\xi}((x,M_x),  \hat{{\cal P}},\Gamma_{\lambda})\neq \hat{\xi}^{[r]}((x,M_x),  \hat{{\cal P}}, \Gamma_{\lambda})\right)\nonumber\\
		& \le C_1\exp(-C_2r),    \quad r \in(0,\infty).
	\end{align}
	If  there exist positive constants $C$ and $\beta$ such that  $\varphi^{(Pr)}(r)  \le Cr^{-\beta},\  r\in (0,\infty)$, then  $\xi$ is said to satisfy polynomial probabilistic stabilization with order $\beta$. 
\end{defi}

Probabilistic stabilization of scores is a localization criterion similar to, but weaker than, the standard notion of stopping-set stabilization. It does not require the existence of a stabilizing stopping set, nor does it require comparing the restricted version in \eqref{restriction} with the original score function. Instead, it only requires the existence of an estimator based on local information that is sufficiently close to the original score function;  see  Section~\ref{section1.4}~(i). 

As soon as the scores $\xi((x,M_x), \hP, \Gamma_{\lambda})$ have enough short-range structure, in the sense that they are close to a family of short-range scores
$\hat\xi^{[r]}((x,M_x), \hP,\Gamma_{\lambda})$ as in  Definition~\ref{prdistance}, uniformly over all $x \in \Gamma_{\lambda},  r\in(0,\infty)$, then $H_{\lambda}$ is well-approximated by the normal, provided  $\hP$ is EDD.

\begin{cor}\label{maincorTV}(rates of normal convergence for sums of scores satisfying probabilistic stabilization) 
	Assume that ${\cal P}$ satisfies the $p_1$-moment condition \eqref{momground} and that the marked point process $\hat{{\cal P}}$ satisfies exponential decay of dependence   (\ref{mixing}). Assume that the scores satisfy the $p_2$-moment condition \eqref{momscore} with 
	$$\ubar{p}_0:=\frac{p_1p_2}{p_1+p_2-1} \ge 3.$$
	Lastly, assume $\Var H_{\lambda} =\Omega(\lambda^\nu)$ for some $\nu>2/3$. 
	\begin{description}
		\item{(a)} If $\xi$ satisfies exponential probabilistic stabilization  \eqref{localization} for the given family of short-range scores   $(\hat{\xi}^{[r]})_{r\in(0,\infty)}$ then
		\begin{equation*}
			d_W\( \frac{H_{\lambda} - \E H_{\lambda}} {\sqrt{\Var H_{\lambda}}}, Z\) = O\((\log \lambda)^{3d}\lambda^{-1.5 \nu+1}\).
		\end{equation*}
		\item{(b)} If $\xi$ satisfies polynomial  probabilistic stabilization for the given family of short-range scores   $(\hat{\xi}^{[r]})_{ r\in(0,\infty)}$ then
		\begin{equation*}
			d_W\( \frac{H_{\lambda} - \E H_{\lambda}} {\sqrt{\Var H_{\lambda}}}, Z\) = O\(\lambda^{-1.5\nu+1+ \tau(\g_0,\g_2)}\),
		\end{equation*}
		with $\g_2=\beta(1-2(p_1+ p_2-1)/(p_1p_2))=\beta (1-2/\ubar{p}_0)=:\beta \ubar{p}_1$, $\g_0=d(2+\ubar{p}_1)$ and, as in \eqref{taudef}, $\tau(\g_0,\g_2)=\frac{3}{2}\(\frac{2\nu d+2\g_0}{3d+\g_2}\vee\frac{\nu d+2\g_0}{3d+\g_2-\g_0}\)$. 
	\end{description}
\end{cor}

\begin{re}\label{re1.5} 
	\begin{description}
		\item{(i)} (on the rates) 
		The assertions of  parts (ii) and (iii) of Remark~\ref{re1.3} also apply to Corollary \ref{maincorTV}. If the score function has a finite dependence range, then it satisfies exponential probabilistic stabilization: indeed, taking $\hat{\xi}^{[r]}=\xi$ for all $r$ larger than the diameter of the dependence range gives $\varphi^{(Pr)}(r)=0$ for all sufficiently large $r$, and hence the result in part
		(a) applies.  Moreover, if $\varphi^{(Pr)}$ in Definition~\ref{prdistance} decays sub-exponentially, i.e., $\varphi^{(Pr)}(r)\le C_{\beta} r^{-\beta}$ for all positive $\beta$ with some positive constant $C_{\beta}$, then $d_W((H_{\lambda} - \E H_{\lambda})/\sqrt{\Var H_{\lambda}},Z)= O(\lambda^{-p})$ for all $p < 3\nu/2 - 1$.  
		\item{(ii)} (on the moments) 
		The trade-off between a $p = (3 + \delta)$-moment on the ground point process, respectively score, and a $3(p - 1)/(p- 3)$ moment on the score, respectively ground point process, as described in Remark~\ref{re1.3}~(iii), applies here as well. 
		\item{(iii)} (comparison with bounded-Lipschitz localization and the results based on it)  Another localization assumption, called {\em bounded-Lipschitz localization (BL-localization)}, is used in Sections 4-7 of \cite{BYY25} to establish a {\em qualitative} central limit theorem for $(H_\lambda)_{\lambda\in(0,\infty)}$.This condition requires for every $r \in (0,\infty)$, $q\in\mathbb{N}$, that the $q$-tuples of scores satisfy
		\begin{equation} \label{dTV}
			d_{BL} \big(\law_{\mathbf{x}}\big(\big(\xi((x_i,M_{x_i}),\hP,\Gamma_{\lambda})\big)_{i=1,...,q}\big),\law_{\mathbf{x}}\big(\big(\hat{\xi}^{[r]}((x_i,M_{x_i}), \hP,\Gamma_{\lambda} )\big)_{i=1,...,q} \big)\big) \leq \varphi_q(r),
		\end{equation}
		uniformly in $\lambda\in (0,\infty)$ and over all distinct $x_1,\ldots,x_q\in \Gamma_{\lambda}$. Here, $d_{BL}$ denotes the bounded-Lipschitz metric, for each $q\in\mathbb{N}$, $\law_{\mathbf{x}}:=\law_{x_1,\dots,x_q}$ denotes the conditional law given $x_1,\ldots,x_q\in\P$, and $\varphi_q$ is a fast decreasing function, that is, one which decays faster than any polynomial power.  When $q=1$, the assumption in \eqref{dTV} is weaker than probabilistic stabilization, since $d_{BL}$ is a weaker metric and the condition only requires closeness in distribution between the score functions and their estimators, rather than closeness in probability, which potentially offers applications to a wider class of scores. On the other hand, \eqref{dTV} is a more complex  requirement  than \eqref{localization}, as it requires closeness of {\em vectors of scores} and their localized versions with respect to higher-order Palm distributions (see Remark~\ref{re1.5}~(iv) and Remark~\ref{diff1stPalm2ndPalm}). By contrast, our localization condition only involves the case $q=1$, but with $d_{BL}$ replaced by the stronger probabilistic assumption.  We also require $\Var H_\lambda = \Omega(\lambda^{\nu})$ for some $\nu > 2/3$, whereas \cite{BYY25} only requires $\nu > 0$.
		\item{(iv)} (on the difference between conditions on higher-order and first-order Palm distributions)
		In general, when it comes to applications, one wishes to impose assumptions involving Palm distributions of the lowest possible order.  We illustrate this through a simple example involving nearest neighbor distances.  Though the  example is based on exponentially decaying functions $(\varphi_q)_{q\in\mathbb{N}}$ in \eqref{dTV}, it  can be easily adapted to the sub-exponential setting. Let $\P$ be a homogeneous Poisson point process in $\mathbb{R}^d$ with unit intensity. For each point $x\in \P$, define $$\xi(x,\P,\Gamma_{\lambda})=\xi(x,\P) =\card(\P\cap B(x,[-\ln d(x,\P\backslash\{x\})]\vee 1)).$$
		Then one can easily verify \eqref{localization} by setting 
		$$\hat{\xi}^{[r]}(x,\P,\Gamma_{\lambda})=\card(\P\cap B(x,r\wedge\{[-\ln d(x,\P\backslash\{x\})]\vee 1\})),$$
		since, for $r>1$, 
		\begin{align*}
			&\mathbb{P}_x\(\xi(x,\P,\Gamma_{\lambda})\ne \hat{\xi}^{[r]}(x,\P,\Gamma_{\lambda})\)\le  \mathbb{P}_x\(-\ln d(x,\P\backslash\{x\})>r\)\\ &=\mathbb{P}_x\(d(x,\P\backslash\{x\})<\mathrm{e}^{-r}\)=\mathbb{P}\(\card{(\P\cap B({\bf 0},\mathrm{e}^{-r}))}\ge 1\)=:\varphi(r),
		\end{align*}
		and $\varphi(r)$ decays exponentially fast.  However, for any given $r$, by taking $x_1$ and $x_2$ sufficiently close to each other, the left-hand side of \eqref{dTV} with $q=2$ can be made arbitrarily close to $1$, since the dependence range of $\law_{x_1,x_2}(\xi(x,\P,\Gamma_{\lambda}))$ becomes arbitrarily large as $d(x_1,x_2)\to 0$. This example not only illustrates the difference between assumptions based on first-order Palm distributions and those based on higher-order Palm distributions, but can also be adapted to show that each time assumptions are  strengthened from the $q$-th to  the $(q+1)$-th order Palm distributions, some potential applications may be  lost. This can be achieved by replacing $d(x,\P\backslash\{x\})$ with the distance between the $(q+1)$-th nearest neighbor of $x$ in $\P$ and the set consisting of $x$ together with its first $q$ nearest neighbors.
	\end{description}
\end{re}

In addition to probabilistic stabilization, it is useful to consider stabilization in the $L^p$ sense. 

\begin{defi}\label{lpstab} ($L^{p_3}$  stabilization)
	Given ${p_3} \in (0, \infty)$, $\xi$ satisfies exponential stabilization  in the $L^{{p_3}}$ metric, abbreviated   exponential $L^{p_3}$-stabilization, if there is a family of short-range scores   $(\hat{\xi}^{[r]})_{ r\in(0,\infty)}$ and  positive constants $C_1:=C_1({p_3})$ and $C_2:=C_2({p_3})$ such that
	\begin{align*} 
		\varphi^{ (L^{p_3}) }(r) &:= \sup_{\lambda\in(0,\infty)} \sup_{x\in\Gamma_{\lambda}} \left(\mathbb{E}_x\left| {\xi}((x,M_x),  \hP,\Gamma_{\lambda}) -  \hat{\xi}^{[r]}((x,M_x),  \hP, \Gamma_{\lambda}) \right|^{p_3}\right)^{1/p_3}\nonumber\\
		&  \le C_1\exp(-C_2r),    \quad r\in (0,\infty).\ignore{\label{localizationLp}}
	\end{align*}
	If  there exist positive constants $C:=C(p_3)$ and $\beta:=\beta(p_3)$ such that  $\varphi^{ (L^{p_3}) }(r)  \le Cr^{-\beta},  r\in (0,\infty),$ then  $\xi$ is said to  satisfy polynomial $L^{p_3}$-stabilization with order $\beta$.  
\end{defi}

Since the moments of the short-range scores cannot, in general, be controlled by \eqref{momscore}, we impose a separate moment condition on the short-range scores themselves. The family of short-range scores  $(\hat{\xi}^{[r]})_{r\in(0,\infty)}$ satisfies {\em a moment condition of order $p_2' \in [1, \infty)$} if
\begin{equation} \label{momrestrictedscore}
b_2' := b_2'(p_2'):=\sup_{\lambda\in(0,\infty)} \sup_{x\in\Gamma_{\lambda}}\sup_{r\in(0,\infty) }
\(\E_x|\hat{\xi}^{[r]}((x,M_x), \hP, \Gamma_{\lambda})|^{p_2'}\)^{1/p_2'} <\infty.
\end{equation}
Under suitable moment assumptions \eqref{momscore} and \eqref{momrestrictedscore} on score functions and their estimators, the probabilistic stabilization implies $L^p$ stabilization with different parameters, but in some applications (for example, in Section~\ref{section1.4} (i))  it is more direct to verify  probabilistic stabilization instead of $L^p$ stabilization. Also, in some other applications, the $L^p$ stabilization is useful for estimators that are not easily seen to satisfy probabilistic stabilization; see Sections~\ref{ex.srg} and \ref{localustatistic}.

\begin{cor}\label{maincorLp} (rates of normal convergence for sums of scores satisfying  $L^{p_3}$-stabilization)
	Assume that ${\cal P}$ has a $p_1$-moment as in \eqref{momground} and that the marked point process $\hat{{\cal P}}$  satisfies exponential decay of dependence  (\ref{mixing}).  Assume that the score $\xi$ has a $p_2$-moment as in \eqref{momscore} and that there is a family of short-range scores $(\hat{\xi}^{[r]})_{ r\in(0,\infty)}$ which satisfy the $p_2'$-moment condition   \eqref{momrestrictedscore} with 
	$$\ubar{p}=\frac{p_1 (p_2'\wedge p_2)}{p_1+ (p_2'\wedge p_2)-1} \ge 3.$$ 
	Lastly, assume $\Var H_{\lambda} =\Omega(\lambda^\nu)$ for some $\nu>2/3$.
	\begin{description}
		\item{(a)} If $\xi$ satisfies exponential $L^{p_3}$-stabilization  for the given family of short-range scores   $(\hat{\xi}^{[r]})_{ r\in(0,\infty)}$,  with $p_3 \ge \ubar{p}/(\ubar{p}-2)$, then 
		\begin{equation*}
			d_W\( \frac{H_{\lambda} - \E H_{\lambda}} {\sqrt{\Var H_{\lambda}}}, Z\) = O\( (\log \lambda)^{3d} \lambda^{-1.5\nu+1}\).
		\end{equation*}
		\item{(b)} If  $\xi$ satisfies polynomial  $L^{p_3}$-stabilization for the given family of short-range scores  $(\hat{\xi}^{[r]})_{r\in(0,\infty)}$, with $p_3\ge \ubar{p}/(\ubar{p}-2)$, then 
		\begin{equation*}
			d_W\( \frac{H_{\lambda} - \E H_{\lambda}} {\sqrt{\Var H_{\lambda}}}, Z\) = O\(\lambda^{-1.5\nu+1+\tau(2d,\beta)}\),
		\end{equation*}
		where $\tau(2d,\beta):=\frac{3d}{2}\(\frac{2\nu +4}{3d+\beta}\vee\frac{\nu +4}{d+\beta}\).$
	\end{description}
\end{cor}

\begin{re}\label{re1.6}
	\begin{description}
		\item{(i)} If $\card({\cal P}\cap\Gamma_1)$ has finite moments of all  orders, i.e., assumption~\eqref{momground} holds for all $p_1\in(0,\infty)$, then it suffices that $p_2\wedge p_2'>3.$ 
		\item{(ii)} For statement (b), the displayed bound converges to $0$ as $\lambda$ goes to infinity whenever  $p_3\in(0,\infty)$ (instead of $p_3\ge \ubar{p}/(\ubar{p}-2)$) provided  $\beta$ is sufficiently large (see Corollary~\ref{cor.k4} (b)).
		\item{(iii)} 
As in (ii),  if $d$ and $\nu$ are fixed and polynomial $L^{p_3}$-stabilization holds for some $p_3\in(0,\infty)$, then the  bound for  	$d_W((H_{\lambda} - \E H_{\lambda})/\sqrt{\Var H_{\lambda}},Z)$ approaches $O(\lambda^{-1.5\nu+1})$ as $\beta\to\infty$. Furthermore, if ${\varphi^{ (L^{p_3}) }(r)}$ in Definition~\ref{lpstab} decays sub-exponentially for some $p_3\in(0,\infty)$, then $d_W((H_{\lambda} - \E H_{\lambda})/\sqrt{\Var H_{\lambda}}, Z)= O(\lambda^{- \kappa})$ for all $ \kappa < 1.5\nu-1$.
		\item{(iv)}  $L^{p_3}$-stabilization  is stronger than the bounded-Lipschitz localization condition \eqref{dTV} when $q=1$. 
However, unlike \eqref{dTV}, it requires neither estimates for vectors of scores nor estimates under higher-order Palm distributions, cf. Remark~\ref{re1.5}~(iii).
	\end{description}
\end{re}

\subsection{Comparison with the  literature}\label{section1.4}

\noindent (i) (Comparison with \cite{CX23})  
As seen below,   standard exponential  stopping set stabilization  for score functions implies exponential probabilistic stabilization  and therefore the main results in \cite{CX23} are a special case  of part (a) of Theorem \ref{mainthm1}. Under exponential stopping set stabilization,  \cite{CX23} obtains a rate of convergence with a factor of $(\log \lambda)^{5d}$ instead of the factor $(\log \lambda)^{3d}$ given here. Polynomial stabilization of scores is not considered in \cite{CX23}.

\begin{defi}\label{defistab}(stopping set stabilization)  
	The score $\xi$ is exponentially stabilizing (resp. polynomially stabilizing with order $\beta$) with respect to $\law(\hat{{\cal P}})$ if for all $\lambda\in(0,\infty)$,  $x\in \Gamma_{\lambda}$,  and almost all realizations $\hatX:=\{(x,m_x)\}_{x\in {\cal X}}$ of the law $\law_x(\hat{{\cal P}})$, where $\mathcal{X}$ is the ground configuration of $\hatX$,  there exists a radius of stabilization 
$$
R:= R^\xi:=R((x,m_x),\hatX,\Gamma_{\lambda})\in(0,\infty),
$$ 
	such that for all locally finite $\hat{{\cal Y}}\subset (\mathbb{R}^d\backslash B(x,R))\times \M$, we have
	\begin{align*}
		R\left(\left(x,m_x\right), \left[\hatX \cap\left(B(x,R)\times \M\right)\right]\cup \hat{{\cal Y}},\Gamma_{\lambda} \right)&=R\left(\left(x,m_x\right), \hatX \cap\left(B(x,R)\times \M\right) ,\Gamma_{\lambda}\right),\\
		\xi\left(\left(x,m_x\right), \left[\hatX \cap\left(B(x,R)\times \M\right)\right]\cup \hat{{\cal Y}},\Gamma_{\lambda} \right)&=\xi\left(\left(x,m_x\right), \hatX \cap\left(B(x,R)\times \M\right),\Gamma_{\lambda}\right),
	\end{align*}
	and there exist constants  $C_1$ and $C_2$ such that the tail probability 
	$$\tau(t):=\sup_{\lambda\in(0,\infty)}\sup_{x\in \Gamma_{\lambda}} \mathbb{P}_{x}\left(R((x,M_x), \hat{{\cal P}},\Gamma_{\lambda})\ge t \right)$$
	satisfies $\tau(t)\le C_1\mathrm{e}^{-C_2t}$ (respectively $\tau(t)\le C_1t^{-\beta}$) for all $t\in(0,\infty)$. 
\end{defi}

Next, take  the short-range score $\hat{\xi}^{[r]}$ to be the restriction of $\xi$ to the radius $r$ ball, as in \eqref{restriction}. For this choice of  $\hat{\xi}^{[r]}$ it follows that exponential stabilization implies exponential probabilistic stabilization and polynomial stabilization implies polynomial probabilistic stabilization with the same rate, since
$$\mathbb{P}_x\left(  {\xi}((x,M_x),  \hat{{\cal P}},\Gamma_{\lambda})\neq \hat{\xi}^{[r]}((x,M_x),  \hat{{\cal P}}, \Gamma_{\lambda})\right)\le \mathbb{P}_{x}\left(R((x,M_x), \hat{{\cal P}},\Gamma_{\lambda})\ge r \right)$$
for all $r\in(0,\infty)$.
Hence,  if the score stabilizes at a given rate, then it satisfies probabilistic stabilization at the same rate. Together with Corollary \ref{maincorTV} (b), this establishes a rate of normal convergence when the score function is polynomially stabilizing. 

\noindent  (ii)  (Related quantitative CLTs)   
Given Gibbsian input whose potential satisfies a finite range condition and given a large enough inverse temperature parameter, Theorem 2.3 of  \cite{SY13} establishes $d_K((H_{\lambda} - \E H_{\lambda})/\sqrt{\Var H_{\lambda}},Z) = O\((\log \lambda)^{3d} \lambda^{-0.5}\)$ provided that the variance is of volume order, and where $d_K$ denotes the Kolmogorov distance.  Later, this bound was sharpened to $O\left((\log\lambda)^{2d}\lambda^{-0.5}\right)$ for a larger class of Gibbsian input in \cite{HOS25}. When the input $\P$  is a  Poisson point process, one can improve  the rates of convergence in our main results and establish presumably optimal rates.  Notably,  if the scores satisfy exponential probabilistic stabilization and a fifth moment condition, then the rates of normal convergence for $(H_{\lambda} - \E H_{\lambda})/\sqrt{\Var H_{\lambda}}$ can be improved to $O(\lambda^{-0.5})$ provided the variance grows like $\Omega(\lambda)$. These Berry-Esseen rates of normal approximation hold in the Wasserstein and Kolmogorov distances,  as shown in \cite{TY} and in fact  hold under {\em bounded Lipschitz localization of scores},  where the closeness of scores in the probabilistic sense in Definition~\ref{prdistance} is replaced by a closeness of scores in the  bounded Lipschitz distance.

\noindent (iii)  (Related qualitative CLT)
It is shown in \cite{BYY25} that if the  input has  fast decay of correlations, if the score has a $(2 + \delta)$-moment,  
and if the score satisfies bounded Lipschitz localization and  $\Var H_{\lambda} =\Omega(\lambda^\nu)$ for some $\nu> 0$
then one obtains a qualitative CLT for  $(H_{\lambda} - \E H_{\lambda})/\sqrt{\Var H_{\lambda}}$, but with no rate of normal convergence.

\section{Auxiliary results and the proof of the main result}

The proof of Theorem~\ref{mainthm1} is preceded by several lemmas.

\begin{lma} \label{moments-ground} (easy moment bounds for $\scrP$)
	If  $\scrP$ satisfies the $p_1$-moment condition \eqref{momground}, then for any $B\in {\scrB}(\real^d)$ that can be covered by $L_B$ unit cubes, we have
	$$\| \card(B\cap {\scrP}) \|_{p_1} \le L_Bb_1,$$
	where $b_1$ is defined in \eqref{momground}.
	In particular, if $B$ is the union of $m$ hyperrectangles with side lengths at least 1, then
	$$ \| \card(B\cap {\scrP})\|_{p_1} \le 2^dmb_1\Vol(B).$$
	Furthermore, if these $m$ hyperrectangles are disjoint, then
	$$\| \card(B\cap {\scrP})\|_{p_1} \le 2^db_1\Vol(B).$$
	Also, if  $B\in {\scrB}(\real^d)$ is a ball with diameter at least 1, then
	$$ \| \card(B\cap {\scrP})\|_{p_1}\le 4^d\Gamma(0.5d+1)\pi^{-0.5d} b_1 \Vol(B)$$
	where $\Gamma(s)$ denotes the gamma function $\int_{0}^{\infty}t^{s-1}\mathrm{e}^{-t}\mathrm{d}t$ for $s\in(0,\infty)$. 
\end{lma}

\noindent{\it Proof.} The first inequality follows from the Minkowski inequality. If $B$ is a hyperrectangle with side lengths $l_i\ge 1$, $1\le i\le d$, then it can be covered by $L_B:=\Pi_{i=1}^d\lceil l_i\rceil\le 2^d\Pi_{i=1}^dl_i=2^d\Vol(B)$ unit cubes. If $B$ is the union of $m$ such hyperrectangles, say $B_i$, $1\le i\le m$, then $L_B\le \sum_{i=1}^mL_{B_i}\le 2^d m \Vol(B)$. Furthermore, if these hyperrectangles are disjoint, then the upper bound becomes $L_B \le \sum_{i=1}^m L_{B_i} \le \sum_{i=1}^m (2^d \Vol(B_i)) = 2^d \Vol(B)$, and hence $\| \card(B\cap {\scrP})\|_{p_1}\le 2^d b_1 \Vol(B).$ If $B$ is a ball, we can cover it with a cube whose side length equals the diameter of the ball, and the statement follows from the ratio between the volumes of the cube and the ball. \qed

Recall the definition of $H_{\lambda,A}$ and  $S_{\lambda,A}$ in \eqref{HlaA} and \eqref{SlaA}, respectively, and recall that $\mu$ is the constant intensity of $\scrP$.

\begin{lma} \label{moments-XB} (moment bounds for $H_{\lambda,B}$ and $S_{\lambda,B}$) 
	If  $\scrP$ satisfies the $p_1$-moment condition \eqref{momground} and if $\xi$ satisfies the $p_2$-moment condition \eqref{momscore},  then there exists a constant $C:=C(b_1,b_2,\mu)$ independent of $\lambda$ such that for any set $B\in {\cal H}_2$,
	\begin{align}
		\|H_{\lambda,B}\|_{p_1p_2/(p_1+p_2-1)}&\le C\Vol(B),\label{momentXB-01}\\
		\|S_{\lambda,B}\|_{p_1p_2/(p_1+p_2-1)}&\le C\Vol(B).\label{momentXB-02}
	\end{align}
\end{lma}

\noindent{\it Proof.} For $s_1,~s_2\in [1,\infty)$, let $t_1,~t_2$ denote their respective conjugate indices, i.e., $1/s_1 + 1/t_1 = 1$ and $1/s_2 + 1/t_2 = 1$.  Repeatedly applying H\"{o}lder's inequality, we have 
\begin{align}
	\|H_{\lambda,B}\|_{s_1}&\le \left\{\mean \left[\left(\sum_{x\in\scrP\cap B}|\sxi|\right)^{s_1}\right]\right\}^{1/s_1}\nonumber\\
	&\le \left\{\mean \left[\sum_{x\in {\scrP} \cap B}|\sxi|^{s_1}\card(B\cap {\scrP})^{s_1/t_1}\right]\right\}^{1/s_1}\nonumber\\
	&\le\left\{\mathbb{E}\left[\left(\sum_{x\in\scrP\cap B}|\sxi|^{s_1s_2}\right)^{1/s_2}\left(\sum_{x\in\scrP\cap B}\card(B\cap {\scrP})^{s_1t_2/t_1}\right)^{1/t_2}\right]\right\}^{1/s_1}\nonumber\\
	&=\left\{\mathbb{E}\left[\left(\sum_{x\in\scrP\cap B}|\sxi|^{s_1s_2}\right)^{1/s_2}\left(\card(B\cap {\scrP})^{1+s_1t_2/t_1}\right)^{1/t_2}\right]\right\}^{1/s_1} \nonumber\\
	& \le\left\{\left[\mathbb{E}\sum_{x\in\scrP\cap B}|\sxi|^{s_1s_2}\right]^{1/s_2}\left[\mathbb{E}\(\card(B\cap{\scrP})^{1+s_1t_2/t_1}\)\right]^{1/t_2}\right\}^{1/s_1} \nonumber\\
	&= \left(\int_B\mean_x (|\sxi|^{s_1s_2})\mu \mathrm{d}x\right)^{1/(s_1s_2)}\|\card(B\cap {\scrP})\|_{1+s_1t_2/t_1}^{1-1/(s_1s_2)}.\label{momentsXB-02}
\end{align}

Setting
$$s_1=\frac{p_1p_2}{p_1+p_2-1},~s_2=\frac{p_1+p_2-1}{p_1}$$
and using the $p_2$-moment condition \eqref{momscore} and Lemma~\ref{moments-ground} in (\ref{momentsXB-02}), we obtain 
\begin{equation*}
	\|H_{\lambda,B}\|_{p_1p_2/(p_1+p_2-1)}\le  b_2(\mu\Vol(B))^{1/p_2}\|\card(B\cap {\scrP})\|_{p_1}^{1-1/p_2}\le 0.5C\Vol(B),
\end{equation*}
where $b_2$ is defined in \eqref{momscore}.
This implies \eqref{momentXB-01}. For \eqref{momentXB-02}, we apply the Minkowski inequality followed by the Lyapunov inequality, which yields
\begin{align*}
	&\|S_{\lambda,B}\|_{p_1p_2/(p_1+p_2-1)} \le \|H_{\lambda,B}\|_{p_1p_2/(p_1+p_2-1)}+\|\mathbb{E}(H_{\lambda,B})\|_{p_1p_2/(p_1+p_2-1)}\\
	&\le 2\|H_{\lambda,B}\|_{p_1p_2/(p_1+p_2-1)}\le C \Vol(B). \tag*{\qed}
\end{align*}

Using the same idea as in the proof of Lemma~\ref{moments-XB}, we establish an upper bound for the $L^p$ norm of 
$$ H_{\lambda,B}^{[r]}=\sum_{x\in {\scrP} \cap B}\left[\sxi- \hat{\xi}^{[r]}\((x,M_x),\hat{{\cal P}},\Gamma_{\lambda}\)\right], \ B\in {\cal H}_2,$$
for a given  family of short-range scores $(\hat{\xi}^{[r]})_{ r\in(0,\infty)}$. 

\begin{lma}\label{cormomentXB}(moment bounds for  $H_{\lambda,B}^{[r]}$)  
	Assume  $\scrP$ has a  $p_1$-moment as in \eqref{momground}. 
	\begin{description}
		\item{(a)} If $\xi$ satisfies exponential $L^{p_3}$-stabilization for some $p_3 \in (0, \infty)$ and if $(\hat{\xi}^{[r]})_{r\in(0,\infty)}$ is the associated family of short-range scores, then there exist constants $C_1':=C_1'(b_1, \mu)$  and $C_2':=C_2'(b_1, \mu)$ independent of $\lambda$ such that for any  $B\in {\cal H}_2$ and all $r\in(0,\infty)$
		\begin{equation}\label{cormomentXB-1}
			\left\|H_{\lambda,B}^{[r]}\right\|_{p_1p_3/(p_1+p_3-1)}\le C_1'\Vol(B)\mathrm{e}^{-C_2'r}.
		\end{equation}
		\item{(b)}  If $\xi$ satisfies polynomial  $L^{p_3}$-stabilization for some $p_3 \in (0, \infty)$ and if   $(\hat{\xi}^{[r]})_{ r\in(0,\infty)}$ is the associated family of short-range scores, then there exists a constant $C_3':=C_3'(b_1,\mu)$ independent of $\lambda$ such that for any set $B\in {\cal H}_2$ and all $r\in(0,\infty)$
		\begin{equation}\label{cormomentXB-2}
			\left\|H_{\lambda,B}^{[r]}\right\|_{p_1p_3/(p_1+p_3-1)}\le C_3'\Vol(B)r^{-\beta}.
		\end{equation}
	\end{description}
\end{lma}
 
\noindent{\it Proof.} The inequalities follow from a step-by-step repetition of the proof of Lemma~\ref{moments-XB}, replacing $H_{\lambda,\cdot}$, $\sxi$, $p_2$, and $b_2$  with $H_{\lambda,\cdot}^{[r]}$, $\sxi- \hat{\xi}^{[r]}((x, M_x), \hat{{\cal P}}, \Gamma_{\lambda})$, $p_3$, and $\varphi^{ (L^{p_3}) }(r)$, respectively. This yields the bound
\begin{equation*}
	\left\|H_{\lambda,B}^{[r]}\right\|_{p_1p_3/(p_1 + p_3 - 1)} \leq \varphi^{ (L^{p_3}) }(r) (\mu\Vol(B))^{1/p_3}  \|\card(B \cap {\cal P})\|_{p_1}^{1 - (1/p_3)},
\end{equation*}
which establishes~(\ref{cormomentXB-1}) and~(\ref{cormomentXB-2}) under the respective assumptions of exponential and polynomial $L^{p_3}$-stabilization. \qed

The proof of our main result relies on blocking arguments over big blocks and small blocks, given respectively by the rectangles  $A_2$ and  $A_1$ in the next lemma, which we will eventually take to have edge lengths proportional to $\log \lambda$ and a constant, respectively.  The next lemma is designed to handle error terms arising in these blocking arguments. 
\begin{lma}\label{mixing-score-a1}(consequence of mixing for the set-function class corresponding to $j = 2$)
	\begin{description}
		\item{(a)} Assume  the exponential mixing condition \eqref{mixing-score} holds for the set-function class corresponding to $j = 2$. For all hyperrectangles $A_1,\ A_2$ in $ {\cal H}_1$ with $A_1\subseteq A_2$, and  for all  $B \in{\cal B}(\Gamma_\lambda)$ which are finite unions of disjoint hyperrectangles in ${\cal H}_1$ with  $d(A_2,B)\ge \g_3 [\log(\diam(A_2) \vee \diam(B))\vee 1]$, we have for  $g \in  {\cal F}_2$
		\begin{equation}\label{mixing-score-a2}
			\left|\mathbb{E}\(S_{A_1}S_{A_2}g(S_B)\)-\mathbb{E}\(S_{A_1}S_{A_2}\)\mathbb{E}\(g(S_B)\)\right|\le 1.5(2d+1)U({\bm\g},A_2,B). 
		\end{equation}
		\item{(b)} If the polynomial mixing condition \eqref{mixing-score'} holds for the set-function class corresponding to $j = 2$
		then the right-hand side of \eqref{mixing-score-a2} can be replaced by $1.5(2d+1)U'({\bm\g},A_2,B)$. 
	\end{description}
\end{lma}

\noindent{\it Proof.}  We will prove (a) only, as the proof of (b) is nearly a word-for-word repetition. We write $A_{{2,0}}=A_1$, partition $A_2 \setminus A_1$ into at most $2d$ hyperrectangles  incident to $A_1$, and write them as $A_{{2,i}}$, $1\le i\le {C\le }2d$. Since $S_{A_2}=\sum_{i=0}^{{C}}S_{A_{{2,i}}}$, we can rewrite the left-hand side of \eqref{mixing-score-a2} as
\begin{align}
	&\mathbb{E}\(S_{A_1}S_{A_2}g(S_B)\)-\mathbb{E}\(S_{A_1}S_{A_2}\)\mathbb{E}\(g(S_B)\)\nonumber\\
	&=\sum_{i=0}^{{C}}\left\{\mathbb{E}\(S_{A_1}S_{A_{{2,i}}}g(S_B)\)-\mathbb{E}\(S_{A_1}S_{A_{{2,i}}}\)\mathbb{E}\(g(S_B)\)\right\}.\label{mixing-score-a3}
\end{align}
We write $S_{A_1}S_{A_{{2,i}}}$ as a sum of squares so that we may apply mixing \eqref{mixing-score} for the set-function class corresponding to $j = 2$. Writing $S_{A_1}S_{A_{{2,i}}}=0.5\{S_{A_1\cup A_{{2,i}}}^2-S_{A_1}^2-S_{A_{{2,i}}}^2\}$, $1\le i\le {C}$, we have
\begin{align}
	&\mathbb{E}(S_{A_1}S_{A_{{2,i}}}g(S_B))-\mathbb{E}(S_{A_1}S_{A_{{2,i}}})\mathbb{E}\(g(S_B)\)\nonumber\\
	&=0.5\left\{\left[\mathbb{E}\(S_{A_1\cup A_{{2,i}}}^2g(S_B)\)-\mathbb{E}(S_{A_1\cup A_{{2,i}}}^2)\mathbb{E}\(g(S_B)\)\right]\right.\nonumber\\
	&\ \ \ -\left[\mathbb{E}\(S_{A_1}^2g(S_B)\)-\mathbb{E}\(S_{A_1}^2\)\mathbb{E}\(g(S_B)\)\right]\nonumber\\
	&\ \ \ -\left.\left[\mathbb{E}\(S_{A_{{2,i}}}^2g(S_B)\)-\mathbb{E}\(S_{A_{{2,i}}}^2\)\mathbb{E}\(g(S_B)\)\right]\right\}.\label{mixing-score-a3-1}
\end{align}
Applying \eqref{mixing-score} when $ j = 2$ for $A=A_1,\ A_{{2,i}}$ and $A_1\cup A_{{2,i}}$ respectively, we have
\begin{align}
	\left|\mathbb{E}\(S_{A_1}^2g(S_B)\)-\mathbb{E}\(S_{A_1}^2\)\mathbb{E}\(g(S_B)\)\right|&\le U({\bm\g},A_1,B),\label{mixing-score-a4} \\
	\left|\mathbb{E}\(S_{A_{{2,i}}}^2g(S_B)\)-\mathbb{E}\(S_{A_{{2,i}}}^2\)\mathbb{E}\(g(S_B)\)\right|&\le U({\bm\g},A_{{2,i}},B), \label{mixing-score-a5} \\
	\left|\mathbb{E}\(S_{A_1\cup A_{{2,i}}}^2g(S_B)\)-\mathbb{E}\(S_{A_1\cup A_{{2,i}}}^2\)\mathbb{E}\(g(S_B)\)\right|&\le U({\bm\g},A_1\cup A_{{2,i}},B). \label{mixing-score-a6}
\end{align}
The diameters of $A_1$, $A_{{2,i}}$, $A_1\cup A_{{2,i}}$, $1\le i\le  C\le 2d$, are all bounded above by $\diam(A_2)$, and their distances to $B$ are all bounded below by $d(A_2,B)$.  Hence  we can combine \eqref{mixing-score-a3}-\eqref{mixing-score-a6} to obtain  \eqref{mixing-score-a2}. \qed

Mixing properties of centered sums of scores are equivalent to those of the uncentered sums,  up to a change in the exponent $\gamma_0$ by $d$.
\begin{lma}\label{mixing-uncentered}
(mixing for centered and uncentered score sums)
 If $\ubar{p}_0=p_1p_2/(p_1+p_2-1)\ge 1$, then the mixing properties \eqref{mixing-score} and \eqref{mixing-score'} in Definition~\ref{WSM} are equivalent, up to increasing the exponent $\gamma_0$ by $d$, to the corresponding properties obtained by replacing $S_A$ and $S_B$ with $H_{\lambda,A}$ and $H_{\lambda,B}$, respectively. More precisely, if the centered version holds with parameter vector $\bm\gamma=\{\gamma_i\}_{i=0}^{3}$, then the uncentered version holds with  parameter vector $\bm\gamma'=\{\gamma_i'\}_{i=0}^{3}$ with  $\gamma_0'=\gamma_0+d$ and $\gamma_i'=\gamma_i$ for $i=2,3$. Conversely, if the uncentered version holds with parameter vector $\bm\gamma'$, then the centered version holds with  parameter vector $\bm\gamma$ satisfying the same relations, i.e. $\gamma_0=\gamma_0'+d$ and $\gamma_i=\gamma_i'$ for $i=2,3$.
\end{lma}
\noindent{\it Proof.} We only prove the assertion for exponential mixing  \eqref{mixing-score};  
the proof for  polynomial mixing \eqref{mixing-score'} is identical, with
$U$ replaced by $U'$. We first show that the centered mixing property follows from the corresponding uncentered one. The converse follows by the same argument, and we omit the details.

Assume that the uncentered score sums satisfy \eqref{mixing-score} with parameter
${\bm\gamma}'=\{\gamma_i'\}_{i=0}^{3}$. For $j=1$ and $g(x)=x$, the statement holds trivially, since centering does not change the covariance.
For $j=1$ and $g\in\mathcal F_b$, define
$g_B(t):=g(t-\mathbb{E}H_{\lambda,B}),$ for all $t\in \mathbb{R}$ and $B\in \mathcal{B}\(\mathbb{R}^d\)$.
Then $g_B\in\mathcal F_b$ and $g_B(H_{\lambda,B})=g(S_B)$. Hence
\begin{align*}
	\left|\mathbb E(S_Ag(S_B))-\mathbb E(S_A)\mathbb E(g(S_B))\right|&=\left|\cov(H_{\lambda,A},g_B(H_{\lambda,B}))\right|\le U({\bm\gamma}',A,B).
\end{align*}

For $j=2$ and $g\in\mathcal F_b$, from \eqref{momentXB-01}, we have
\begin{align*}
	&\quad \left|\mathbb E(S_A^2g(S_B))-\mathbb E(S_A^2)\mathbb E(g(S_B))\right| \nonumber\\
	&\quad=\left|\cov\left((H_{\lambda,A}-\mathbb{E}H_{\lambda,A})^2,g_B(H_{\lambda,B})\right)\right| \nonumber\\
	&\quad \le\left|\cov(H_{\lambda,A}^2,g_B(H_{\lambda,B}))\right|+2|\mathbb{E}H_{\lambda,A}|\left|\cov(H_{\lambda,A},g_B(H_{\lambda,B}))\right| \nonumber\\
	&\quad \le
	U({\bm\gamma}',A,B)+C(\diam(A)^d\vee 1)U({\bm\gamma}',A,B)  \nonumber\\ 
&\quad
 \le (C+1)(\diam(A)^d\vee 1)U({\bm\gamma}',A,B),
\end{align*}
for some positive constant $C$. Thus the centered score sums satisfy \eqref{mixing-score} with the same $\gamma_2,\gamma_3$, and with $\gamma_0$ replaced by $\gamma_0'+d$. This completes the proof. \qed

Now we may prove Theorem~\ref{mainthm1}.

\noindent{\it Proof of Theorem~\ref{mainthm1}.} To facilitate the proof, we replace the Euclidean metric on $\R^d$ with the  maximum distance metric
$$ d_\infty(x,y)=\max_{1\le i\le d} |x_i-y_i|,  \quad x,y\in\R^d,
$$ 
 so that balls defined in terms of $d_\infty$ are cubes. This change of metric is localized to this proof and elsewhere we keep the Euclidean metric on $\R^d$.
For $\rho\in(0,\infty)$ and $A\in{\cal B}(\R^d)$, let $B_\infty(A,\rho)=\{y:\ d_\infty(x,y)<\rho\mbox{ for some }x\in A\}$.

\noindent (a) 
Without loss, we assume $\lambda_0\ge 1$ and $\sigma_{\lambda}\ge 2$ for all $\lambda>\lambda_0$. For $\lambda>\lambda_0$, we subdivide $\Gamma_{\lambda}$ into a maximal collection of disjoint, open, congruent cubes, each with a volume greater than $1$. The boundaries of such cubes have zero measure and do not affect the distribution of the sums, so we ignore them throughout the blocking argument.  Let ${\cal C}:=\{\mathbb{C}_1,\dots,\mathbb{C}_{k_{\lambda}}\}$ be an  enumeration  of these cubes, where $\mathbb{C}_i$ are cubes with edge length $\lambda^{1/d}/\lfloor\lambda^{1/d}\rfloor$ for all $1\le i\le k_{\lambda}$. This partition  ensures that the volume of each  cube in the partition is in $[1,2^d)$ for all admissible $\lambda$. Then $k_{\lambda}=\Omega(\lambda)$. For a given $\rho\in (0,\infty)$, let $N_{i,\lambda, \rho}=B_\infty(\mathbb{C}_i,\rho )\cap\Gamma_{\lambda}$ and $N_{i,\lambda, 2\rho}=B_\infty(\mathbb{C}_i,2\rho)\cap\Gamma_{\lambda}$. We have $N_{i,\lambda, \rho}\subset B_\infty(\mathbb{C}_i,\rho)$ and $N_{i,\lambda, 2\rho}\subset B_\infty(\mathbb{C}_i,2\rho)$, and hence  the volumes of $N_{i,\lambda,\rho}$ and $N_{i,\lambda,2\rho}$
are bounded by $O(\rho^d)$,  $\rho$  large, for all $1\le i\le k_{\lambda}$. Let 
$$ S_{i,\lambda}=S_{\mathbb{C}_i}/\sigma_{\lambda}, \quad S_{i,\lambda,\rho}=S_{N_{i,\lambda,\rho }}/\sigma_{\lambda}, \quad S_{i,\lambda,2\rho }=S_{N_{i,\lambda,2\rho }}/\sigma_{\lambda}.$$ 

We  use Stein's method for normal approximation and modify the argument in the proof of \cite[Theorem~2.1]{CX23} to establish~\eqref{mainthm1-01}. 
To this end, recall the Stein equation for normal approximation from \citet[p.~15]{CGS11}:
\begin{equation}\label{steineq1}
	f'(w)-wf(w)=h(w)-Nh,
\end{equation} 
where $Nh:=\mean h(Z)$. The solution of \eqref{steineq1} is 
$$
f_h(w)=\mathrm{e}^{w^2/2}\int^w_{-\infty}\mathrm{e}^{-t^2/2}(h(t)-Nh)\mathrm{d}t=-\mathrm{e}^{w^2/2}\int_{w}^{\infty}\mathrm{e}^{-t^2/2}(h(t)-Nh)\mathrm{d}t,
$$ and it satisfies the following properties~\cite[p.~16]{CGS11}: 
$$\|f_h\|_\infty\le 2,~\|f_h'\|_\infty\le \sqrt{\frac{2}{\pi}},~\|f_h''\|_\infty\le 2.$$ 
The Stein equation \eqref{steineq1} transforms the estimate of $d_W(Z_{\lambda},Z)$, with $Z_{\lambda}=(H_{\lambda} - \E H_{\lambda})/\sqrt{\Var H_{\lambda}}$,  into the study of $f'(Z_{\lambda})-Z_{\lambda}f(Z_{\lambda})$:
\begin{align}\label{Stein}
	d_W(Z_{\lambda},Z)&=\sup_{h\in \cF_{\rm Lip}}|\mathbb{E}(h(Z_{\lambda})-Nh)| \le\sup_{f\in \cF}\left|\mathbb{E}\(f'(Z_{\lambda})-Z_{\lambda}f(Z_{\lambda})\)\right|,
\end{align} 
where $\cF:=\left\{f: ~\mathbb{R}\rightarrow\mathbb{R}, \|f\|_\infty\le 2,~\|f'\|_\infty\le \sqrt{2/\pi},~\|f''\|_\infty\le 2\right\}$. 

The first term in \eqref{Stein} satisfies  
\begin{align}
	&\mathbb{E}f'(Z_{\lambda})\nonumber\\
	&= \mathbb{E}\(Z_{\lambda}^2\)\mathbb{E}f'\(Z_{\lambda}\)\nonumber\\
	&=\sum_{i=1}^{k_{\lambda}}\mathbb{E}\(S_{i,\lambda}S_{i,\lambda,\rho}\)\mathbb{E}f'\(Z_{\lambda}\)+ \mathbb{E}\(Z_{\lambda}^2-\sum_{i=1}^{k_{\lambda}}S_{i,\lambda}S_{i,\lambda,\rho}\)\mathbb{E}f'(Z_{\lambda})\nonumber\\
	&=\sum_{i=1}^{k_{\lambda}}\mathbb{E}\(S_{i,\lambda}S_{i,\lambda,\rho}\)\left[\mathbb{E}f'(Z_{\lambda})-\mathbb{E}f'(Z_{\lambda}-S_{i,\lambda,2\rho})+\mathbb{E}f'(Z_{\lambda}-S_{i,\lambda,2\rho})\right]+\epsilon_1\nonumber\\
	&=\sum_{i=1}^{k_{\lambda}}\mathbb{E}\(S_{i,\lambda}S_{i,\lambda,\rho}\)\mathbb{E}\int_{0}^{S_{i,\lambda,2\rho}}f''(Z_{\lambda}-x)\mathrm{d}x+ \sum_{i=1}^{k_{\lambda}}\mathbb{E}\(S_{i,\lambda}S_{i,\lambda,\rho}\)\mathbb{E}f'\(Z_{\lambda}-S_{i,\lambda,2\rho}\)+\epsilon_1 \nonumber
	\\&=\sum_{i=1}^{k_{\lambda}}\mathbb{E}\(S_{i,\lambda}S_{i,\lambda,\rho}f'(Z_{\lambda}-S_{i,\lambda,2\rho})\)+\epsilon_1+\epsilon_2+\epsilon_3,\label{Stein2}
\end{align} 
where, keeping in mind that $Z_\lambda = \sum_{i=1}^{k_{\lambda}} S_{i,\lambda}$, we have 
\begin{align*}
	\epsilon_1=& \mathbb{E}\(Z_{\lambda}^2-\sum_{i=1}^{k_{\lambda}}S_{i,\lambda}S_{i,\lambda,\rho}\)\mathbb{E}f'(Z_{\lambda})=\sum_{i=1}^{k_{\lambda}}\mathbb{E}\(S_{i,\lambda}\(Z_{\lambda}-S_{i,\lambda,\rho}\)\)\mathbb{E}f'(Z_{\lambda});\\
	\epsilon_2=& \sum_{i=1}^{k_{\lambda}}\mathbb{E}\(S_{i,\lambda}S_{i,\lambda,\rho}\)\mathbb{E}\(\int_{0}^{S_{i,\lambda,2\rho}}f''(Z_{\lambda}-x)\mathrm{d}x\);\\
	\epsilon_3=& - \sum_{i=1}^{k_{\lambda}}\mathbb{E}\left[S_{i,\lambda}S_{i,\lambda,\rho}\(f'(Z_{\lambda}-S_{i,\lambda,2\rho})-\mathbb{E}f'(Z_{\lambda}-S_{i,\lambda,2\rho})\)\right].
\end{align*} 

The terms $\epsilon_1, \epsilon_2, \epsilon_3$ may be bounded by either using exponential mixing conditions to control score sums on sets separated by $\rho$ or by using moment bounds on score sums over regions with diameter proportional to $\rho$.  The first approach produces errors which are decreasing with increasing  $\rho$, whereas the second produces errors increasing with $\rho$.  A good balance is achieved by taking  $\rho$ to be proportional to $\log \lambda$. The details go as follows.

For $\epsilon_1$,  using the bound  $\|f'\|_\infty\le \sqrt{2/\pi}$ and exponential mixing condition \eqref{mixing-score} with $j=1$, $g(x)=x,~x\in \real$, $A=\mathbb{C}_i,$ $B=\Gamma_{\lambda}\backslash N_{i,\lambda,\rho} ,$  $1\le i\le k_{\lambda}$, we have 
\begin{align}
	|\epsilon_1|&\le\sqrt{\frac{2}{\pi}} k_{\lambda}\sigma_{\lambda}^{-2}\g_1\(\(2\sqrt{d}\)^{\g_0}\vee 1\)\(\(\sqrt{d}\lambda^{1/d}\)^{\g_0}\vee 1\)\mathrm{e}^{-\g_2 \rho}\nonumber\\
	&= O\(\sigma_{\lambda}^{-2}\lambda^{\g_0/d+1}\mathrm{e}^{-\g_2 \rho}\)= O\(\lambda^{-1}\)\label{e1}
\end{align} 
for  $\rho >C_1\log \lambda$, where $C_1$ is a positive constant. 

For $\epsilon_2$, using $\|f''\|_\infty\le 2$, Lemma~\ref{moments-XB} and Lyapunov's inequality, we have
\begin{align}
	|\epsilon_2|\le& 2  \sum_{i=1}^{k_{\lambda}}\mathbb{E}\(\left|S_{i,\lambda}S_{i,\lambda,\rho}\right|\)\mathbb{E}\(\left|S_{i,\lambda,2\rho}\right|\)\nonumber \\
	\le &\sum_{i=1}^{k_{\lambda}}\mathbb{E}\(S_{i,\lambda}^2+S_{i,\lambda,\rho}^2\)\mathbb{E}\(\left|S_{i,\lambda,2\rho}\right|\)\nonumber\\
	\le& \sum_{i=1}^{k_{\lambda}}\(\|S_{i,\lambda}\|_3^2+\|S_{i,\lambda,\rho}\|_3^2\)\|S_{i,\lambda,2\rho}\|_3\nonumber\\
	=& O\(\sigma_{\lambda}^{-3}k_{\lambda}\rho^{3d}\)= O\(\lambda^{-1.5\nu+1}\rho^{3d}\).\label{e2}
\end{align}

To bound $\epsilon_3$, as $\|f'\|_\infty\le \sqrt{2/\pi}$, we apply Lemma~\ref{mixing-score-a1}~(a) with $g(x)=f'(x/\sigma_{\lambda})$ for~$x\in \real$, $A_1=\mathbb{C}_i,~A_2=N_{i,\lambda,\rho},~B=\Gamma_{\lambda}\backslash N_{i,\lambda,2\rho}$, $1\le i\le k_{\lambda}$, to obtain
\begin{align}
	|\epsilon_3|\le &  \sum_{i=1}^{k_{\lambda}}\left|\mathbb{E}\(S_{i,\lambda}S_{i,\lambda,\rho}\(f'(Z_{\lambda}-S_{i,\lambda,2\rho})-\mathbb{E}f'(Z_{\lambda}-S_{i,\lambda,2\rho})\)\)\right|\nonumber\\ 
	\le & 1.5k_{\lambda}\sigma_{\lambda}^{-2}\g_1(2d+1)\(\(\sqrt{d}\(2+2\rho\)\)^{\g_0}\vee 1\)\( \(\sqrt{d}\lambda^{1/d}\)^{\g_0}\vee 1\)\mathrm{e}^{-\g_2 \rho}\nonumber\\ 
	= &O\(\sigma_{\lambda}^{-2}\rho^{\g_0}\lambda^{\g_0/d+1}\mathrm{e}^{-\g_2 \rho}\) = O\(\lambda^{-1}\),\label{e3}
\end{align}
when $ \rho \ge C_2\log \lambda$ for some positive constant $C_2$, since  $\|g(x)\|_{\infty}\le\sqrt{2/\pi}\le1$, $\|g'(x)\|_{\infty}=\|f''(x/\sigma_{\lambda})\|_{\infty}/\sigma_{\lambda}\le 2/\sigma_{\lambda}\le 1$, and so $g\in {\cal F}_b\subset{\cal F}_1$.

For the second term in \eqref{Stein}, we have
\begin{align}
	\mathbb{E}\(Z_{\lambda}f(Z_{\lambda})\)=&\sum_{i=1}^{k_{\lambda}}\mathbb{E}\left[S_{i,\lambda}f(Z_{\lambda})\right]\nonumber
	\\=&\sum_{i=1}^{k_{\lambda}}\mathbb{E}\left[S_{i,\lambda}\(f(Z_{\lambda})-f(Z_{\lambda}-S_{i,\lambda,\rho})\)\right]\nonumber\\
	&+\sum_{i=1}^{k_{\lambda}}\mathbb{E}\left[S_{i,\lambda}\(f(Z_{\lambda}-S_{i,\lambda,\rho})-\mathbb{E}\(f(Z_{\lambda}-S_{i,\lambda,\rho})\)\)\right]\nonumber\\
	&+\sum_{i=1}^{k_{\lambda}}\mathbb{E}\(S_{i,\lambda}\)\mathbb{E}f\(Z_{\lambda}-S_{i,\lambda,\rho}\)\nonumber\\
	=&\sum_{i=1}^{k_{\lambda}}\mathbb{E}\(S_{i,\lambda}S_{i,\lambda,\rho}\int_{0}^1f'(Z_{\lambda}-uS_{i,\lambda,\rho})\mathrm{d}u\)+\epsilon_4+\epsilon_5,\label{Stein3}
\end{align} 
where 
\begin{align}
	\epsilon_4=&\sum_{i=1}^{k_{\lambda}}\left\{\mathbb{E}\left[S_{i,\lambda}f(Z_{\lambda}-S_{i,\lambda,\rho})\right]-\mathbb{E}\(S_{i,\lambda}\)\mathbb{E}\(f(Z_{\lambda}-S_{i,\lambda,\rho})\)\right\};\nonumber\\
	\epsilon_5=& \sum_{i=1}^{k_{\lambda}}\mathbb{E}\(S_{i,\lambda}\)\mathbb{E}\(f(Z_{\lambda}-S_{i,\lambda,\rho})\)=0.\label{e5}
\end{align} 

To bound $\epsilon_4$, since $\|f\|_\infty\le 2$, an application of  exponential mixing \eqref{mixing-score} with $g(x)=0.5f(x/\sigma_{\lambda})$ for $x\in \real$, $j=1$, $A=\mathbb{C}_i,$ $B=\Gamma_{\lambda}\backslash N_{i,\lambda,\rho}$, $1\le i\le k_{\lambda}$,  yields  
\begin{align}
	|\epsilon_4|\le & 2k_{\lambda}\sigma_{\lambda}^{-1}\g_1\((2\sqrt{d})^{\g_0}\vee 1\)\({\(\sqrt{d}\lambda^{1/d}\)}^{\g_0}\vee 1\)\mathrm{e}^{-\g_2 \rho}\nonumber\\
	= & O\(\sigma_{\lambda}^{-1}\lambda^{\g_0/d+1}\mathrm{e}^{-\g_2  \rho}\)= O\(\lambda^{-1}\), \label{e4}
\end{align}
for $ \rho >C_{3}\log \lambda$ with a sufficiently large $C_3$, since $\|g(x)\|_{\infty}\le 1$, $$\|g'(x)\|_{\infty}=\|0.5f'(x/\sigma_{\lambda})\|_{\infty}/\sigma_{\lambda}\le \sqrt{1/2\pi}/\sigma_{\lambda}\le1,$$ so $g\in {\cal F}_b\subset {\cal F}_1$.

The remainder of the difference between \eqref{Stein2} and \eqref{Stein3} is
 \begin{align}
	|\epsilon_6|:=&\left|\sum_{i=1}^{k_{\lambda}}\mathbb{E}\(S_{i,\lambda}S_{i,\lambda,\rho}f'(Z_{\lambda}-S_{i,\lambda,2\rho})\)-\sum_{i=1}^{k_{\lambda}}\mathbb{E}\(S_{i,\lambda}S_{i,\lambda,\rho}\int_{0}^1f'(Z_{\lambda}-uS_{i,\lambda,\rho})\mathrm{d}u\) \right|\nonumber\\ 
	= &\sum_{i=1}^{k_{\lambda}}\left|\mathbb{E}\(S_{i,\lambda}S_{i,\lambda,\rho}\int_{0}^1\int_{0}^{S_{i,\lambda,2\rho}-uS_{i,\lambda,\rho}}f''(Z_{\lambda}-S_{i,\lambda,2\rho}+v)\mathrm{d}v\mathrm{d}u\)\right|\nonumber\\
	\le
	&2\sum_{i=1}^{k_{\lambda}}\mathbb{E}\left[\left|S_{i,\lambda}S_{i,\lambda,\rho}\right|\(\left|S_{i,\lambda,2\rho}\right|+\left|S_{i,\lambda,\rho}\right|\)\right]\nonumber\\
	\le &\frac{2}{3}\sum_{i=1}^{k_{\lambda}}\mathbb{E}\(2\left|S_{i,\lambda}\right|^3+3\left|S_{i,\lambda,\rho}\right|^3+\left|S_{i,\lambda,2\rho}\right|^3\),\label{e6}
\end{align}
where the penultimate  inequality follows from the bound $\|f''\|_\infty\le 2$ and the last inequality is from the inequality of arithmetic and geometric means, that is, $a_1a_2a_3\le (a_1^3+a_2^3+a_3^3)/3$ for nonnegative $a_1,a_2,a_3$, applied to $\left|S_{i,\lambda}S_{i,\lambda,\rho}S_{i,\lambda,2\rho}\right|$ and $\left|S_{i,\lambda}\right|S_{i,\lambda,\rho}^2$. Applying Lemma~\ref{moments-XB} in \eqref{e6}, we have
\begin{equation}\label{e7}
	|\epsilon_6| = O\(k_{\lambda}\sigma_{\lambda}^{-3} \rho^{3d}\)=O\(\lambda^{-1.5\nu+1} \rho^{3d}\).
\end{equation}

Combining \eqref{Stein} - \eqref{e7}, we obtain
$$ d_W(Z_{\lambda},Z) = O(\lambda^{-1.5\nu+1} \rho^{3d})$$ 
for \qcon{} $ \rho>C_{4}\log \lambda$, where $C_{4}=\max\{C_{1},C_{2},C_{3}\}$, as claimed. 

\noindent (b) Setting $N_{i,\lambda,2 \rho}'=B_\infty(\mathbb{C}_i, (\g_3\vee 1)(1+2 \rho)) \cap \Gamma_{\lambda}$, we still have the volume of $N_{i,\lambda,2 \rho}'$ bounded above by $O( \rho^d)$. By repeating the steps of the proof for part~(a), with $N_{i,\lambda,2 \rho}$ replaced by $N_{i,\lambda,2 \rho}'$, we can see that 
\begin{equation}\label{Stein4}
	d_W(Z_{\lambda},Z) \le \sum_{i=1}^{6}\epsilon_i, 
\end{equation}
for the same $\epsilon_i$'s as in part (a). Let $\rho \ge C_5\log\lambda$ with a sufficiently large $C_5$. By adjusting the bounds in \eqref{e1}, \eqref{e3} and \eqref{e4} - replacing the bounds in exponential mixing condition \eqref{mixing-score} with those in \eqref{mixing-score'} - along with \eqref{e2}, \eqref{e5} and \eqref{e7}, we obtain
\begin{align*}
	|\epsilon_1|&= O\(\lambda^{\g_0/d+1-\nu} \rho^{-\g_2}\);\\
	|\epsilon_2|&= O\(\lambda^{-1.5\nu+1} \rho^{3d}\);\\
	|\epsilon_3|&= O\(\lambda^{\g_0/d+1-\nu} \rho^{\g_0-\g_2}\);\\
	|\epsilon_4|&= O\(\lambda^{\g_0/d+1-\nu/2} \rho^{-\g_2}\);\\
	\epsilon_5&=0;\\
	|\epsilon_6|&= O\(\lambda^{-1.5\nu+1} \rho^{3d}\).
\end{align*}
Combining the above bounds with \eqref{Stein4} and by taking 
$$ \rho=\Omega\left(\lambda^{\frac{\nu+\g_0/d}{3d+\g_2}\vee\frac{0.5\nu+\g_0/d}{3d+\g_2-\g_0}}\right), $$
we obtain the bound
\begin{equation*}
	d_W(Z_{\lambda},Z)= O\(\lambda^{-1.5\nu+1+\tau(\g_0,\g_2)}\)=o(1),
\end{equation*} 
given the assumptions, where, as in \eqref{taudef}, 
$$ 
\tau(\g_0,\g_2)=\frac{3}{2}\(\frac{2\nu d+2\g_0}{3d+\g_2}\vee\frac{\nu d+2\g_0}{3d+\g_2-\g_0}\).$$ 
This concludes the proof of Theorem~\ref{mainthm1}. \qed

\noindent{\it Proof of Corollary~\ref{maincorTV}. }
\newcounter{maincocc}
\setcounter{maincocc}{1}
\newcommand{\qmaincocc}[1]{\addtocounter{maincocc}{1}}
\newcounter{maincocca}
\newcounter{maincoccb}
\newcounter{maincoccc}
\newcounter{maincoccd}
\newcounter{maincocce}
\newcounter{maincoccf}
\newcounter{maincoccg}
\newcounter{maincocch}
\newcounter{maincocci}
\newcounter{maincoccj}
\newcounter{maincocck}
\newcounter{maincoccl}
\newcounter{maincoccm}
\newcounter{maincoccn}
\newcounter{maincocco}
We show that the conditions of Theorem~\ref{mainthm1} are fulfilled. The moment condition on the ground point process as well as the moment condition on the score are assumed;  thus it remains to verify  exponential  mixing \eqref{mixing-score}  of the score sums for part (a) and  polynomial  mixing \eqref{mixing-score'} for part (b).  We first prove part (a).

\noindent {\em Step (i).  Coupling}.  Let   $A, B\in {\cal B}(\mathbb{R}^{d} )$ be disjoint sets, not necessarily hyperrectangles, such that  $d(A,B)\ge  C_{\themaincocc}\setcounter{maincocca}{\themaincocc}\qmaincocc{} \log(\diam(A)\vee \diam(B)\vee C_{\themaincocc}\setcounter{maincoccb}{\themaincocc}\qmaincocc{})$ for some positive constants $C_{\themaincocca}$ and $C_{\themaincoccb}$. We put $r = d(A,B)/3$.  Recall that $\tilde{\hat{{\cal P}}}$ is an independent copy of $\hat{{\cal P}}$.
The total variation bound ~\eqref{mixing} characterizing EDD point processes  implies that we can construct a  {\it maximal coupling} $(\Psi_1,\Psi_2)$ and $(\Psi_3,\Psi_4)$ in terms of the total variation distance between 
$$ (\hat{{\cal P}}|_{B_r(A)}, \hat{{\cal P}}|_{B_r(B)})\mbox{ and  }(\hat{{\cal P}}|_{B_r(A)}, \tilde{\hat{{\cal P}}}|_{B_r(B)}).$$ 
More precisely, the marked point processes $(\Psi_1,\Psi_2)$ and $(\Psi_3,\Psi_4)$ are random elements defined on the same probability space, taking values in ${\cal N}_{B_r(A) \times  \mathbb{M}}\times {\cal N}_{B_r(B) \times  \mathbb{M}}$, equipped with the $\sigma$-algebra generated by the vague topology, such that 
\begin{align*}
(\Psi_1,\Psi_2)&\overset{d}{=}(\hat{{\cal P}}|_{B_r(A)}, \hat{{\cal P}}|_{B_r(B)}), \\
(\Psi_3,\Psi_4)&\overset{d}{=}(\hat{{\cal P}}|_{B_r(A)}, \tilde{\hat{{\cal P}}}|_{B_r(B)}),
\end{align*} 
and by ~\eqref{mixing} we have
\begin{equation}\label{maincor1.1}
	\mathbb{P}(E_1):=\mathbb{P}\((\Psi_1,\Psi_2)\neq 
	(\Psi_3,\Psi_4)\)\le C_{\themaincocc}\setcounter{maincoccc}{\themaincocc}\qmaincocc{}(\diam(A)^{C_{\themaincocc}\setcounter{maincoccd}{\themaincocc}\qmaincocc{}}\vee 1)(\diam(B)^{C_{\themaincoccd}}\vee 1) \mathrm{e}^{-C_{\themaincocc}\setcounter{maincocce}{\themaincocc}\qmaincocc{} d(A,B)} 
\end{equation}
for some positive constants $C_{\themaincoccc}$, $C_{\themaincocce}$, and a non-negative constant $C_{\themaincoccd}$. From the coupling $(\Psi_3,\Psi_4)$, we can construct two independent marked point processes $\Psi_5$ and $\Psi_6$, each with the same distribution as $\hat{{\cal P}}$ such that $\Psi_5|_{B_r(A)}=\Psi_3$, $\Psi_6|_{B_r(B)}=\Psi_4$. Let $\bar{\Psi}_5$ and $\bar{\Psi}_6$ denote their respective ground processes. Define
\begin{align*}
	H_{\lambda,A}'&:=\sum_{x\in \bar{\Psi}_5\cap A}\xi((x,M_{5,x}),\Psi_5,\Gamma_{\lambda}),\ S_A':=H_{\lambda,A}'-\mathbb{E}(H_{\lambda,A}'),\\
	H_{\lambda,B}'&:=\sum_{x\in \bar{\Psi}_6\cap B}\xi((x,M_{6,x}),\Psi_6,\Gamma_{\lambda}),\ S_B':=H_{\lambda,B}'-\mathbb{E}(H_{\lambda,B}').
\end{align*}

\noindent {\em Step (ii).
 Approximating score sums by sums of short-range scores}. Recall
$r = d(A,B)/3$. We use this value of $r$ to construct short-range versions of $\bar{S}_A'$ and $\bar{S}_B'$ given by
\begin{align}
	\bar{S}_A^{[r]} &:=\sum_{x\in \bar{\Psi}_5\cap A}\hat{\xi}^{[r]}((x,M_{5,x}),\Psi_5,\Gamma_{\lambda})-\mathbb{E}(H_{\lambda,A}'),\label{maincor2a1}\\
	\bar{S}_B^{[r]}&:=\sum_{x\in \bar{\Psi}_6\cap B}\hat{\xi}^{[r]}((x,M_{6,x}),\Psi_6,\Gamma_{\lambda})-\mathbb{E}(H_{\lambda,B}')\label{maincor2a2}.
\end{align}
By construction, we have 
\begin{equation*}
S_A'\overset{d}{=}S_A, \ S_B'\overset{d}{=}S_B, \ 
\bar{S}_A^{[r]} \in \sigma(\Psi_5|_{B_r(A)})=\sigma(\Psi_3), 
\bar{S}_B^{[r]} \in \sigma(\Psi_6|_{B_r(B)})=\sigma(\Psi_4),
\end{equation*} 
and 
$$ S_A'\indep S_B',\ \bar{S}_A^{[r]}\indep \bar{S}_B^{[r]},$$
where $\indep $ denotes independence.

Let $E_2 := \{S_A' \neq \bar{S}_A^{[r]}\}$ and $E_3 := \{S_B' \neq \bar{S}_B^{[r]}\}$. We compute 
\begin{align}
	\mathbb{P}(E_2) & \le \mathbb{E}\int_{x\in A}\bone_{[\xi((x,M_{5,x}),\Psi_5,\Gamma_{\lambda})\ne \hat{\xi}^{[r]}((x,M_{5,x}),\Psi_5),\Gamma_{\lambda})]}\bar{\Psi}_5(\mathrm{d}x)  \nonumber \\
	& =\int_{x\in A}\mathbb{P}_x(\xi((x,M_{5,x}),\Psi_5,\Gamma_{\lambda})\ne \hat{\xi}^{[r]}((x,M_{5,x}),\Psi_5,\Gamma_{\lambda}))\mu \mathrm{d}x\nonumber
	 \nonumber \\  
	& \le  \mu \varphi^{(Pr)}(r) \Vol(A),  \label{maincor1.2}
\end{align}
and
\begin{align}
	\mathbb{P}(E_3) 
	&  \le \mathbb{E}\int_{x\in B}\bone_{[\xi((x,M_{6,x}),\Psi_6,\Gamma_{\lambda})\ne \hat{\xi}^{[r]}((x,M_{6,x}),\Psi_6,\Gamma_{\lambda})]}\bar{\Psi}_6(\mathrm{d}x) \nonumber\\
	&=\int_{x\in B}\mathbb{P}_x(\xi((x,M_{6,x}),\Psi_6,\Gamma_{\lambda})\ne \hat{\xi}^{[r]}((x,M_{6,x}),\Psi_6,\Gamma_{\lambda}))\mu \mathrm{d}x\nonumber \\
	&\le  \mu\varphi^{(Pr)}(r) \Vol(B).  \label{maincor1.3}
\end{align}

\noindent {\em Step (iii). Using short-range score sums to show $S_A = S_A'$ on a high probability event. }
Let $E=(E_1\cup E_2\cup E_3)^c$. On the coupling space, $S_A$ and $S_B$ are realized on $\hP$, whose restrictions to $B_r(A)$ and $B_r(B)$ are $\Psi_1$ and $\Psi_2$, respectively. Since the localized estimators depend only on these restrictions, on $E_1^c$, the localized sums from the two coupled copies coincide. Combining~\eqref{localization}, (\ref{maincor1.1}), (\ref{maincor1.2}) and~(\ref{maincor1.3}), we find on the event $E$,
$$
 S_A=S_A'=\bar{S}_A^{[r]}, \ S_B=S_B'=\bar{S}_B^{[r]}.
$$ 
Recalling that $\varphi^{(Pr)}(\cdot)$ decays exponentially fast and recalling $r = d(A,B)/3$, 
 the probability of the complement satisfies 
$$
 \mathbb{P}(E^c)\le C_{\themaincocc}\setcounter{maincocch}{\themaincocc}\qmaincocc{}(\diam(A)^{C_{\themaincocc}\setcounter{maincocci}{\themaincocc}\qmaincocc{}}\vee 1)(\diam(B)^{C_{\themaincocci}}\vee 1)\mathrm{e}^{-C_{\themaincocc}\setcounter{maincoccj}{\themaincocc}\qmaincocc{} d(A,B)}
$$
for some positive constants $C_{\themaincocch}$, $C_{\themaincocci}$ and $C_{\themaincoccj}$. Here, $\bar{S}_A^{[r]}$ and $\bar{S}_B^{[r]}$ are only needed to construct a coupling between $(S_A,S_B)$ and $(S_A',S_B')$. 

\noindent {\em Step (iv).  Conclusion.}  Applying H\"{o}lder's inequality and Lemma~\ref{moments-XB}, we obtain for
all  $j \in \{1,2\}$ and $g\in {\cal F}_j$, $A\in  {\cal H}_j$, and all admissible $B\subset \Gamma_\lambda$ satisfying the separation condition in Definition~\ref{WSM},
\begin{align}
	&\left|\mathbb{E}(S_A^jg(S_B))-\mathbb{E}(S_A^j)\mathbb{E}\(g(S_B)\)\right|\nonumber\\
	&=\left|\mathbb{E}\((S_A')^jg(S_B')\)-\mathbb{E}\((S_A')^jg(S_B')\mathbf{1}_{E^c}\)+\mathbb{E}\(S_A^jg(S_B)\mathbf{1}_{E^c}\)-\mathbb{E}\((S_A')^j\)\mathbb{E}\(g(S_B')\)\right|\nonumber\\
	&\le\left|\mathbb{E}\(S_A^jg(S_B)\mathbf{1}_{E^c}\)\right|+\left|\mathbb{E}\((S_A')^jg(S_B')\mathbf{1}_{E^c}\)\right|\label{maincor1.4.1}\\
	&\le U({\bm\gamma},A,B)\label{maincor1.4.2}
\end{align}
for some $\bm\gamma$, which confirms that exponential mixing \eqref{mixing-score} holds. Hence, the result follows by direct application of Theorem~\ref{mainthm1} (a).

\noindent (b) This is a step-by-step repetition of the proof of (a), with the primary modification being the bounds in \eqref{maincor1.2} and \eqref{maincor1.3}, which, using $\varphi^{(Pr)}(r)\le Cr^{-\beta}$, now take the form
\begin{equation}\label{for-b-1}
	\mathbb{P}(E_2)\le C_{\themaincocc}\setcounter{maincocck}{\themaincocc}\qmaincocc{} \Vol(A)d(A,B)^{-\beta}, \  \ \mathbb{P}(E_3)\le  C_{\themaincocck}\Vol(B)d(A,B)^{-\beta}
\end{equation}
for some constant $C_{\themaincocck}\in (0,\infty)$. Combining (\ref{maincor1.1}) and (\ref{for-b-1}), and choosing a suitable constant $\g_3$ for \eqref{mixing-score'}, we obtain
$$\mathbb{P}(E^c)\le C_{\themaincocc}\setcounter{maincoccl}{\themaincocc}\qmaincocc{} (\Vol(A)\vee 1)(\Vol(B)\vee 1)d(A,B)^{-\beta}$$
for some constant $C_{\themaincoccl}\in (0,\infty)$.
 Recall that
$$\ubar{p}_1=1-\frac{2}{\ubar{p}_0}=1-\frac{2(p_1+ p_2-1)}{p_1 p_2}.$$ 
Then, applying (\ref{maincor1.4.1}), H\"{o}lder's inequality, and Lemma~\ref{moments-XB}, we obtain the following bounds.

When $j=1$ and $g \in \cF_1$, i.e.,
$g(x)=x$ or $g\in \cF_b$, we have 
\begin{align*}
	&\left|\mathbb{E}\(S_Ag(S_B)\)-\mathbb{E}\(S_A\)\mathbb{E}\(g(S_B)\)\right|\nonumber\\
	&\le\left|\mathbb{E}\(S_Ag(S_B)\mathbf{1}_{E^c}\)\right|+\left|\mathbb{E}\(S_A'g(S_B')\mathbf{1}_{E^c}\)\right|\nonumber\\
	&\le2\|S_A\|_{p_1p_2/(p_1+p_2-1)}\|g(S_B)\|_{p_1p_2/(p_1+p_2-1)}\mathbb{P}(E^c)^{\ubar{p}_1}\nonumber\\
	&\le  C_{\themaincocc}\setcounter{maincoccm}{\themaincocc}\qmaincocc{} (\diam(A)^{d(1+\ubar{p}_1)}\vee 1)(\diam(B)^{d(1+\ubar{p}_1)}\vee 1)d(A,B)^{-\beta \ubar{p}_1}\ignore{\label{maincor1.7}}
\end{align*}
for some constant $C_{\themaincoccm}\in(0,\infty)$.

When $j=2$ and $g\in \cF_b$, we similarly obtain 
\begin{align*}
	&\left|\mathbb{E}\(S_A^2g(S_B)\)-\mathbb{E}\(S_A^2\)\mathbb{E}\(g(S_B)\)\right|\nonumber\\
	&\le\left|\mathbb{E}\(S_A^2\mathbf{1}_{E^c}\)\right|+\left|\mathbb{E}\((S_A')^2\mathbf{1}_{E^c}\)\right|\nonumber\\
	&\le 2\|S_A\|_{p_1p_2/(p_1+p_2-1)}^2\mathbb{P}(E^c)^{\ubar{p}_1}\nonumber\\
	&\le  C_{\themaincocc}\setcounter{maincoccn}{\themaincocc}\qmaincocc{} (\diam(A)^{d(2+\ubar{p}_1)}\vee 1)(\diam(B)^{d(2+\ubar{p}_1)}\vee 1)d(A,B)^{-\beta \ubar{p}_1}\ignore{\label{maincor1.8}}
\end{align*}
for some constant $C_{\themaincoccn}\in(0,\infty)$. Thus, polynomial mixing holds with  parameters $\g_0=d(2+\ubar{p}_1),\ \g_2=\beta \ubar{p}_1$, and the proof is complete by applying Theorem~\ref{mainthm1}~(b). \qed

Before going into the proof of Corollary~\ref{maincorLp}, we state a simple lemma.

\begin{lma}\label{lma.k4} 
	If $X$ is a random variable such that $\|X\|_{m_2}$ exists for some positive constant $m_2$, then for all $m_1 < m_3 < m_2$,
	$$\|X\|_{m_3}\le \|X\|_{m_1}^{m_1(m_2-m_3)/(m_3(m_2-m_1))}\|X\|_{m_2}^{m_2(m_3-m_1)/(m_3(m_2-m_1))}.$$
\end{lma}

\noindent{\it Proof.} For any $m_3\in (m_1,m_2)$, H\"{o}lder's inequality implies 
\begin{align*}
	\mathbb{E}\(|X|^{m_3}\)=&\mathbb{E}\(|X|^{m_2(m_3-m_1)/(m_2-m_1)}|X|^{m_1(m_2-m_3)/(m_2-m_1)}\)\\
	\le&\left\||X|^{m_2(m_3-m_1)/(m_2-m_1)}\right\|_{(m_2-m_1)/(m_3-m_1)}\left\||X|^{m_1(m_2-m_3)/(m_2-m_1)}\right\|_{(m_2-m_1)/(m_2-m_3)}\\
	=&\mathbb{E}\(|X|^{m_1}\)^{(m_2-m_3)/(m_2-m_1)}\mathbb{E}\(|X|^{m_2}\)^{(m_3-m_1)/(m_2-m_1)}.
\end{align*}
The statement holds by taking $m_3$-th roots on both sides. \qed

\begin{cor}\label{cor.k4} Assume that $\xi$ 
satisfies the $p_2$-moment condition \eqref{momscore}  
and $\hat{\xi}^{[r]}$ satisfies the $p_2'$-moment condition 
\eqref{momrestrictedscore}. 
\begin{description}
	\item{(a)} If exponential $L^p$ stabilization holds for $p = p_3$,
	then  it also holds for all $p \le p_3 \vee (p_2\wedge p_2')$.
	\item{(b)} If polynomial  $L^p$ stabilization holds for $p = p_3$ with $\beta=C_2$ for some positive constant  $C_2$, then it also holds for any $p \in (p_3, p_2\wedge p_2')$ with $\beta=C_2p_3[(p_2\wedge p_2')-p]/\{p[(p_2\wedge p_2')- p_3]\}$.
\end{description}
\end{cor}
\noindent{\it Proof.} (a) The statement holds trivially for all $p < p_3$ by Lyapunov's inequality. For any $p \in (p_3, p_2\wedge p_2')$, the statement holds by applying Lemma~\ref{lma.k4} with $X\sim {\cal L}_x\left(\xi((x,M_x),\hP,\Gamma_{\lambda})-\hat{\xi}^{[r]}((x,M_x),\hP,\Gamma_{\lambda})\right)$, $m_1=p_3$ and $m_2=p_2\wedge p_2'$ once we note that $\|X\|_{m_1}$ is uniformly bounded by  Minkowski's inequality.\\
(b) Statement (b) holds by taking $X\sim {\cal L}_x\left(\xi((x,M_x),\hP,\Gamma_{\lambda})-\hat{\xi}^{[r]}((x,M_x),\hP,\Gamma_{\lambda})\right)$, $m_1=p_3$ and $m_2=p_2\wedge p_2'$ in Lemma~\ref{lma.k4}. \qed

\noindent{\it Proof of Corollary~\ref{maincorLp}.} The moment conditions on the ground point process and on the score are already given;  it remains to verify  exponential mixing \eqref{mixing-score} for part (a) and  polynomial  mixing \eqref{mixing-score'} for part (b) and to apply Theorem~\ref{mainthm1}.  Recall the definitions of the sums of short-range scores in \eqref{maincor2a1} and \eqref{maincor2a2} 
where we have taken $r = d(A,B)/3$.  The independence condition  $\bar{S}_A^{[r]}\indep \bar{S}_B^{[r]}$ implies for any $j \in \{1,2\}$ and  $g\in {\cal F}_j$, $A\in  {\cal H}_j$, 
$$\mathbb{E}\(\(\bar{S}_A^{[r]}\)^jg\(\bar{S}_B^{[r]}\)\)-\mathbb{E}\(\(\bar{S}_A^{[r]}\)^j\)\mathbb{E}\(g\(\bar{S}_B^{[r]}\)\)=0.$$
This yields the following bound for any $j \in \{1,2\}$ and   $g\in {\cal F}_j$, $A\in  {\cal H}_j$,
\begin{align}
	\epsilon(r):= & \left|\mathbb{E}\(S_A^jg(S_B)\)-\mathbb{E}\(S_A^j\)\mathbb{E}\(g(S_B)\)\right|\nonumber\\
	=&\left|\left[\mathbb{E}\(S_A^jg(S_B)\)-\mathbb{E}\(S_A^j\)\mathbb{E}\(g(S_B)\)\right]-\left[\mathbb{E}\(\(\bar{S}_A^{[r]}\)^jg\(\bar{S}_B^{[r]}\)\)-\mathbb{E}\(\(\bar{S}_A^{[r]}\)^j\)\mathbb{E}\(g\(\bar{S}_B^{[r]}\)\)\right]\right| \nonumber\\
	\le& \left|\mathbb{E}\(\[S_A^j-\(\bar{S}_A^{[r]}\)^j\]g(S_B)\)\right|+\left|\mathbb{E}\(\(\bar{S}_A^{[r]}\)^j\[g(S_B)-g\(\bar{S}_B^{[r]}\)\)\]\right|\nonumber\\
	&+\left|\mathbb{E}\(S_A^j-\(\bar{S}_A^{[r]}\)^j\)\mathbb{E}(g(S_B))\right|+\left|\mathbb{E}\(\(\bar{S}_A^{[r]}\)^j\)\mathbb{E}\(g(S_B)-g\(\bar{S}_B^{[r]}\)\)\right|\nonumber\\
	=&:\epsilon_1+\epsilon_2+\epsilon_3+\epsilon_4.\label{maincor3.1}
\end{align} 
By Lemma~\ref{moments-XB}, with $\xi$, $p_2$, and $b_2$ replaced by $\hat{\xi}^{[r]}$, $p_2'$, and $b_2'$, respectively, we obtain
\begin{equation}\label{maincor3.2}
	\|\bar{S}_B^{[r]}\|_3\le \|\bar{S}_B^{[r]}\|_{p_1p_2'/(p_1+p_2'-1)}\le C_1 \Vol(B)
\end{equation}
for some constant $C_1\in (0,\infty)$ and all $r \in(0,\infty)$. To guarantee that both $p_1p_2'/(p_1+p_2'-1)\ge 3$ and $p_1p_2/(p_1+p_2-1)\ge 3$, it suffices to require $\ubar{p}=p_1(p_2\wedge p_2')/(p_1+p_2\wedge p_2'-1)\ge 3.$

Moreover, for any random variable $X$, any conjugate pair $p,q\ge 1$ satisfying $1/p+1/q=1$, and for all functions $g \in {\cal F}_1$, namely for all $g(x)=x$ or $g(x)\in \cF_b$, H\"{o}lder's inequality yields:
$$ \left|\mathbb{E}\left\{X \left[g(S_B)-g\(\bar{S}_B^{[r]}\)\right]\right\}\right|\le \|X\|_p\left\|g(S_B)-g\(\bar{S}_B^{[r]}\)\right\|_q\le \|X\|_p\|g'\|_\infty\|S_B-\bar{S}_B^{[r]}\|_q.$$
In particular, we obtain the following useful inequalities:
\begin{align*}
	&\left|\mathbb{E}\left\{X \left[g(S_B)-g\(\bar{S}_B^{[r]}\)\right]\right\}\right|\le \|X\|_p\|H_{\lambda,B}^{[r]}\|_q,\\
	&\left\|g(X)\right\|_p\le \|X\|_p\vee 1.
\end{align*}
These inequalities will be used repeatedly in the proofs.

(a) The assumption $\ubar{p}\ge 3$ implies $ p_2\wedge p_2'>3\ge \ubar{p}/(\ubar{p}-2)>\ubar{p}/(\ubar{p}-1)$. Then from Corollary~\ref{cor.k4} (a), 
the exponential $L^{p}$ stabilization 
also holds for $p$ taking $\ubar{p}/(\ubar{p}-1)$ and $\ubar{p}/(\ubar{p}-2)$.  When $j=1$ and either $g(x)=x$ or $g(x)\in \cF_b$, then by \eqref{maincor3.2}, Lemma~\ref{moments-XB}, Lemma~\ref{cormomentXB}, and H\"{o}lder's inequality,  the four terms on the right-hand side of \eqref{maincor3.1} satisfy  the bounds
\begin{equation}\label{maincor3.class.1}
	\begin{aligned}
		\epsilon_1=&\left|\mathbb{E}\(H_{\lambda,A}^{[r]} g(S_B)\)\right|\le\|H_{\lambda,A}^{[r]}\|_2\|g(S_B)\|_2= O\((\Vol(A)\vee 1)(\Vol(B)\vee 1)\mathrm{e}^{-C_2' r}\) ,    \\
		\epsilon_2 \le& \|\bar{S}_A^{[r]}\|_2\|H_{\lambda,B}^{[r]}\|_2= O\((\Vol(A)\vee 1)(\Vol(B)\vee 1)\mathrm{e}^{-C_2' r}\),\\
		\epsilon_3\le &\|H_{\lambda,A}^{[r]}\|_1\|g(S_{B})\|_1= O\((\Vol(A)\vee 1)(\Vol(B)\vee 1)\mathrm{e}^{-C_2' r} \),\\
		\epsilon_4\le&\|\bar{S}_A^{[r]}\|_1\|H_{\lambda,B}^{[r]}\|_1= O\((\Vol(A)\vee 1)(\Vol(B)\vee 1)\mathrm{e}^{-C_2' r} \).
	\end{aligned} 
\end{equation}
Therefore, each $\epsilon_i, 1 \leq i\leq 4,$  is bounded above by $O\left((\Vol(A)\vee 1)(\Vol(B)\vee 1)\mathrm{e}^{-C_2' r}\right)$ for some constant $C_2'\in(0,\infty)$. 

Similarly, when $j=2$ and $g(x)\in \cF_b$, the four terms on the right-hand side of \eqref{maincor3.1} satisfy the following bounds:
\begin{align}
	\epsilon_1=&\left|\mathbb{E}\(\(S_A-\bar{S}_A^{[r]}\)\(S_A+\bar{S}_A^{[r]}\)g(S_B)\)\right|\le\|H_{\lambda,A}^{[r]}\|_{\ubar{p}/(\ubar{p}-1)}(\|S_A\|_{\ubar{p}} +\|\bar{S}_A^{[r]}\|_{\ubar{p}} )\nonumber \\
	=& O\((\Vol(A)^2\vee 1)(\Vol(B)\vee 1)\mathrm{e}^{-C_2' r} \),\nonumber\\
	\epsilon_2\le&\left|\mathbb{E}\(\(\bar{S}_A^{[r]}\)^2\left|g(S_B)-g\(\bar{S}_B^{[r]}\)\right|\)\right|\le\left\|\(\bar{S}_A^{[r]}\)^2\right\|_{\ubar{p}/2}\|H_{\lambda,B}^{[r]}\|_{\ubar{p}/(\ubar{p}-2)}=\|\bar{S}_A^{[r]}\|_{\ubar{p}}^2\|H_{\lambda,B}^{[r]}\|_{\ubar{p}/(\ubar{p}-2)}\nonumber\\
	=& O\(\Vol(A)^2\vee 1)(\Vol(B) \vee 1)\mathrm{e}^{-C_2'r}\),\nonumber\\
	\epsilon_3\le &\left|\mathbb{E}\(H_{\lambda,A}^{[r]}\(S_A+\bar{S}_A^{[r]}\)\)\right|\le \|H_{\lambda,A}^{[r]}\|_2\(\|S_A\|_2+\|\bar{S}_A^{[r]}\|_2\) \nonumber\\
	= &  O\(\Vol(A)^2\vee 1)(\Vol(B)\vee 1)\mathrm{e}^{-C_2'r}\),      \nonumber\\
	\epsilon_4\le&\left|\mathbb{E}\(\(\bar{S}_A^{[r]}\)^2\)\right|\|H_{\lambda,B}^{[r]}\|_1\le \|\bar{S}_A^{[r]}\|_2^2\|H_{\lambda,B}^{[r]}\|_1\nonumber\\
	=  &  O\( \Vol(A)^2 \vee 1)(\Vol(B) \vee 1))\mathrm{e}^{-C_2'r} \).\nonumber
\end{align} 

Hence, 
\begin{equation}\label{maincor3.class.2}
	\max_{i \leq 4}  \epsilon_i =   O\((\Vol(A)^2 \vee 1)(\Vol(B) \vee 1) \mathrm{e}^{-C_2' r} \)
\end{equation} 
for some constant $C_2'\in(0,\infty)$. Combining the bounds in \eqref{maincor3.class.1} and \eqref{maincor3.class.2} with \eqref{maincor3.1}, we conclude that
\begin{equation*}\ignore{\label{maincor3.4}}
	\epsilon(r) = O\left( (\Vol(A)^2 \vee 1)(\Vol(B) \vee 1) \mathrm{e}^{-C_2' r} \right).
\end{equation*}
This verifies the exponential mixing condition \eqref{mixing-score} with parameters $\g_0=2d$,  $\g_2= C_2'/3$ and any $\g_3>0$.  The conclusion follows  by Theorem~\ref{mainthm1}~(a).

(b) As in the proof of part (a),  we obtain the following similar bounds. For $j=1$ and either $g(x)=x$ or $g(x)\in \cF_b$,  the four terms on the right-hand side of \eqref{maincor3.1} satisfy
\begin{equation}\label{maincor3.class.3}
	\begin{aligned}
		\epsilon_1\le&\|H_{\lambda,A}^{[r]}\|_2\|g(S_B)\|_2= O\((\Vol(A)\vee 1)(\Vol(B)\vee 1)r^{-\beta}\),\\
		\epsilon_2\le& \|\bar{S}_A^{[r]}\|_2\|H_{\lambda,B}^{[r]}\|_2= O\((\Vol(A)\vee 1)(\Vol(B)\vee 1)r^{-\beta}\),\\
		\epsilon_3\le &\|H_{\lambda,A}^{[r]}\|_1\|g(S_{B})\|_1= O\((\Vol(A)\vee 1)(\Vol(B)\vee 1)r^{-\beta}\),\\
		\epsilon_4\le&\|\bar{S}_A^{[r]}\|_1\|H_{\lambda,B}^{[r]}\|_1= O\((\Vol(A)\vee 1)(\Vol(B)\vee 1)r^{-\beta}\).
	\end{aligned} 
\end{equation}
For $j=2$ and $g(x)\in \cF_b$, we have
\begin{equation}\label{maincor3.class.4}
	\begin{aligned}
		\epsilon_1\le&\|H_{\lambda,A}^{[r]}\|_{\ubar{p}/(\ubar{p}-1)}(\|S_A\|_{\ubar{p}} +\|\bar{S}_A^{[r]}\|_{\ubar{p}} )= O\((\Vol(A)^2\vee 1)(\Vol(B)\vee 1)r^{-\beta}\),\\
		\epsilon_2\le& \left\|\bar{S}_A^{[r]}\right\|_{\ubar{p}}^2\|H_{\lambda,B}^{[r]}\|_{\ubar{p}/(\ubar{p}-2)}
		= O\((\Vol(A)^2\vee 1)(\Vol(B)\vee 1)r^{-\beta}\),\\
		\epsilon_3\le& \|H_{\lambda,A}^{[r]}\|_2\(\|S_A\|_2+\|\bar{S}_A^{[r]}\|_2\)= O\((\Vol(A)^2\vee 1)(\Vol(B)\vee 1)r^{-\beta}\),\\
		\epsilon_4\le& \|\bar{S}_A^{[r]}\|_2^2\|H_{\lambda,B}^{[r]}\|_1= O\((\Vol(A)^2\vee 1)(\Vol(B)\vee 1)r^{-\beta}\).
	\end{aligned} 
\end{equation}
Combining the bounds in \eqref{maincor3.class.3} and \eqref{maincor3.class.4} with \eqref{maincor3.1}, we obtain
\begin{equation*}\ignore{\label{maincor3.5}}
	\epsilon(r) = O\((\Vol(A)^2\vee 1)(\Vol(B)\vee 1)r^{-\beta}\).
\end{equation*}
This implies polynomial mixing of scores  with parameters $\g_0=2d$, $\g_2=\beta$ and any $\g_3>0$,
and the conclusion follows by Theorem~\ref{mainthm1}~(b). \qed

\section{Applications}
This section provides examples illustrating the applicability of Theorem~\ref{mainthm1}, Corollary~\ref{maincorTV} and Corollary~\ref{maincorLp}. The order of presentation parallels that of the main results in Section~1.

\subsection{Marked determinantal point processes}\label{ex.dpp}

We show that determinantal point processes having an exponentially decaying kernel
and equipped with i.i.d. marks $\{M_x\}_{x \in {\cal P}}$ give rise to statistics $H_{\lambda} := \sum_{x \in {\cal P} \cap \Gamma_{\lambda}}M_x$ which satisfy the conditions of Theorem~\ref{mainthm1}.  In the degenerate case $M_x\equiv 1$, the statistic $H_{\lambda}$ reduces to the counting statistic. We deduce  that these counting statistics satisfy a quantitative CLT with a nearly optimal rate.   Local (and weaker) qualitative central limit theorems for the point counts of unmarked determinantal point processes  were  established in  \citet{S02} and \citet{FL14}. However, for the counting statistics and localized marked sums considered here under merely exponential decay of the kernel, comparable nearly optimal Wasserstein bounds do not appear to be available. Closely  related results include the Wasserstein bound of \citet{W19},
which has a rate slower than $\lambda^{-1/4}$; the quantitative CLT of \cite{CX23}, which assumes the \emph{super-exponential} decay condition \eqref{reEDDdoubleexp} and thereby excludes commonly used determinantal point processes, including the Ginibre point process; and the Kolmogorov distance  bound of \citet{D25} for linear statistics of unmarked determinantal point processes.

The key idea is that for these statistics,  the mixing condition~\eqref{mixing-score} can be upper bounded by sums of expectations  expressed with respect to first- and second-order Palm distributions of ${\cal P}$. By employing the couplings in~\citet{MO21}, we find that the discrepancies between the distribution of ${\cal P}$ and its first- and second-order Palm distributions are localized near the reference points of the Palm distributions. More precisely, under the conditioning that one or two reference points lie in $A$, the configuration in $B$ changes only slightly, in the sense that under the relevant coupling, the probability that the original and Palm distribution configurations differ by at least one point decays exponentially fast as $d(A,B)\to\infty$, and the total discrepancy between these configurations in the set $B$ is bounded by two points. This, together with some additional estimates, yields mixing of score sums.

We focus on a stationary determinantal point process ${\cal P}$ on $\real^d$ with Lebesgue measure as the reference measure and a Hermitian kernel $K: \R^d \times \R^d  \to \bf{C}$, namely $K(x,y):=K_0(x-y)=\overline{K_0(y-x)}$. Let  $\hat K_0(x)=\int_{\mathbb{R}^d} \mathrm{e}^{-2\pi i x\cdot t}K_0(t)\mathrm{d} t$ be  the Fourier transform of $K_0$, where $x{\cdot}t$ is the dot product. We assume:
\begin{description}
	\item{\bf Condition 1.} The function $K_0$ is continuous and complex-valued on $\mathbb{R}^d$, $K_0\not\equiv 0$, and $|K_0(x)|=O(\mathrm{e}^{-|x|})$ as $|x|\to\infty$.
	\item{\bf Condition 2.} $0\le \hat K_0\le 1.$
\end{description}

The determinantal point process ${\cal P}$ is simple with its distribution determined by 
$$\mean\left[\prod_{i=1}^m{\cal P}(A_i)\right]=\int_{A_1\times A_2\times \dots\times A_m}{\rm det}\[K_0(u_i-u_j)\]_{i,j=1}^mdu_1\dots \mathrm{d}u_m,$$
for all mutually disjoint bounded sets $A_1,\dots,A_m \in {\cal{B}}(\R^d)$, $m\in \mathbb{N}$. The function 
$\rho_m: (\R^d)^m \to [0, \infty)$ given by 
$\rho_m(x_1,\dots,x_m):=\det\[K_0(x_i-x_j)\]_{1\le i,j\le m}$ is the $m$-th order correlation function of ${\cal P}$.

\begin{thm}\label{example1} (normal approximation for sum of marks on a determinantal point process)
	For the determinantal point process ${\cal P}$ with kernel  satisfying Conditions~1 and 2, let $\hat{{\cal P}}$ be a marked point process  equipped with i.i.d. marks $\{M_x\}_{x\in {\cal P}}$. Assume that $0< \|M_{\0}\|_{p_0}<\infty$ for some $p_0>3$. Then with
	$$H_{\lambda}=\sum_{x\in {\cal P}\cap \Gamma_{\lambda}} M_x,\quad Z_{\lambda}:= \frac{H_{\lambda}-\E H_{\lambda}} {\sqrt{\var H_{\lambda}}},
$$
 we have 
	$$d_W(Z_{\lambda},Z) =  O\((\log \lambda)^{3d} \lambda^{-0.5}\).$$
\end{thm}

\begin{re} \label{re3.2} (comparison with related literature) The assumptions imply that  ${\cal P}$ has exponentially decaying correlations and thus by Theorem 5.2 in \cite{BYY25}, $(Z_{\lambda})_{\lambda \geq 1}$ converges to the standard normal in distribution.  Theorem \ref{example1} establishes a rate of convergence for the case of i.i.d. marks, which also covers the unmarked case. This result also extends the quantitative CLT in \cite{CX23}, since the determinantal point processes satisfying Conditions~1 and~2 are not currently known to be EDD. Furthermore, following the same idea as in the proof, we can show that the statement still holds if the marks $M_x$ depend on the location $x$, while remaining independent of one another and of all other points. 
\end{re}

It is not known whether ${\cal P}$ is EDD. To prove Theorem \ref{example1}, we verify directly that the conditions of
Theorem~\ref{mainthm1} are fulfilled.

\begin{lma}\label{lma1.2} 
	If Conditions~1 and 2 hold, then $K_0(\0)>\int_{\mathbb{R}^d}|K_0(x)|^2\mathrm{d}x:=\vertiii{K_0}_2^2$. 
\end{lma}

\noindent{\it Proof.} We prove the statement by contradiction. Assume that $K_0(\0)=\int_{\mathbb{R}^d}|K_0(x)|^2 \mathrm{d}x$. The Fourier inversion theorem states that
\begin{equation}\label{FIT}
	K_0(x)=\int_{\real^d}\mathrm{e}^{ 2\pi i x{\cdot}t }\hat K_0(t)\mathrm{d}t.
\end{equation} 
Evaluating at $x= \0$,  we get
\begin{equation}\label{lma1.2-1}
	K_0(\0)=\int_{\mathbb{R}^d}\hat K_0(t)dt\ge \int_{\mathbb{R}^d}(\hat K_0(t))^2\mathrm{d}t=\int_{\mathbb{R}^d}|K_0(x)|^2\mathrm{d}x.
\end{equation}
The inequality in \eqref{lma1.2-1} follows from Condition~2 whereas the last equality follows from the Plancherel theorem. Since $K_0(\0)=\int_{\mathbb{R}^d}|K_0(x)|^2 \mathrm{d}x$, the inequality  \eqref{lma1.2-1} is an  equality and Condition~2 ensures that $ \{t\in\mathbb{R}^d:\ \hat K_0(t)\not\in\{0,1\}\}$ has Lebesgue measure $0$.
 
On the other hand, Condition~1 ensures that $K_0$ is both integrable and square integrable with respect to the Lebesgue measure, which implies that $ \hat K_0$  must be continuous. This leads to the conclusion that $ \hat K_0$ must be constant, either $ \hat K_0\equiv 0$ or $ \hat K_0\equiv 1$. If $ \hat K_0\equiv 0$, then $K_0\equiv 0$, which contradicts $K_0\not\equiv 0$. If $ \hat K_0\equiv 1$, then $K_0$ is not square integrable with respect to the Lebesgue measure, which contradicts Condition~1.  
Thus, we have reached a contradiction, and the original assumption must be false. \qed

\noindent{\it Proof of Theorem~\ref{example1}.}  For all $x \in \P, \lambda \geq 1,$ put  $\sxi= M_x$. To apply part (a) of Theorem~\ref{mainthm1}, it is enough to check conditions \eqref{momground}, \eqref{momscore}, exponential mixing of score sums as in Definition \ref{WSM}, and that $\sigma_{\lambda}^2 =\var H_{\lambda} =\Omega(\lambda)$. From Section~2.2.2 and Section~2.1, Remark~(i) of \cite{BYY19}, $\mathbb{E}\({\cal P}(B)^k\)$ is finite for all bounded $B\in {\cal B}(\mathbb{R}^d)$ and all $k\in \mathbb{N}$, which ensures  \eqref{momground}. This, together with the independence of marks, $\|M_{\0}\|_{p_0}<\infty$ for some $p_0>3$ and Remark~\ref{re1.3} (iii), also ensures condition \eqref{momscore} in Theorem~\ref{mainthm1}. By combining Theorem~1.1.2~(ii) and Remark~(iii) related to Theorem~1.1.3 of \cite{BYY19} and Lemma~\ref{lma1.2}, we conclude that $\sigma_{\lambda}^2=\Omega(\lambda)$. It remains to show exponential mixing of score sums  \eqref{WSM}.  This will necessitate coupling arguments and some additional background material. 

Let $\P_x^{!}$ be the reduced Palm process of ${\cal P}$ at point $x$, that is, $\P_x^{!}+\delta_x\sim{\cal L}_x(\P)$.  Recall the assumptions required for the existence of a coupling  of the determinantal point process with kernel function $K(u,v)$ and its reduced Palm distributions in \cite{MO21}:
\begin{description}
	\item[Assumption 1.] The kernel  $K$ is Hermitian, that is, $K(u,v)=\overline{K(v,u)}$ for all $u,v\in \mathbb{R}^d.$
	\item[Assumption 2.] For any compact set $S\subset \mathbb{R}^d$, $\iint_{S^2}|K(u,v)| du dv<\infty,$ which means that $K$ is locally square integrable.
	\item[Assumption 3.] For any compact set $S\subset \mathbb{R}^d$, $\int_{S}|K(u,u)|\mathrm{d}u <\infty,$  which means that $K$ is of local trace class.
	\item[Assumption 4.] For any compact set $S\subset \mathbb{R}^d$, all eigenvalues of $K$ on $L^2(S)$, the space of square-integrable complex functions with respect to the Lebesgue measure restricted to $S$, lie within the interval $[0,1]$.
\end{description}

It is stated in \cite{MO21} that {Assumption}~$4$ is equivalent to the existence of the determinantal point process
 given that  {Assumptions} $1-3$ are satisfied, and the determinantal point process is unique, see, e.g., \citet[Theorem 4.5.5  \ \rm{and}  \ Lemma 4.2.6]{H09}. 

For the determinantal point process $\scrP$, $K(u,v)=K_0(u-v)$. 
We claim that Conditions~1 and 2 imply these four properties  (i)~$K_0$ is Hermitian, 
(ii)~$K_0(\0)\ge\vertiii{K_0}_2^2$, (iii)~$K_0(\0)> |K_0(x)|$, for all $x\ne \0$, and (iv)~all eigenvalues lie within $[0,1]$. Property (i) can be directly verified from \eqref{FIT}, and Property (ii) follows from \eqref{lma1.2-1}. Property (iv) is a direct consequence of Condition~$2$ and \citet[Proposition $3.1$]{LMR15}. Therefore, it remains to prove Property~(iii). To this end, write $b_1=\int_{\real^d}\cos^2(2\pi x{\cdot}t)\hat K_0(t)\mathrm{d}t$ and $b_2=\int_{\real^d}\sin^2(2\pi x{\cdot}t)\hat K_0(t)\mathrm{d}t$. Since $b_1+b_2=K_0(\0)\in(0,\infty)$, at least one of $b_1$ and $b_2$ is positive, so without loss, we assume that $b_1\in(0,\infty)$. Then
\begin{equation}\label{iii-1}
	\frac{|\cos(2\pi x{\cdot}t)|}{\sqrt{b_1}}\cdot \frac{1}{\sqrt{K_0(\0)}}\le \frac12\left(\frac{\cos^2(2\pi x{\cdot}t)}{b_1}+\frac{1}{K_0(\0)}\right),
\end{equation}
and 
\begin{align*}
	S_x&:=\left\{t:\ \frac{|\cos(2\pi x{\cdot}t)|}{\sqrt{b_1}}\cdot \frac{1}{\sqrt{K_0(\0)}}= \frac12\left(\frac{\cos^2(2\pi x{\cdot}t)}{b_1}+\frac{1}{K_0(\0)}\right)\right\}\\
	&=\left\{t:\ |\cos(2\pi x{\cdot}t)|=\sqrt{\frac{b_1}{K_0(\0)}}\right\}
\end{align*}
has Lebesgue measure $0$. Multiplying both sides of \eqref{iii-1} by $\hat K_0(t)$ and integrating in terms of $t$, we obtain
$$\int_{\real^d}\frac{|\cos(2\pi x{\cdot}t)|}{\sqrt{b_1}}\cdot \frac{1}{\sqrt{K_0(\0)}}\hat K_0(t) \mathrm{d}t< 1,$$
which is equivalent to 
\begin{equation}\label{iii-2}
	\int_{\real^d}|\cos(2\pi x{\cdot}t)|\hat K_0(t) \mathrm{d}t<\sqrt{b_1K_0(\0)}.
\end{equation}
Now, applying \eqref{iii-2} for the term with $\cos$ and the Cauchy-Schwarz inequality for the term with $\sin$ below, we get
\begin{align} 
	|K_0(x)|^2&=\left(\int_{\real^d}\cos(2\pi x{\cdot}t)\hat K_0(t) \mathrm{d}t\right)^2+\left(\int_{\real^d}\sin(2\pi x{\cdot}t)\hat K_0(t) \mathrm{d}t\right)^2\nonumber\\
	&<b_1K_0(\0)+b_2K_0(\0)=K_0(\0)^2,\label{iii-3}
\end{align}
hence Property (iii) follows.

Now, Properties (i)--(iv) show that $K(u,v)$ satisfies Assumptions~1-4.  
According to \cite[Theorem~1]{MO21}, $\P_x^{!}$ is also a determinantal point process with kernel $$K^{x}(u,v)=K_0(u-v)-\frac{K_0(u-x)K_0(x-v)}{K_0(\0)},$$
and there exists a coupling of ${\cal P}$ and ${\cal P}_x^{!}$ such that, almost surely, ${\cal P}_x^{!}\subseteq {\cal P}$ and ${\P}^{\Delta,x}={\cal P}\backslash {\cal P}_x^{!}$ 
 consists of at most one point, with 
\begin{equation*} \ignore{\label{onept}}
	p_x:=\mathbb{P}({\P}^{\Delta,x}\neq \emptyset)=\frac{\vertiii{K_0}_2^2}{K_0(\0)}, 
\end{equation*}
and conditional on ${\P}^{\Delta,x}\neq \emptyset$, the point in ${\P}^{\Delta,x}$ has intensity decaying exponentially fast with the distance to $x$:
\begin{equation*} \ignore{\label{twopt}}
	f_{x}(v)= \frac{|K_0(v-x)|^2} {\vertiii{K_0}_2^2}.
\end{equation*}
Furthermore, one can verify that $K^{x}(u,v)$ also satisfies Assumptions~1-3. For distinct $x_1, x_2 \in \real^d$, the second order intensity function at $x_1$ and $x_2$ is
$$\rho(x_1,x_2 )=K_0(\0)^2-|K_0(x_2-x_1)|^2\in \left(0, K_0(\0)^2\right].$$

Since the reduced Palm process ${\cal P}_{x_1}^{!}$ is also a determinantal point process, we can apply the result from \cite{MO21} to analyze the properties of the second-order reduced Palm process ${\cal P}_{x_1,x_2}^{!}$ of ${\cal P}$ at the distinct points $x_1$ and $x_2$. This process is the same as the reduced Palm distribution of ${\cal P}_{x_1}^{!}$ at point $x_2$, provided that ${\cal P}_{x_1}^{!}$ satisfies Assumptions~1-4. We thank J. M{\o}ller for pointing this out in a personal communication.

The Hermitian assumption can be directly verified using the Hermitian property of $K_0$.  {Assumptions}~2 and 3 follow from the compactness of the set $S$ and the boundedness of the kernel as shown in \eqref{iii-3}. Assumption~4 follows from the existence of the Palm process. 

Thus we have shown that Conditions 1 and 2 imply {Assumptions} 1-4.  Now we can finally establish the pertinent coupling.  Applying  \cite[Theorem~1]{MO21} again, there exists a coupling of ${\cal P}_{x_1,x_2}^{!}$ and ${\cal P}_{x_1}^{!}$ such that, almost surely, ${\cal P}_{x_1,x_2}^{!}\subseteq {\cal P}_{x_1}^{!}$, and ${\P}^{\Delta,x_1,x_2}:={\cal P}_{x_1}^{!}\backslash {\cal P}_{x_1,x_2}^{!}$ consists of at most one point. The probability of this event is given by 
$$p_{x_1,x_2}:=\mathbb{P}({\P}^{\Delta,x_1,x_2}\neq \emptyset)=\frac{\vertiii{K^{x_1}(x_2,\cdot)}_{2}^2}{K^{x_1}(x_2,x_2)}, $$ 
and conditional on ${\P}^{\Delta,x_1,x_2}\neq \emptyset$, the intensity of the point in ${\P}^{\Delta,x_1,x_2}$ is $$f_{x_1,x_2}(v)=|K^{x_1}(x_2,v)|^2/\vertiii{K^{x_1}(x_2,\cdot)}_{2}^2.$$

In summary, we can find a coupling of ${\cal P}_{x_1,x_2}^{!}$, ${\cal P}_{x_1}^{!}$ and ${\cal P}$ such that ${\cal P}_{x_1,x_2}^{!}\subseteq {\cal P}_{x_1}^{!}\subseteq {\cal P}$, and ${\P}^{\Delta,x_1,x_2}={\cal P}_{x_1}^{!}\backslash {\cal P}_{x_1,x_2}^{!}$ and $ {\P}^{\Delta,x_1}={\cal P}\backslash {\cal P}_{x_1}^{!}$ both contain at most one point and the intensity $f_{x_1,x_2}$ decays exponentially fast with respect to $|x_1 - v|$ and $|x_2 - v|$. Since the marks are independent, we automatically obtain a coupling of the corresponding marked point processes.

From Lemma~\ref{mixing-uncentered}, to show the exponential mixing  \eqref{mixing-score}, it is equivalent to show the corresponding property for the sums of scores without centering.  
We start with $j=2$. 
 Since $H_{\lambda,A}^2=\iint_{A^2}M_{x_1}M_{x_2}{\cal P}(d x_1){\cal P}(d x_2)$, using the second order Palm distribution and the independence of the marks conditional on the ground point process, we have for any disjoint  $A$ and $B$ in $\mathcal{B}(\mathbb{R}^d)$
\begin{align}
	&\left|\mathbb{E}\(H_{\lambda,A}^2g(H_{\lambda,B})\)-\mathbb{E}\(H_{\lambda,A}^2\)\mathbb{E}\(g(H_{\lambda,B})\)\right|\nonumber\\
	&= \left|\mathbb{E}\(\iint_{A^2}M_{x_1}M_{x_2}\left[g\(\int_{B}M_y{\cal P}(\mathrm{d}y)\)-g\(\int_{B}\tilde{M}_y{\cal P}'(\mathrm{d}y)\)\right]{\cal P}(d x_1){\cal P}(d x_2)\)\right|\nonumber\\
	&\le \left|\iint_{A^{2,\neq}}\rho(x_1,x_2)\mathbb{E}\(M_{\0}\)^2 \mathbb{E}\left[g\(\int_{B}M_y{\cal P}_{x_1,x_2}^{!}(\mathrm{d}y)\)-g\(\int_{B}\tilde{M}_y{\cal P}'(\mathrm{d}y)\)\right]dx_1dx_2\right|\nonumber\\
	&\quad+ \left|\int_{A}K_0(\0)\mathbb{E}\(M_{\0}^2\)\mathbb{E}\( \left[g\(\int_{B}M_y{\cal P}_x^{!}(\mathrm{d}y)\)-g\(\int_{B}\tilde{M}_y{\cal P}'(\mathrm{d}y)\)\right]\)\mathrm{d}x\right|,\label{thm3.1.1}
\end{align}
where here and in the following, ${\cal P}'$, $\tilde{M}_x$ denote independent copies of ${\cal P}$ and $M_x$, respectively. The assumption of independence between the marks and the ground process enables us to couple $\hat\P$ and its corresponding Palm process via $\P$ and its Palm process, resulting in the clean explicit form in \eqref{thm3.1.1}.

By condition 1, there are positive constants $C_1$ and $C_2$ such that $|K_0(x)|\le C_1\mathrm{e}^{-C_2|x|}$ for all $x\in \mathbb{R}^d$. 
For any $x \in A$ we bound $K_0(x - v)$ by $C_1\mathrm{e}^{-C_2d(A,B)}$. We find, for $x_1\neq x_2$,
\begin{align*}
	&\prob\(\({\cal P}\backslash {\cal P}_{x_1,x_2}^{!}\)(B)\ne0\)=\prob\(\({{\P}}^{\Delta,x_1,x_2}\cup {{{\P}}^{\Delta,x_1}}\)(B)\ne 0\)\nonumber\\
	&\le  {  \int_{B}  (p_{x_1}f_{x_1}(v)+ p_{x_1,x_2}  f_{x_1,x_2}(v) ) \mathrm{d}v } \nonumber\\
	&=\int_{B}\(\frac{|K_0(v-x_1)|^2}{K_0(\0)}+\frac{|K^{x_1}(x_2,v)|^2}{K^{x_1}(x_2,x_2)}\)\mathrm{d}v\nonumber\\
	&=\int_{B}\frac{1}{K_0(\0)}\(|K_0(v-x_1)|^2+\frac{|K_0(x_2-v)K_0(\0)-K_0(x_2-x_1)K_0(x_1-v)|^2}{K_0(\0)^2-|K_0(x_2-x_1)|^2}\)\mathrm{d}v\nonumber\\
	&\le\frac{1}{K_0(\0)} \text{Vol}(B)\[C_1^2\mathrm{e}^{-2C_2d(A,B)}+\frac{2(C_1^4\vee 1)\mathrm{e}^{-2C_2d(A,B)}}{K_0(\0)^2-|K_0(x_2-x_1)|^2}\]\\
	&=\frac{1}{K_0(\0)} \text{Vol}(B)\[C_1^2\mathrm{e}^{-2C_2d(A,B)}+\frac{2(C_1^4\vee 1)\mathrm{e}^{-2C_2d(A,B)}}{\rho(x_1,x_2)}\].
\end{align*}

By independence and the bound $\|M_{\0}\|_1\le \|M_{\0}\|_{p_0}<\infty$, the first term on the right-hand side of \eqref{thm3.1.1} is bounded above by
\begin{align}
	&2K_0(\0)^2\text{Vol}(A)^2\text{Vol}(B)\mathbb{E}(M_{\0})^2\frac{1}{K_0(\0)}\[C_1^4\mathrm{e}^{-2C_2d(A,B)}+(C_1^4\vee 1)\mathrm{e}^{-2C_2d(A,B)}\]\nonumber\\
	&=:\text{Vol}(A)^2\text{Vol}(B)C_3\mathrm{e}^{-C_4d(A,B)}.\label{thm3.1.2}
\end{align}
Similarly, for the second term on the right-hand side of \eqref{thm3.1.1}, we have
\begin{align}
	\prob\(\({\cal P}\backslash {\cal P}_{x}^{!}\)(B)\ne 0\)=\int_{B}\frac{|K_0(v-x_1)|^2}{K_0(\0)}\mathrm{d}v\le \frac{1}{K_0(\0)}\text{Vol}(B)C_1^2\mathrm{e}^{-2C_2d(A,B)},\label{thm3.1.4}
\end{align} 
which, together with the boundedness of $g$ and the bound  $\|M_{\0}\|_2\le \|M_{\0}\|_{p_0}<\infty$, ensures that the second term is bounded above by 
\begin{equation}
	2\text{Vol}(A)\text{Vol}(B)\mathbb{E}(|M_{\0}|)C_1^2\mathrm{e}^{-2C_2d(A,B)}:=\text{Vol}(A)\text{Vol}(B)C_5\mathrm{e}^{-C_4d(A,B)}. \label{thm3.1.3}
\end{equation} 
Combining \eqref{thm3.1.1}, \eqref{thm3.1.2} and \eqref{thm3.1.3} gives
\begin{equation}\label{thm3.1.7}
\left|\mathbb{E}\(H_{\lambda,A}^2g(H_{\lambda,B})\)-\mathbb{E}\(H_{\lambda,A}^2\)\mathbb{E}\(g(H_{\lambda,B})\)\right|\le (\text{Vol}(A)^2\vee 1)\text{Vol}(B)(C_3+C_5)\mathrm{e}^{-C_4d(A,B)}.
\end{equation}

For $j=1$ and $g(x)=x$, we have
\begin{align}
	&\left|\mathbb{E}(H_{\lambda,A}H_{\lambda,B})-\mathbb{E}(H_{\lambda,A})\mathbb{E}(H_{\lambda,B})\right|\nonumber\\
	&=\left|\mathbb{E}\iint_{A\times B}M_{x_1}M_{x_2}\mathcal{P}(\mathrm{d}x_1)\mathcal{P}(\mathrm{d}x_2)-\(\mathbb{E}\int_A M_{x}\mathcal{P}(\mathrm{d}x)\)\(\mathbb{E}\int_B M_{y}\mathcal{P}(\mathrm{d}y)\)\right|\nonumber\\
	&=\left|\int_{A}K_0(\bm{0})\mathbb{E}(M_{\bm{0}})\mathbb{E}\(\int_B M_y\mathcal{P}_x^{!}(\mathrm{d}y)-\int_B \tilde{M}_y\mathcal{P}'(\mathrm{d}y)\)dx\right|\nonumber\\
	&\le 2\mathbb{E}(M_{\bm{0}})^2\Vol(A)\Vol(B)C_1^2\mathrm{e}^{-2C_2d(A,B)},\label{thm3.1.5}
\end{align}
where the last inequality follows from \eqref{thm3.1.4}. If $g\in\mathcal{F}_b$ instead, replacing $M_{x_2}$ in \eqref{thm3.1.5} by $g(M_{x_2})$ yields 
\begin{align}
	&\left|\mathbb{E}(H_{\lambda,A}g(H_{\lambda,B}))-\mathbb{E}(H_{\lambda,A})\mathbb{E}(g(H_{\lambda,B}))\right|\nonumber\\
	&=\left|\int_{A}K_0(\bm{0})\mathbb{E}(M_{\bm{0}})\mathbb{E}\(\int_B g(M_y)\mathcal{P}_x^{!}(\mathrm{d}y)-\int_B g(\tilde{M}_y)\mathcal{P}'(\mathrm{d}y)\)dx\right|\nonumber\\
	&\le 2|\mathbb{E}(M_{\bm{0}})|\Vol(A)\Vol(B)C_1^2\mathrm{e}^{-2C_2d(A,B)}. \label{thm3.1.6}
\end{align}
Combining \eqref{thm3.1.5}, \eqref{thm3.1.6}, and the moment properties of $M_{\bm{0}}$ discussed above, we obtain:
$$
\left|\mathbb{E}\(H_{\lambda,A}g(H_{\lambda,B})\)-\mathbb{E}\(H_{\lambda,A}\)\mathbb{E}\(g(H_{\lambda,B})\)\right|\le \text{Vol}(A)\text{Vol}(B)C_6\mathrm{e}^{-C_4d(A,B)}.
$$
Together with \eqref{thm3.1.7} and Lemma~\ref{mixing-uncentered}, this implies the exponential
 mixing  of score sums \eqref{mixing-score}
for all disjoint $A, B$ without requiring that they satisfy a separation condition.  We have
  $\gamma_0=\gamma_0'+d=2d+d=3d$,  $\gamma_2=\gamma_2'=C_4$, and any $\gamma_3>0$, since $\Vol(D)\le (\diam(D))^d$ for all $D\in{\cal B}(\mathbb{R}^d)$. The proof now follows from  Theorem~\ref{mainthm1}~(a). \qed

\subsection{Unevenly spaced time series}

Consider an unevenly spaced time series, namely  a sequence of random events\\ $\{(T_i,U_i)\}_{i\in\Z:=\{\dots,-2,-1,0,1,2,\dots\}}$, where the time intervals between successive recording times $\{T_i\}_{i\in\Z}$ are not constant. A common approach to analyzing such data involves transforming it into equally spaced observations through interpolation (see \citet{T90}, \citet{BD16}, and references therein). We take a different approach and work with the untransformed data.

In this section, we  use Theorem \ref{mainthm1}~(b) to establish a quantitative CLT for sums of score functions driven by the sequence $\{(T_i,U_i)\}_{i\in\Z}$, where  $\{T_i\}_{i\in\Z}$ is an increasing sequence forming a stationary renewal point process, with $T_1$ being the first renewal after time 0, and $\{U_i\}_{i\in\Z}$  a strictly stationary sequence of random variables independent of $\{T_i\}_{i\in\Z}$. The process $\{(T_i,U_i)\}_{i\in\Z}$ is a marked point process, with the stationary renewal point process as the underlying ground point process ${\cal P}=\sum_{i\ge 1}\delta_{T_i}$ and with marks  $U_i$ associated with the points $T_i$ taking values in a given measurable mark space $\M$.  We do not assume that   ${\cal P}$ satisfies the EDD property.

The behavior of a renewal point process depends on the properties of its renewal times. The renewal point process is called {\it non-lattice} if there do not exist constants $a\in \mathbb{R}$ and $b\in(0,\infty)$ such that the support of the inter-arrival times is a subset of $a+b\mathbb{N}$. It is called {\it spread-out} if there exists a positive integer $i$ such that the $i$-th convolution of the distribution function, $G$, of the inter-arrival times has an absolutely continuous component. That is, $G^{*i}=c_1G_a+c_2G_s$ holds for two constants $c_1\in(0,\infty),~c_2=1-c_1\ge 0$, where $G_a$ is an absolutely continuous distribution and $G_s$ is another distribution. If a renewal point process is spread-out, then it is non-lattice. However, the reverse is not true: if the inter-arrival distribution is discrete and takes all positive rational values with positive probability, then the process is non-lattice but not spread-out.

The dependence structure of the marks is as follows. Fix $k_0 \in \N$. Each mark $U_i$ depends on the values of the preceding $k_0$ marks, with additional dependence on $R_i$, where   $\{R_i\}_{i\in\Z}$ is a sequence of $\iid$ random variables taking values in some measurable space $({\cal E},{\mathscr{E}})$. We assume 
\begin{equation}\label{timeseriesdefa1}
	U_i=f(\{U_j\}_{i-k_0 \le j< i}, R_i),   \  i \in \Z,  
\end{equation}
where $f: {\M}^{k_0}\times {\cal E} \to \M$ is measurable,  and $R_i$ is independent of both $\{T_j\}_{j\in\Z}$ and $\{U_j\}_{j\le i-1}$. This assumption includes the autoregressive process of order $k_0$ as a special case and guarantees that $\{U_i\}_{i\in \Z}$ is strictly stationary. Under this setting, the marked point process $\hat{{\cal P}}=\sum_{i\in \mathbb{Z}}\delta_{(T_i,U_i)}$ on $\mathbb{R}\times \M$ reduces to a point-process representation of an AR($k_0$) model when $T_i=i$ for $i\in\mathbb{Z}$, that is, when the renewal point process is lattice-based.

Define $\xi: (\R \times \mathbb{M}) \times  {\cal N}_{\R \times  \mathbb{M}} \to \R$ by  
\begin{equation}\label{timeseriesdefa2}
	\xi((T_i,U_i),\hat{{\cal P}})=\xi_{loc}\((T_i,U_i),\hat{{\cal P}}|_{[T_{i-j_0}\wedge (T_i-t_0), T_{i+j_0}\vee (T_i+t_0)]}\)
\end{equation}
for all $i \in \mathbb{Z}$, where $j_0 \in \mathbb{N}\cup\{0\}$ and  $t_0\ge 0$ are constants, and where $\xi_{loc}:  (\R \times \mathbb{M}) \times  {\cal N}_{\R \times  \mathbb{M}}  \to \R$ characterises the interaction between a renewal point (together with its mark) and its surrounding window in time, as measured by $t_0$ and/or the number of renewals $j_0$. 
We define $H_{\lambda}=\sum_{T_i\in {\cal P}\cap \Gamma_{\lambda}}\xi((T_i, U_i),\hat{{\cal P}})$. Statistics $H_{\lambda}$ of interest are generated by  choosing $\xi$ as follows:
\begin{description}
	\item{(i)} $\xi((T_i,U_i),\hat{{\cal P}})=U_i$, in which case $H_{\lambda}$ gives the total sum of a strictly stationary sequence of random variables counted by the renewal process;
	\item{(ii)} $\xi((T_i,U_i),\hat{{\cal P}})=(U_i-\mathbb{E}U_{i})(U_{i-j}-\mathbb{E}U_{i-j})$, in which case $H_{\lambda}$ becomes, up to a multiplicative constant, the estimator of the $j$-th autocovariance in time series, cf.  \citet[Section $3.1$]{H94}; 
	\item{(iii)} $\xi((T_i,U_i),\hat{{\cal P}})=g(T_i-T_{i-1})$ for some  function $g: \R \to \R$, when one studies the renewal-reward process (cf. \citet[Section~$7.4$]{R10}); and
	\item{(iv)} $\xi((T_i,U_i),\hat{{\cal P}})=h(U_i)$ for some function $h:\mathbb{M}\rightarrow \real$;  when $k_0=1$ and $\{T_i\}_{i\in \Z}$ reduces to the lattice $\Z$, one obtains  the discrete version of the additive Markov functionals, cf. \citet{R70}. 
\end{description}

Since the ground process and the marks arise from an iterative approach, we impose a decay-of-dependence condition on the marks.  This is achieved through the {\it $\phi$-mixing coefficient} between two $\sigma$-algebras $\mathscr{A}$ and $\mathscr{B}$, defined as 
$$ \phi(\mathscr{A}, \mathscr{B}):=\sup_{E\in \mathscr{A}, E'\in \mathscr{B},\mathbb{P}(E)>0}\left|\mathbb{P}(E'|E)-\mathbb{P}(E')\right|.$$

We define $H_{\lambda}=\sum_{T_i\in {\cal P}\cap \Gamma_{\lambda}}\xi((T_i, U_i),\hat{{\cal P}})$, $\sigma_{\lambda}^2=\Var H_{\lambda}$, and $Z_{\lambda}=(H_{\lambda}-\mathbb{E}H_{\lambda})/\sigma_{\lambda}$, where $\Gamma_{\lambda}=[-0.5\lambda,0.5\lambda]$. The main result of this section is as follows:

\begin{thm}\label{example2}(normal approximation for statistics of a spread-out renewal point process)
	Given  a marked stationary spread-out renewal point process,  assume that  the inter-arrival times have a finite $\alpha$-th moment for some $\alpha>2$, $\mathbb{E}_{{\bf 0}}\(|\xi(({\bf 0},U_{\bf 0}),\hat{{\cal P}})|^{p}\)<\infty$ for some constant $p>p_0 \ge 3$,  and
	\begin{equation}\label{phimixing}
		\phi(\sigma(U_1,U_2,\dots,U_{k_0}),\sigma(U_{n_0+1},U_{n_0+2},\dots, U_{n_0+ k_0}))<\frac12
	\end{equation}
	for some $n_0 \in \N, n_0 > k_0$. If $\var H_{\lambda}=\Omega(\lambda^\nu)$ for some $\nu>2/3$, and $$\frac{(36-6\nu)\vee 28}{3\nu-2}<\frac{\alpha (p_0-2)}{p_0},$$
	then 
	\begin{align*}
		d_W\( \frac{H_{\lambda} - \E H_{\lambda}} {\sqrt{\Var H_{\lambda}}}, Z\)&=O\(\lambda^{-1.5\nu+1+\tau(\g_0,\g_2)}\)= O\(\lambda^{-1.5\nu+1+3p_0\left(\frac{2\nu+4}{6p_0+\alpha(p_0-2)}\vee\frac{\nu+4}{2p_0+\alpha(p_0-2)}\right)}\)\\
		&= o\(1\),
	\end{align*}
	where $Z\sim N(0,1)$, $\g_0=2$, $\g_2=\alpha (p_0-2)/(2p_0)$ and $\tau$ is defined in \eqref{taudef}.
\end{thm}

\begin{re} 
	When  the ground process  $\P$ is a homogeneous Poisson point process, we can apply Theorem \ref{mainthm1}~(a) instead of (b) and use the same approach to derive a better bound. We omit the details.
\end{re}

The outline of the proof is as follows. We apply Theorem~\ref{mainthm1}~(b), where the main challenge is to establish the polynomial decay property in Definition~\ref{WSM}. Since $d=1$, $A$ is an interval and $B$ lies on both sides of $A$
with $d(A,B) = \rho$. The estimate of $\cov(S_A,g(S_B))$ reduces to accounting for the cost of replacing $S_B$ by independent copies of $S_{B_1}$ and $S_{B_2}$, where $B_1 \subseteq B$ is the interval to the left of $A$ and  $B_2 \subseteq B$ is the interval to the right.

To carry out this replacement, we must also account for the number of renewals  between  $A$ and $B_1$ and also between $A$ and $B_2$. Lemma~\ref{example2.lma1} shows that the number of renewals is roughly linear in the length of the interval of interest, while Lemma~\ref{example2.lma2} shows that the cost of replacing each piece in the ground process by an independent copy can be controlled using the renewal point process property: given sufficient time, the renewal point process returns to stationarity, allowing us to apply the same argument one interval at a time. The cost of replacing the marks is also controllable by considering the lifted mark process $\{{\bf U}_i\}_{i\in \Z}$, where ${\bf U}_i:=(U_{i},U_{i+1},\dots, U_{i+k_0-1})$, as a Markov process and using the $\phi$-mixing property of the Markov process, as in \citet{B05}, to control $\phi$-mixing on $\sigma$-algebras distant $\rho$ from each other. The final step is to combine these two lemmas; the complexity arises from the fact that each of $A$, $B_1$ and $B_2$ requires two-sided neighborhoods to allow for the dependence of the marks. 

Now we consider the asymptotic linearity of the number of renewals in an interval.

\begin{lma}\label{example2.lma1} (asymptotic linearity of the number of renewals) 
	Given a stationary non-lattice renewal point process $\{T_i\}$, assume the inter-arrival times $\tau_i:=T_{i+1}-T_i$ satisfy $\mean(\tau_i^\alpha)<\infty$ for some $\alpha>2$. Then for any constant  $C\in \(0,(2\mathbb{E}\tau_1)^{-1}\)$, there exist positive constants $C'$ and $C''$ depending on $C$ such that the number of renewals $N_t$ in the time interval $[0,t]$ satisfies 
	$$ \mathbb{P}\left(N_t\le C t\right) \le C't^{-\alpha/2} \mbox{ for }t\ge C''.$$
\end{lma}

\noindent{\it Proof.} Let $G$ be the distribution of the inter-arrival times with mean $\mu_G=\mathbb{E}(\tau_1)$, and let $N_{o,t}+1$ represent the number of renewals in the interval $[0,t]$, conditioned on a renewal at $0$. Let $F_{T_1}$ be the distribution function of $T_1$ and let $n_t=\lfloor Ct\rfloor$. 
Then
\begin{align}
	\mathbb{P}\left(N_t\le Ct\right)&= \int_0^\infty\mathbb{P}\left(N_t\le n_t|T_1=s\right)\mathrm{d}F_{T_1}(s)\nonumber\\
	&\le \int_0^{t/2}\mathbb{P}\left(N_{o,t-s}+1\le n_t\right)\mathrm{d}F_{T_1}(s)+\mathbb{P}\left(T_1>\frac{t}{2}\right)\nonumber\\
	&\le \mathbb{P}\left(N_{o,t/2}\le n_t-1\right)+\mathbb{P}\left(T_1>\frac{t}{2}\right).\label{example2.lma1.1}
\end{align}
Now, since $\mathbb{E}(T_1^{\alpha-1})= \mathbb{E}(\tau_1^\alpha)/(\alpha\mu_G) <\infty$ and $(t/2)^{\alpha-1}\mathbb{P}(T_1>t/2)\le \int_{t/2}^\infty s^{\alpha-1}\mathrm{d}F_{T_1}(s)\to 0$ as $t\to\infty$, we obtain
\begin{equation}\label{example2.lma1.2}
	\mathbb{P}\(T_1>\frac{t}2\)=o(t^{-(\alpha-1)}).
\end{equation}

	For the first term in \eqref{example2.lma1.1}, let $Y_i=\tau_i-\mu_G$, $S_{n}:=\sum_{i=1}^n\tau_i$, and $\zeta_{n}:=\sum_{i=1}^nY_i$. The probability $\mathbb{P}(N_{o,t/2}\le n_t-1)$ is related to the sum $S_{n_t}$, which can be expressed as follows:
\begin{equation}\label{example2.lma1.3}
	\mathbb{P}\left(N_{o,t/2}\le n_t-1\right)=\mathbb{P}\left(S_{n_t}>\frac{t}2\right)=\mathbb{P}\left(\zeta_{n_t}>\frac{t}2-n_t\mu_G\right).
\end{equation}
Next, we define $w_t=t/2-n_t\mu_G > 0$, and apply Chebyshev's inequality followed by a Rosenthal inequality in \citet[p.~279]{Rosenthal70}, which states that given $\alpha>2$, there exists a positive constant $C_\alpha$ depending only on $\alpha$  such that for any sequence of independent random variables $Y_i$ in the space $L^{\alpha}$ with $\mathbb{E}Y_i=0$, $1\le i\le n$,  we have 
$$
\left\|\sum_{i=1}^nY_i\right\|_\alpha\le C_\alpha\max\left\{\(\sum_{i=1}^n\mathbb{E}\(|Y_i|^\alpha\)\)^{1/\alpha},\(\sum_{i=1}^n\mathbb{E}\(Y_i^2\)\)^{1/2}\right\}.$$
This gives 
$$
\mathbb{P}(\zeta_{n_t}>w_t)\le \frac{\mathbb{E}(|\zeta_{n_t}|^\alpha)}{w_t^\alpha}\le c_\alpha w_t^{-\alpha}\max\left\{n_t\mathbb{E}(|Y_1|^\alpha),n_t^{\alpha/2}(\mathbb{E}(Y_1^2))^{\alpha/2}\right\},$$
where $c_\alpha:=C_\alpha^\alpha$ is a positive constant and hence
\begin{align}\label{example2.lma1.4}
	\mathbb{P}(\zeta_{n_t}>w_t)&=O\left(\frac{n_t^{\alpha/2}}{w_t^\alpha}\right)=O(t^{-\alpha/2}).
\end{align}
Combining (\ref{example2.lma1.1}) - (\ref{example2.lma1.4}) completes the proof. \qed 

\begin{lma}\label{example2.lma2} (the cost of replacing renewals on intervals by independent copies)
	Let $\{T_i\}$ be a stationary, spread-out renewal point process, where the inter-arrival times have a finite $\alpha$-th moment for some $\alpha>2$. Let $I_i=[a_i,b_i]$ for $i=1,\dots,{m}$, where ${m}\ge 2$ and $-\infty\le a_1<b_1<a_2<b_2<\dots<a_{m}<b_{m}\le \infty$, with the convention $(a,\infty]=(a,\infty)$ and $[-\infty, b]= (-\infty, b]$.  Let  $\mathscr{R}_i$ be the restriction of the renewal point process to the interval $I_i$,  and let $\mathscr{R}'_i$ be a copy of $\mathscr{R}_i$, such that \{$\mathscr{R}'_i:\ 1\le i\le {m}\}$ are independent of each other. For any constant $C\in(0,\infty)$, if the minimum separation distance between intervals satisfies $\min_{1\le i\le {m}-1}(a_{i+1}-b_i)\ge Ct$, then we have as $t\to\infty$
	$$ d_{TV}\left((\mathscr{R}_1,\dots, \mathscr{R}_{m}), (\mathscr{R}'_1,\dots,\mathscr{R}'_{m})\right) = o\left(t^{-(\alpha-1)}\right). $$
\end{lma}

\noindent{\it Proof.} Without loss of generality, we assume that $\{\mathscr{R}_i\}$ and 
$\{\mathscr{R}'_i\}$ are defined on the same probability space and are independent. Let ${\cal R}_i=\left(\mathscr{R}_1,\dots,\mathscr{R}_i,\mathscr{R}'_{i+1},\dots,\mathscr{R}'_{m}\right)$ for $1\le i\le m$ with the convention ${\cal R}_m=\left(\mathscr{R}_1,\dots,\mathscr{R}_m\right)$. The triangle inequality gives 
$$d_{TV}\left((\mathscr{R}_1,\dots, \mathscr{R}_{m}), (\mathscr{R}'_1,\dots, \mathscr{R}'_{m})\right)\le \sum_{i=2}^{m}d_{TV}\left({\cal R}_i, {\cal R}_{i-1}\right),
$$
and hence it suffices to show that 
$$d_{TV}\left({\cal R}_i, {\cal R}_{i-1}\right)=o\left(t^{-(\alpha-1)}\right)$$
for all $2\le i\le {m}$. 

To this end, let $\Omega$ denote all locally finite subsets of $\cup_{j=1}^{m}I_j$, representing all possible configurations of the renewal point processes restricted to $\cup_{j=1}^{m}I_j$,
and let $\Omega_i$ represent all locally finite subsets of $\left(\cup_{j=1}^{m}I_j\right)\backslash I_i$, representing all possible configurations of the renewal point processes restricted to $\left(\cup_{j=1}^{m}I_j\right)\backslash I_i$. Each configuration can be viewed as a finite integer-valued measure, and we equip  both $\Omega$ and $\Omega_i$ with the weak topology. Let $\cF$ and $\cF_i$ be the $\sigma$-algebras generated on $\Omega$ and $\Omega_i$, respectively. For ease of notation, let ${\cal R}_{i-}=\left(\mathscr{R}_1,\dots,\mathscr{R}_{i-1},\mathscr{R}'_{i+1},\dots,\mathscr{R}'_{m}\right)$. Without loss of generality, assume $b_{i-1}=0$, so we can use the notation $T_1$ to denote the first renewal time after $b_{i-1}$. Then, we have
\begin{align}
	&d_{TV}\left({\cal R}_i, {\cal R}_{i-1}\right)\nonumber\\
	&=\sup_{A\in \cF} \left|\prob\left({\cal R}_i\in A\right)-\prob\left({\cal R}_{i-1}\in A\right)\right| \nonumber\\
	&\le\sup_{A\in \cF}\int_{\Omega_i}\int_0^\infty \left|\prob\left({\cal R}_i\in A|{\cal R}_{i-}=\xi,T_1=s\right)-\prob\left({\cal R}_{i-1}\in A|{\cal R}_{i-}=\xi,T_1=s\right)\right|\nonumber\\
	&\ \ \ \ \ \ \ \ \ \times\prob\left({\cal R}_{i-}\in \mathrm{d}\xi,T_1\in \mathrm{d}s\right)\nonumber\\
	&\le \int_{\Omega_i}\int_0^\infty d_{TV}\left({\cal L}(\mathscr{R}_i|{\cal R}_{i-}=\xi,T_1=s),{\cal L}(\mathscr{R}_i)\right)\prob\left({\cal R}_{i-}\in \mathrm{d}\xi,T_1\in \mathrm{d}s\right)\nonumber\\
	&\le \int_{\Omega_i}\int_0^{0.5Ct} d_{TV}\left({\cal L}(\mathscr{R}_i|{\cal R}_{i-}=\xi,T_1=s),{\cal L}(\mathscr{R}_i)\right)\prob\left({\cal R}_{i-}\in \mathrm{d}\xi,T_1\in \mathrm{d}s\right)\nonumber\\
	&\ \ \ +\prob\left(T_1>0.5Ct\right).\label{renewalxia-02}
\end{align}
Conditional on $T_1=s$, the distribution of $\mathscr{R}_i$ does not depend on ${\cal R}_{i-}$, and hence the bound \eqref{renewalxia-02} simplifies to
\begin{align}
	&d_{TV}\left({\cal R}_i, {\cal R}_{i-1}\right)\nonumber\\
	&\le\int_0^{0.5Ct} d_{TV}\left({\cal L}(\mathscr{R}_i|T_1=s),{\cal L}(\mathscr{R}_i)\right)\prob\left(T_1\in \mathrm{d}s\right)+\prob\left(T_1>0.5Ct).\right.\label{renewalxia-03}
\end{align}
Now, we use the coupling defined in \citet[Section~11.2]{Rta}
 to complete the proof. Let $\{T_{o,i}\}$, with $T_{o,0}=0$, be the zero-delayed renewal point process where the inter-renewal times have the same distribution as that of $\tau_1$. From the proof in \cite[p.~435--436]{Rta}, we know that there exists a positive random variable $\mathscr{T}$ and a coupling $\left(\{T_{o,i}'\},\{T_i'\}\right)$ such that
(i) ${\cal L}\left(\{T_{o,i}'\}\right)={\cal L}\left(\{T_{o,i}\}\right)$ and ${\cal L}\left(\{T_i'\}\right)={\cal L}\left(\{T_i\}\right)$, (ii) the renewal times of $\{T_{o,i}'\}$ coincide with those of $\{T_i'\}$ at and after $\mathscr{T}$, (iii) $\mean \left[\mathscr{T}^{\alpha-1}\right]<\infty$ 
\cite[Lemma~3.1]{Rta}. Using this coupling, we have for all $s \in [0,0.5Ct]$,
\begin{align}
	&d_{TV}\left({\cal L}(\mathscr{R}_i|T_1=s),{\cal L}(\mathscr{R}_i)\right)\nonumber\\
	&\le \prob(\{T_{o,i}'\}\cap [a_i-s,b_i-s]\ne \{T_i'\}\cap [a_i-s,b_i-s])\nonumber\\
	&\le \prob(\mathscr{T}> a_i-s).\label{renewalxia-04}
\end{align}
Thus, using the same proof as that for \eqref{example2.lma1.2}, from \eqref{renewalxia-03} and \eqref{renewalxia-04}, we obtain as $t \to \infty$
\begin{equation*}
	d_{TV}\left({\cal R}_i, {\cal R}_{i-1}\right)\le \prob(\mathscr{T}>a_i-0.5Ct)+\prob\left(T_1>0.5Ct\right)=o\left(t^{-(\alpha-1)}\right),
\end{equation*}
since $\mean\left(\mathscr{T}^{\alpha-1}\right)<\infty$ and $a_i-0.5Ct\ge 0.5Ct$, where we have assumed that $b_{i - 1} = 0$. This completes the proof. \qed 

\noindent{\it Proof of Theorem~\ref{example2}.}  
We deduce this from Theorem \ref{mainthm1}~(b).  The proof has multiple steps. 

\noindent Step (i).   Showing moment assumptions on ground point process and on the score.

We begin by verifying the moment assumptions \eqref{momground}, \eqref{momscore}.  Denote by $N_t$ and $N_{o,t}+1$ the numbers of renewals of the process $\{T_i\}_{i\in \Z}$ and its $0$-delayed counterpart $\{T_{o,i}\}$, respectively, in the time interval $[0,t]$.  From the construction, it follows that $N_t$ is stochastically dominated by $N_{o,t}+1$ for all $t$. 

To verify the moment condition on the ground process  \eqref{momground}, it is sufficient to show that $N_{o,1}$ has finite moments of arbitrary order. Since the inter-arrival times $\tau_i=T_{i+1}-T_i$ are positive with a positive probability, there exists a positive $x_0$ such that $\mathbb{P}(\tau_1\ge x_0)>p'$ for some constant $p' >0$. Therefore, $N_{o,1}$ is stochastically dominated by a negative-binomial random variable with parameters $r=\left\lceil 1/x_0 \right\rceil+1$ and $p=p'$, which has finite moments of arbitrary order. Thus, assumption \eqref{momground} holds for all large $p_1$. 

The assumed $p$-th moment condition with $p>3$ together with arbitrarily large $p_1$ ensures the moment conditions in Theorem \ref{mainthm1}~(b) according to Remark~\ref{re1.3}~(iii).

Step (ii).  Constructing a partition of $ \Gamma_{\lambda}$ and a coupling on this partition.

We aim to verify polynomial mixing of score sums  in \eqref{mixing-score'} with $A= [a,b]$ in \eqref{mixing-score'}  and $\Gamma_{\lambda} = [-0.5\lambda, 0.5\lambda]$.  Without loss, put  $B = \Gamma_{\lambda} \setminus B_\rho(A)$  for some positive $\rho$. If $B\subset \Gamma_{\lambda} \setminus B_\rho(A)$, the argument works with the same coupling. Since $d=1$, $A$ is an interval and $B$ lies on both sides of $A$.  As mentioned earlier,  the estimate of $\cov(S_A,g(S_B))$ reduces to accounting for the cost of replacing $S_B$ by independent copies of $S_{B_1}$ and $S_{B_2}$, where $B_1$ is the subset of $B$ to the left of $A$ and $B_2$ is the subset to the right.  To do this, we introduce a partition to help control dependencies of renewals and dependencies of marks, which goes as follows.

Without loss of generality, we can assume that $\max\{|a-\rho|, |b+\rho|\}<0.5\lambda$ and $\rho\ge 3t_0$, which,  together with Remark~\ref{remark1.5a}~(i), also ensure that $A$ and $B$ satisfy the condition in Definition~\ref{WSM} with $\gamma_3=3t_0$. 
These conditions ensure that $B=[-0.5\lambda,a-\rho]\cup[b+\rho,0.5\lambda]=:B_1\cup B_2$ lies on both sides of $A$, and in the following construction, the neighborhoods of intervals $A$, $B_1$ and $B_2$ all have length greater than or equal to $t_0$. When $\rho$ is sufficiently large, with high probability, the score functions evaluated at renewals in $A$ are determined by renewal points in the  $\rho/3$-enlarged  interval $B_{\rho/3}(A)$ given by
$$ A_e := \left[a-\frac{\rho}{3},\; b+\frac{\rho}{3}\right]$$
together with their marks, while the score functions with centers in $B$ are determined by the renewal points in an enlarged  interval around $B$, namely
$$ \left(-\infty,\; a-\frac{2\rho}{3}\right] \cup \left[b+\frac{2\rho}{3},\; \infty\right) =: B_{1,e} \cup B_{2,e}$$
together with their marks.
The gap between these two sets,
$$ D := \left(a-\frac{2\rho}{3},\; a-\frac{\rho}{3}\right) \cup \left(b+\frac{\rho}{3},\; b+\frac{2\rho}{3}\right) =: D_1 \cup D_2,$$
is used with high probability to break the dependence between the sums of score functions with arrival times in sets $A$ and $B$.
Since $t_0$ also plays a role in the dependence range of the score function and since  it extends the dependence range even further, we further define the following intervals:
\begin{align*}
	B_{1,re}&=\left(a-\rho,a-\frac{2\rho}{3}\right],\ \ B_{2,le}=\left[b+\frac{2\rho}{3},b+\rho-t_0\right),\\
	A_{le}&=\left[a-\frac{\rho}{3},a-t_0\right),  \ \ A_{re}=\left(b,b+\frac{\rho}{3}\right]. 
\end{align*}
The above intervals are illustrated in Figure~\ref{figurepm1}.

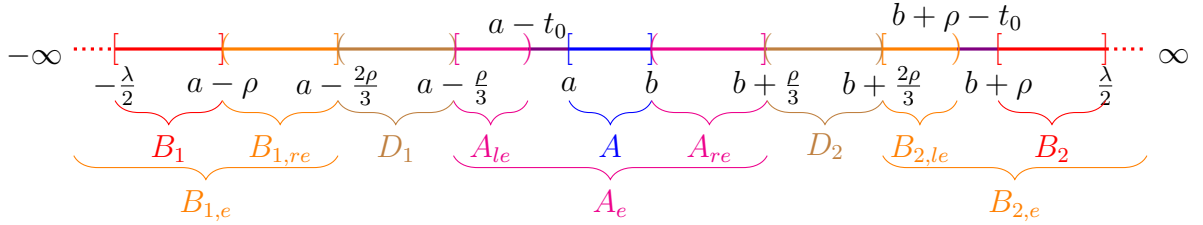
\begin{figure}[htbp]
	\centering
	\begin{tikzpicture}[scale=0.91, x=1.2cm, y=1cm]
		
		\coordinate (L)   at (-6.5,0);
		\coordinate (Ln)   at (-6,0);
		\coordinate (ar)  at (-4.7,0);   
		\coordinate (a2)  at (-3.3,0);   
		\coordinate (a3)  at (-1.9,0);   
		\coordinate (at0) at (-1,0);   
		\coordinate (a)   at (-0.5,0);
		\coordinate (b)   at (0.5,0);
		\coordinate (b3)  at (1.9,0);      
		\coordinate (b2)  at (3.3,0);      
		\coordinate (bt0) at (4.2,0);      
		\coordinate (br)  at (4.7,0);      
		\coordinate (Rn)  at (6,0);       
		\coordinate (R)   at (6.5,0);
		
		\draw[dotted, very thick, red]        (L)  -- (Ln);
		\draw[very thick, red]        (Ln)  -- (ar);
		\draw[very thick, orange]     (ar) -- (a2);
		\draw[very thick, brown] (a2) -- (a3);
		\draw[very thick, magenta]      (a3) -- (at0);
		\draw[very thick, violet]       (at0) -- (a);
		\draw[very thick, blue]       (a)  -- (b);
		\draw[very thick, magenta]    (b) -- (b3);
		\draw[very thick, brown]      (b3) -- (b2);
		\draw[very thick, orange]       (b2) -- (bt0);
		\draw[very thick, violet]       (bt0) -- (br);
		\draw[very thick, red]       (br) -- (Rn);
		\draw[dotted, very thick, red]       (Rn) -- (R);
		
		\node[red] at (Ln) {$[$};
		\node[red] at (ar) {$]$};
		\node[orange] at (-4.67,0) {$($};
		\node[orange] at (a2) {$]$};
		\node[brown] at (-3.27,0) {$($};
		\node[brown] at (a3) {$)$};
		\node[magenta] at (-1.87,0) {$[$};
		\node[magenta] at (at0) {$)$};
		\node[blue] at (a) {$[$};
		\node[blue] at (b) {$]$};
		\node[magenta] at (0.523,0) {$($};
		\node[magenta] at (1.87,0) {$]$};
		\node[brown] at (b3) {$($};
		\node[brown] at (3.27,0) {$)$};
		\node[orange] at (b2) {$[$};
		\node[orange] at (bt0) {$)$};
		\node[red] at (br) {$[$};
		\node[red] at (Rn) {$]$};
		\node[below, yshift=-2.5pt] at (Ln)  {$-\frac{\lambda}{2}$};
		\node[below, yshift=-5pt] at (ar)  {$a-\rho$};
		\node[below, yshift=-2.5pt] at (a2)  {$a-\tfrac{2\rho}{3}$};
		\node[below, yshift=-4pt] at (a3)  {$a-\tfrac{\rho}{3}$};
		\node[above] at (at0) {$a-t_0$};
		\node[below, yshift=-6.5pt] at (a)   {$a$};
		\node[below, yshift=-4pt] at (b)   {$b$};
		\node[above, yshift=2.5pt] at (bt0) {$b+\rho-t_0$};
		\node[below, yshift=-4pt] at (b3)  {$b+\tfrac{\rho}{3}$};
		\node[below, yshift=-2.5pt] at (b2)  {$b+\tfrac{2\rho}{3}$};
		\node[below, yshift=-4pt] at (br)  {$b+\rho$};
		\node[below, yshift=-2.5pt] at (Rn)  {$\frac{\lambda}{2}$};
		\node[left, yshift=-2.5pt]  at (L) {$-\infty$};
		\node[right, yshift=-2.5pt] at (R) {$\infty$};
		
		\draw[decorate, blue, decoration={brace,mirror,amplitude=9pt}]
		(-0.5,-0.75) -- (0.5,-0.75)
		node[midway, blue, below=9pt] {$A$};
		
		\draw[decorate, magenta, decoration={brace,mirror,amplitude=9pt}]
		(-1.9,-0.75) -- (-1,-0.75)
		node[midway, magenta, below=9pt] {$A_{le}$};
		
		\draw[decorate, magenta, decoration={brace,mirror,amplitude=9pt}]
		(0.5,-0.75) -- (1.9,-0.75)
		node[midway, magenta, below=9pt] {$A_{re}$};
		
		\draw[decorate, magenta, decoration={brace,mirror,amplitude=9pt}]
		(-1.9,-1.55) -- (1.9,-1.55)
		node[midway, magenta, below=9pt] {$A_e$};
		
		\draw[decorate, red, decoration={brace,mirror,amplitude=9pt}](-6,-0.75) -- (-4.7,-0.75)
		node[midway,below=9pt] {$B_1$};
		\draw[decorate, red, decoration={brace,mirror,amplitude=9pt}](4.7,-0.75) -- (6,-0.75)
		node[midway,below=9pt] {$B_2$};
		\draw[decorate, orange, decoration={brace,mirror,amplitude=9pt}](-6.5,-1.55) -- (-3.3,-1.55)
		node[midway,below=9pt] {$B_{1,e}$};
		\draw[decorate, orange, decoration={brace,mirror,amplitude=9pt}](3.3,-1.55) -- (6.5,-1.55)
		node[midway,below=9pt] {$B_{2,e}$};
		\draw[decorate, orange, decoration={brace,mirror,amplitude=9pt}](3.3,-0.75) -- (4.2,-0.75)
		node[midway,below=9pt] {$B_{2,le}$};
		\draw[decorate, orange, decoration={brace,mirror,amplitude=9pt}](-4.7,-0.75) -- (-3.3,-0.75)
		node[midway,below=9pt] {$B_{1,re}$};

		\draw[decorate, brown, decoration={brace,mirror,amplitude=9pt}]
		(-3.3,-0.75) -- (-1.9,-0.75)
		node[midway, brown, below=9pt] {$D_1$};
		
		\draw[decorate, brown, decoration={brace,mirror,amplitude=9pt}]
		(1.9,-0.75) -- (3.3,-0.75)
		node[midway, brown, below=8pt] {$D_2$};
	\end{tikzpicture}
	\caption{Interval partition of $\Gamma_{\lambda}$ for polynomial mixing}
	\label{figurepm1}
\end{figure}

The renewal points in the stationary renewal point process ${\cal P}$ on a given interval $I$ are simply the points of the process ${\cal P}$ that fall within the interval $I$, expressed  as ${\cal P}|_I$. From Lemma~\ref{example2.lma2}
(with $t$ and $C$ replaced by $\rho/3$ and $1$, respectively)
and extending the probability space if necessary, we can construct, on the same probability space, a copy $\{{\cal P}'|_{B_{1,e}}, {\cal P}'|_{B_{2,e}}, {\cal P}'|_{A_{e}}\}$ of $\{{\cal P}|_{B_{1,e}}, {\cal P}|_{B_{2,e}}, {\cal P}|_{A_{e}}\}$, such that ${\cal P}'|_{B_{1,e}}, {\cal P}'|_{B_{2,e}}, {\cal P}'|_{A_{e}}$ are independent and 
$$E_1:=\left\{({\cal P}|_{B_{1,e}}, {\cal P}|_{B_{2,e}}, {\cal P}|_{A_{e}})= ({\cal P}'|_{B_{1,e}}, {\cal P}'|_{B_{2,e}}, {\cal P}'|_{A_{e}})\right\}$$
satisfies
\begin{equation*}\ignore{\label{thm-ex2-01}}
	\mathbb{P}(E_1^c)\le o\(\rho^{-(\alpha-1)}\).
\end{equation*}
To fill the gap of ${\cal P}'$ on $D$ and to complete the definition of  ${\cal P}'$,  we define ${\cal P}'|_D={\cal P}|_D.$
It is important to note that the inter-arrival times of ${\cal P}'$ are generally no longer independent and identically distributed.

Step (iii). Breaking the direct dependence of the score function outside the neighborhood and controlling the dependence between marks within the neighborhoods.

To break the direct dependence in the marks, we require a sufficient number of renewal points in each of the three  sub-intervals between $B_1$ and $A$ and in each of the three sub-intervals between $A$ and $B_2$; hence we set 
$$ E_2=\left\{\min\{{\cal P}\(B_{1,re}\),{\cal P}\(B_{2,le}\),{\cal P}\(A_{le}\), {\cal P}\(A_{re}\),{\cal P}\(D_1\),{\cal P}\(D_2\)\}\ge C_1 \rho \right\},$$
where $C_1 :=(4\mathbb{E}\tau_1)^{-1}\in \(0,(2\mathbb{E}\tau_1)^{-1}\)$ as in Lemma~\ref{example2.lma1}, and let $E'_2$ be the corresponding event for ${\cal P}'$. 
Lemma~\ref{example2.lma1} ensures that
\begin{equation*}\ignore{\label{thm-ex2-02}}
	\mathbb{P}\(E_2^c\)=\mathbb{P}\((E_2')^c\)= O\(\rho^{-\alpha/2}\).
\end{equation*}
Without loss of generality, we assume that $\rho\ge (k_0+j_0)/C_1$ (in addition to  $\rho \geq 3t_0$) for $k_0$ and $j_0$ specified in \eqref{timeseriesdefa1} and \eqref{timeseriesdefa2}. This condition ensures that we can rule out any direct dependence of the score functions in $A$, $B_1$, and $B_2$ on the marks $U$ outside their respective neighborhoods $A_e$, $B_{1,e}$ and $B_{2,e}$ on event $E_2$. The requirement on the number of points in $D_1$ and $D_2$ also guarantees that the dependence of the marks within these neighborhoods is controllable on event $E_2$, which will be detailed later in \eqref{phimixing-result}.

Step (iv).  Constructing the marked point process $\left.\widehat{{\cal P}'}\right|_{\Gamma_\lambda}$ given  $\left.\widehat{{\cal P}}\right|_{\Gamma_\lambda}$.

To equate $\widehat{\mathcal{P}}|_{\Gamma_\lambda}$ with its marked counterpart $\left.\widehat{\mathcal{P}'}\right|_{\Gamma_\lambda}$ with high probability, we reuse the marks associated with $\mathcal{P}|_{\Gamma_\lambda}$, say $\{U_l,U_{l+1},\dots\}$, and assign them in the same order to the points of $\left.\mathcal{P}'\right|_{\Gamma_\lambda}$ (for example, $U_l$ goes to the smallest point of $\mathcal{P}'|_{\Gamma_\lambda}$, $U_{l+1}$ goes to the next smallest, and so on). The reason for  constructing $\left.\widehat{\mathcal{P}'}\right|_{\Gamma_\lambda}$ is that the marks on the intervals $B_{1,e}\cup B_{2,e}\cup A_e$ are determined by the configuration of renewals beyond these intervals, as stated in \eqref{timeseriesdefa1}. The advantage of $\widehat{\mathcal{P}'}|_{\Gamma_\lambda}$ is that it permits us to use the independent restrictions of $\mathcal{P}'|_{\Gamma_\lambda}$ to eliminate the dependence of the marks. From this construction, we observe that the corresponding $S'_A$ and $S'_B$ of $\widehat{{\cal P}'}$ satisfy the following equalities:
$$S_A\bone_{E_1\cap E_2}=S'_A\bone_{E_1\cap E'_2},\ \ \ 
S_B\bone_{E_1\cap E_2}=S'_B\bone_{E_1\cap E'_2}.$$

Step (v). Showing that mixing of score sums $S_A'$, $S_B'$ implies mixing of score sums for the function-set class corresponding to  $j=1$ in Definition~\ref{WSM}.

We aim to replace $S_A$ and $S_B$ with $S'_A$ and $S'_B$ on $E_3:=E_1\cap E_2$ via  the following inequality:
\begin{align}
	&\left|\mathbb{E}(S_AS_B)-\mathbb{E}(S_A)\mathbb{E}(S_B)\right|\nonumber\\ 
	&=\left|\mathbb{E}(S_A(\bone_{E_3}+\bone_{E_3^c})S_B)-\mathbb{E}(S_A(\bone_{E_3}+\bone_{E_3^c}))\mathbb{E}(S_B)\right|\nonumber\\
	&\le\left|\mathbb{E}(S_A\bone_{E_3}S_B)-\mathbb{E}(S_A\bone_{E_3})\mathbb{E}(S_B)\right|
	+\left|\mathbb{E}(S_A\bone_{E_3^c}S_B)\right|+\left|\mathbb{E}(S_A\bone_{E_3^c})\mathbb{E}(S_B)\right|.\label{example2.mainthm.1}
\end{align}
Additionally, by Lemma~\ref{moments-XB}, since $p_1$ in assumption \eqref{momground} can be arbitrarily large, both $S_A$ and $S_B$ have finite $L^{p_0}$ norms for $p_0$ in the statement of Theorem~\ref{example2}, 
which are bounded above by $C_2\text{Vol}(A)$ and $C_2\text{Vol}(B)$, respectively, for some positive constant $C_2$.

H\"{o}lder's inequality yields
\begin{align}
	\left|\mathbb{E}(S_A\bone_{E_3^c}S_B)\right|
	&\le \|S_A\|_{p_0}\|S_B\|_{p_0}\mathbb{P}(E_3^c)^{(p_0-2)/p_0}\nonumber\\
	& = O\left(\text{Vol}(A)\text{Vol}(B)\rho^{-\alpha (p_0-2)/(2p_0)}\right),\label{example2.mainthm.2}
\end{align}
and
\begin{align}
	\left|\mathbb{E}(S_A\bone_{E_3^c})\mathbb{E}(S_B)\right|
	&\le \|S_A\|_{p_0}\|S_B\|_{p_0}\mathbb{P}(E_3^c)^{(p_0-2)/p_0}\nonumber\\
	& =  O\left(\text{Vol}(A)\text{Vol}(B)\rho^{-\alpha (p_0-2)/(2p_0)}\right).\label{example2.mainthm.3}
\end{align}

Combining \eqref{example2.mainthm.1}, \eqref{example2.mainthm.2} and \eqref{example2.mainthm.3}, we obtain
\begin{align}
	&\left|\mathbb{E}(S_AS_B)-\mathbb{E}(S_A)\mathbb{E}(S_B)\right|\nonumber\\ 
	\le&\left|\mathbb{E}(S_A\bone_{E_3}S_B)-\mathbb{E}(S_A\bone_{E_3})\mathbb{E}(S_B)\right|+O\left(\text{Vol}(A)\text{Vol}(B)\rho^{-\alpha (p_0-2)/(2p_0)}\right).\label{example2.mainthm.4}
\end{align}

Each term involving $E_3^c$ incurs a cost of at most $O\left(\text{Vol}(A)\text{Vol}(B)\rho^{-\alpha (p_0-2)/(2p_0)}\right)$.  Therefore, we focus on the leading term that does not involve $E_3^c$, and instead add the error due to the manipulation required to transition from one form to the other. We use the notation $\beta_1\leftslice\beta_2$ to denote that $\beta_1\le \beta_2+O\left(\text{Vol}(A)\text{Vol}(B)\rho^{-\alpha (p_0-2)/(2p_0)}\right)$. With this, we can obtain from \eqref{example2.mainthm.4} the following expression: 
\begin{align}
	&\left|\mathbb{E}(S_AS_B)-\mathbb{E}(S_A)\mathbb{E}(S_B)\right|\nonumber\\ 
	\leftslice&\left|\mathbb{E}(S_AS_B\bone_{E_3})-\mathbb{E}(S_A\bone_{E_3})\mathbb{E}(S_B\bone_{E_3})\right|\nonumber\\
	=&\left|\mathbb{E}(S'_AS'_B\bone_{E_3'})-\mathbb{E}(S_A\bone_{E_3})\mathbb{E}(S_B\bone_{E_3})\right|\nonumber\\
		\leftslice&\left|\mathbb{E}(S'_AS'_B)-\mathbb{E}(S'_A)\mathbb{E}(S'_B)\right|, \label{example2.mainthm.5}
\end{align}
where $E'_3:=E_1\cap E'_2=E_3$. To see this, note that by construction,  the coupled processes ${\cal P}$ and ${\cal P}'$ agree on $D_1$ and $D_2$ and  on the event $E_1$ they also agree on the relevant neighborhoods of $A$ and $B$ and hence the events $E_2$ and $E_2'$ coincide on $E_1$.

Step (vi). Showing the mixing of score sums $S_A'$, $S_B'$ using classical $\phi$-mixing bounds.  

When we lift the marks to higher-dimensional space by defining 
$$ {\bf U}_i=(U_{i},U_{i+1}, \dots, U_{i+k_0-1})\in \M^{k_0},$$ 
then  $\{{\bf U}_i\}_{i\in\mathbb{Z}}$ forms a strictly stationary $k_0$-dimensional Markov process. For each $i$, the mark $U_i$ is the first component of the lifted vector ${\bf U}_i$. Let $\mathscr{A}'$ denote the $\sigma$-algebra generated by all the marked points in $\widehat{{\cal P}'}$ with the lifted marks ${\bf U}_i$, where at least one component is used in $S'_A$. From the construction, we have $S'_A\in \mathscr{A}'\subset \sigma\(\widehat{{\cal P}'}_{A_e}\)$, $S'_B\in\sigma\(\widehat{{\cal P}'}_{B_{1,e}\cup B_{2,e}}\)$. From \cite[Theorem~3.3~(2)]{B05}, which states that when \eqref{phimixing} holds, the $\phi$-mixing coefficient converges to $0$ exponentially fast as $n_0$ goes to infinity, along with the fact that $\{{\cal P}'_{B_{1,e}}, {\cal P}'_{B_{2,e}}, {\cal P}'_{A_{e}}\}$ are independent, the $\phi$-mixing coefficient satisfies
\begin{equation}
	\phi(\mathscr{A}', \sigma(S'_B)) = O\(\mathrm{e}^{-C_3\rho}\),\label{phimixing-result}
\end{equation}
for some positive constant $C_3$.

Moreover, by \citet[Lemma 2.1 and its Corollary]{D68}, for any two $\sigma$-fields $\mathscr{A}$ and $\mathscr{B}$ with the $\alpha$-coefficient $\alpha(\mathscr{A}, \mathscr{B})=\sup_{E\in \mathscr{A}, E'\in \mathscr{B}}|\mathbb{P}(E\cap E')-\mathbb{P}(E)\mathbb{P}(E')|$, we have for any $p,q\in (1,\infty]$ such that $1/p+1/q<1$ and random variables $X\in L^p(\mathscr{A})$, $Y\in L^q(\mathscr{B})$,
\begin{equation*}
	|\cov(X,Y)|\le 12 \alpha(\mathscr{A}, \mathscr{B})^{1-1/p-1/q}\|X\|_p\|Y\|_q.
\end{equation*} Together with the fact that $\alpha(\mathscr{A}, \mathscr{B})\le \phi(\mathscr{A}, \mathscr{B})$ from the definition, this gives \begin{equation}\label{covphi}
|\cov(X,Y)|\le 12 \phi(\mathscr{A}, \mathscr{B})^{1-1/p-1/q}\|X\|_p\|Y\|_q.
\end{equation}
Although a bound with a sharper constant is available in \citet[(1.3)]{R93}, we do not pursue this refinement here, since our focus is on the rate of the approximation error.
Since $S_A'$ is $\mathcal A'$-measurable, applying the covariance
inequality \eqref{covphi} with $p=q=3$ together with \eqref{phimixing-result} and the moment bounds $\|S_A'\|_{3}=O(\Vol(A))$
and $\|S_B'\|_{3}=O(\Vol(B))$ gives
\begin{align}
	\left|\mathbb{E}(S_A'S_B')-\mathbb{E}(S_A')\mathbb{E}(S_B')\right|&=|\cov(S_A',S_B')| \le 12\phi(\mathcal A',\sigma(S_B'))^{1/3}
	\|S_A'\|_{3}\|S_B'\|_{3}\nonumber\\
	&=O\(\Vol(A)\Vol(B)\mathrm{e}^{-C_3\rho/3}\).\label{example2.mainthm.6}
\end{align}

Combining \eqref{example2.mainthm.5} and \eqref{example2.mainthm.6}, we obtain
\begin{equation}\label{example2.mainthm.7}
	\left|\mathbb{E}(S_AS_B)-\mathbb{E}(S_A)\mathbb{E}(S_B)\right|= O\left(\text{Vol}(A)\text{Vol}(B)\rho^{-\alpha (p_0-2)/(2p_0)}\right).
\end{equation}

To show mixing of score sums for the function-set class corresponding to  $j=2$ in Definition~\ref{WSM}, we use a similar approach. Abusing notation, let  $\beta_1\leftslice\beta_2$ signify  that $\beta_1\le \beta_2+O\left((\text{Vol}(A))^2\rho^{-\alpha (p_0-2)/(2p_0)}\right)$. Specifically, we have:
\begin{align}
	&\left|\mathbb{E}\(S_A^2g(S_B)\)-\mathbb{E}\(S_A^2\)\mathbb{E}\(g(S_B)\)\right|\nonumber\\
	&\leftslice\left|\mathbb{E}\(S_A^2g(S_B)\bone_{E_3}\)-\mathbb{E}\(S_A^2\bone_{E_3}\)\mathbb{E}\(g(S_B)\bone_{E_3}\)\right|\nonumber\\
	&=\left|\mathbb{E}\((S'_A)^2g(S'_B)\bone_{E_3'}\)-\mathbb{E}\((S'_A)^2\bone_{E_3'}\)\mathbb{E}\(g(S'_B)\bone_{E_3'}\)\right|\nonumber\\
	&\leftslice\left|\mathbb{E}\((S'_A)^2g(S'_B)\)-\mathbb{E}\((S'_A)^2\)\mathbb{E}\(g(S'_B)\)\right|.
	\label{example2.mainthm.8}
\end{align}

Applying the covariance inequality \eqref{covphi} with $p=3/2$ and $q=\infty$ gives 
\begin{align*}
	\left|\mathbb{E}((S'_A)^2g(S'_B))-\mathbb{E}((S'_A)^2)\mathbb{E}(g(S'_B))\right|=& \cov\((S'_A)^2,g(S'_B)\)\nonumber\\
	\le& 12\phi(\mathcal A',\sigma(S_B'))^{1/3}
	\|(S_A')^2\|_{3/2}\|g(S_B')\|_{\infty}\nonumber\\
	=&O\(\text{Vol}(A)^2\mathrm{e}^{-C_3\rho/3}\),
	\end{align*}
which, together with \eqref{example2.mainthm.8}, guarantees
\begin{equation}\label{example2.mainthm.9}
	\left|\mathbb{E}\(S_A^2g(S_B)\)-\mathbb{E}\(S_A^2\)\mathbb{E}\(g(S_B)\)\right|
	= O\left(\text{Vol}(A)^2(\text{Vol}(B)\vee 1)\rho^{-\alpha (p_0-2)/(2p_0)}\right).
\end{equation}

Combining \eqref{example2.mainthm.7} and \eqref{example2.mainthm.9} implies that polynomial mixing of score sums \eqref{mixing-score'} is satisfied
by taking $\g_0=2$, letting $\g_1$ be the constant implicit in the `$\mathop{O}$' bound, and setting $\g_2=\alpha(p_0-2)/(2p_0)$, noting that $\Vol(A)\le \diam(A)$ for all sets $A$ when $d=1$. Recall $\g_3 = 3t_0$.  Therefore, condition \eqref{relationship} with $d=1$ becomes 
$$ 2\g_2=\frac{\alpha (p_0-2)}{p_0}>\frac{(36-6\nu)\vee28}{3\nu-2},$$
as assumed.  Theorem~\ref{example2} follows by applying Theorem~\ref{mainthm1}~(b).
 \qed

\subsection{Continuum percolation models} \label{ex3.3}

\ignore{Given a graph $G$ with vertex set $V(G)$  the {\it {graph} distance} $d_G(x, y)$ between $x, y \in V(G)$ is  the minimum number of edges connecting $x$ and $y$. By convention, $d_G(x, y) = \infty$ if there is no path between $x$ and $y$.  The graph ball centered at vertex $v_0 \in V(G)$ and  with radius $r \in \mathbb{N}$ is
\begin{equation*}\ignore{\label{graphball}}
	B^G(v_0,r):=G_{\{v;~d_G(v, v_0)\le r\}},
\end{equation*}
where $G_U$ denotes the subgraph $G$ induced by the vertex set $U$.}

For a locally finite point set $V$ in $\mathbb{R}^d$, $d \ge 2$,  and $r\ge 1$,  by ${\cal G}(V, r)$ we mean  the graph with vertex set $V$ and edge set $E := \{(u, v); \|u - v\| \le r,~u, v \in V^{2, \neq}\} $, where $V^{2, \neq} := \{(u, v); u \neq v,~u, v \in V\}$. When $V$ is finite, we define $C(V, i) := C(V, i, 1)$ and $L(V, i) := |C(V, i)|$ as the $i$-th largest component of the graph ${\cal G}(V, 1)$ and its cardinality, respectively. By convention, we set $C(V, i) = \emptyset$ and $L(V, i) = 0$ if the number of components in ${\cal G}(V, 1)$ is less than $i$.  Components with the same size are ordered by the lexicographical order of the smallest vertices in each component.

Let $\mathcal{P}_{\mu}$ be a homogeneous Poisson point process on $\mathbb{R}^d$ with intensity $\mu$.
We have ${\cal G}(\mathcal{P}_{\mu}, r)=r^{1/d}{\cal G}(r^{-1/d}\mathcal{P}_{\mu}, 1)$, where multiplying a graph by $r^{1/d}$ means scaling all vertex coordinates by $r^{1/d}$ while keeping the edge structure unchanged,  and equality means that both sides have the same vertices and the same edges.
Thus, without loss of generality, we focus on ${\cal G}(V, r)$ with $r = 1$.

For a point process $\P$ on $\real^d$,  ${\cal G}(\P, r)$ is
the  {\it random geometric graph} generated by  $\P$ with radius $r$,  a model for continuum percolation.  
We consider ${\cal G}({\cal P} \cap \Gamma_{\lambda}, 1)$  for a point process ${\cal P}$ and a positive constant $\lambda$. 
The model can be cast in terms of urban agglomeration, and therefore we may refer to each component 
$C({\cal P} \cap \Gamma_{\lambda}, i)$  as a {\it community}.

For a fixed positive integer $l_0 \in \N$, the component/community $C({\cal P} \cap \Gamma_{\lambda}, i)$ is called {\it $l_0$-isolated} if its cardinality satisfies $L({\cal P} \cap \Gamma_{\lambda}, i) \leq l_0$. We let $I_0$ be the indices such that $i \in I_0$ implies that $C({\cal P} \cap \Gamma_{\lambda}, i)$ is  $l_0$-isolated. Each component  $C({\cal P} \cap \Gamma_{\lambda}, i), i \geq 1,$ has  a range of influence given by $f(L({\cal P} \cap \Gamma_{\lambda}, i))$, where $f: \mathbb{N} \to [0, \infty)$ is a non-decreasing function. 

An $l_0$-isolated community $C({\cal P} \cap \Gamma_{\lambda}, i), i \in I_0$,  is said to be {\it remote} 
if its distance to every non-isolated community $C({\cal P} \cap \Gamma_{\lambda}, j)$, if such communities exist, is greater than $f(L({\cal P} \cap \Gamma_{\lambda}, j))$.    We let $I_0^{rem} \subseteq I_0$ be the indices of the remote components. 

We service each remote community by connecting it to its nearest non-isolated community $C({\cal P} \cap \Gamma_{\lambda}, j), j \notin I_0$.   This ensures that all  communities are serviced, as we assume that non-remote communities are already serviced.

This construction gives rise to the graph ${\cal G}^{iso}({\cal P}, \lambda)$ thus obtained by connecting each remote community $C({\cal P} \cap \Gamma_{\lambda}, i),\ i \in I_0,$ to the collection of  non-isolated communities ${\cal N}^{ni}({\cal P} \cap \Gamma_{\lambda}) := \{ C({\cal P} \cap \Gamma_{\lambda}, i) \}_{i \notin I_0}$  by adding an edge 
between $C({\cal P} \cap \Gamma_{\lambda}, i),\ i \in I_0,$  and the nearest point in ${\cal N}^{ni}({\cal P} \cap \Gamma_{\lambda})$.  For a remote community $C({\cal P} \cap \Gamma_{\lambda}, i)$, if there exist multiple pairs of nearest points between it and ${\cal N}^{ni}({\cal P} \cap \Gamma_{\lambda})$ at the same distance, one pair is chosen uniformly at random. The existence of isolated communities depends on the configuration of all points. The long-range dependence structure precludes letting the short-range estimator $\hat{\xi}^{[r]}$ be a restricted score function, which is often the natural choice for tackling problems such as this one, and which is customarily  employed to define stabilization, cf. \cite{CX23} and references therein. Instead we will 
use a different choice for the short-range estimator $\hat{\xi}^{[r]}$. As will be shown in the proof, our result does not depend on the particular method used to connect points in remote communities to those in non-isolated communities.

Given a graph $G$ with vertex set $V(G)$,  the {\it graph distance} $d_G(x, y)$ between $x, y \in V(G)$ is  the minimum number of edges connecting $x$ and $y$. By convention, $d_G(x, y) = \infty$ if there is no path between $x$ and $y$.  The graph ball centered at vertex $v_0 \in V(G)$ and  with radius $r \in \mathbb{N}$ is
\begin{equation*}\ignore{\label{graphball}}
	B^G(v_0,r):=G_{\{v;~d_G(v, v_0)\le r\}},
\end{equation*}
where $G_U$ denotes the subgraph $G$ induced by the vertex set $U$. We now consider the case where the score function 
$$\xi(x, {\cal P}, \Gamma_{\lambda}) = h(B^{{\cal G}^{iso}({\cal P}, \lambda)}(x, r_0))$$
is given by a function $h$ defined on the graph ball in 
${\cal G}^{iso}({\cal P}, \lambda)$
of graph radius~$r_0$. This type of model can be used to represent transportation costs or urban influence and is particularly relevant in the context of spatial growth and development. While traffic congestion and the reliability of transportation networks are key factors, another critical aspect of development is ensuring logistical connectivity between remote communities and well-serviced urban centres.

Define
$$
H_{\lambda} :=\sum_{x\in{\scrP}\cap\Gamma_{\lambda}}\xi(x, {\cal P}, \Gamma_{\lambda}).
$$
We note that the score function $\xi$ is typically not stabilizing, as its evaluation may require information about the entire configuration ${\cal P} \cap \Gamma_{\lambda}$ in order to determine the existence of remote communities. This statistic can be used to analyze commonly studied problems in the literature, including counting subgraphs that are isomorphic to a given graph or counting components of size $r_0$ for some $r_0\in \mathbb{N}$. With a simple adjustment, the same idea can also be used to analyze these problems for random geometric graphs in the continuum percolation model.

Let ${\cal P}_{\mu}^{\bm{0}}$ denote ${\cal P}_{\mu}\cup \{\bm{0}\}$, where $\bm{0}$ is the origin in $\mathbb{R}^d$. 
The probability $p_{\infty}(\mu)$ represents the probability that the origin $\bm{0}$ is contained in an infinite component of ${\cal G}\({\cal P}_{\mu}^{\bm{0}},1\)$, 
and is the {\it percolation probability} (or {\it percolation function}). The threshold intensity for the phase transition 
giving the existence of an infinite component, known as the {\it critical intensity} $\mu_c$, is defined by 
$$
\mu_c:=\inf\{\mu\in(0,\infty);~p_{\infty}(\mu)>0\}.
$$ 
A fundamental result in continuum percolation theory states that $\mu_c$ is finite when $d\ge 2$; see \citet[Chapter 3]{MY96}.   
Recall that a point process ${\cal Q}_1$ stochastically dominates ${\cal Q}_2$ on the same space if there exists a coupling of them, $(\bar{{\cal Q}}_1, \bar{{\cal Q}}_2)$, such that $\bar{{\cal Q}}_1 \overset{d}{=}{\cal Q}_1$, $\bar{{\cal Q}}_2 \overset{d}{=}{\cal Q}_2$, and $\bar{{\cal Q}}_2 \subset \bar{{\cal Q}}_1$ almost surely. 

With this setup, we have the following theorem.

\begin{thm}\label{example3} Let  $d \ge 2$. We assume:
	\begin{description}  
		\item[$\cdot$] the point process ${\cal P}$ satisfies the moment condition \eqref{momground} for some positive constant $p_1$, 
		\item[$\cdot$] ${\cal P}$ exhibits exponential decay of dependence as in (\ref{mixing}), and stochastically dominates ${\cal P}_{\mu_0}$ for some $\mu_0 > \mu_c$,
		\item[$\cdot$]  the function $f$ that determines the remote communities satisfies $f(s) > \sqrt{d}s$ for all $s \in \mathbb{N}$, 
		\item[$\cdot$]  the score function $\xi$ is a function of the graph ball 
		in ${\cal G}^{iso}({\cal P}, \lambda)$
		with graph radius  $r_0 \in \mathbb{N}$ and satisfies \eqref{momscore} for some positive  constant $p_2$ such that $p_1p_2/(p_1+p_2-1)\ge 3$,
		\item[$\cdot$] $\Var H_{\lambda} = \Omega(\lambda^\nu)$ for some $\nu > 2/3$.
	\end{description}
	Then 
	\begin{equation*}	
		d_W\( \frac{H_{\lambda} - \E H_{\lambda}} {\sqrt{\Var H_{\lambda}}}, Z\) = O\((\log \lambda)^{3d} \lambda^{-1.5\nu+1}\),
	\end{equation*}
	where $Z\sim N(0,1)$.
\end{thm}

Since the conditions of Theorem~\ref{example3} are insufficient to imply that the Palm process of ${\cal P}$ stochastically dominates a Poisson point process with a supercritical intensity, and since there is no existing result establishing that every Palm process of ${\cal P}$ possesses a giant component with probability converging to $1$ as $\lambda$ grows, we cannot establish \eqref{localization} directly.  Because \eqref{localization} is formulated in terms of the Palm distribution, Corollary~\ref{maincorTV} therefore cannot be applied directly in this case.

\noindent{\it Proof of Theorem~\ref{example3}.~} The existence of remote communities depends on the entire configuration of ${\cal P}\cap \Gamma_{\lambda}$. It  follows that $\xi$ is not stabilizing, hence the usual stabilization methods based on restricted score functions do not work. However, we can construct an estimate $\hat{\xi}^{[r]}$ of $\xi$ using the local  information about the continuum percolation and then deduce Theorem~\ref{example3} from Theorem~\ref{mainthm1}.

The moment assumptions \eqref{momground} and \eqref{momscore} are assumed. We thus only need to show exponential mixing of the score sums in~\eqref{WSM}. 
For this purpose, we follow the same steps as in the proof of 
Corollary~\ref{maincorTV}~(a) to construct the corresponding point processes $\Psi_i$, for $1\le i \le 6$, and to define the event $E_1$, as well as the random variables $S_A'$, $S_B'$, $\bar{S}_A^{[r]}$ and $\bar{S}_B^{[r]}$, in the same manner as in the proof of Corollary~\ref{maincorTV} with 
$$
\hat{\xi}^{[r]}(x, {\cal P}, \Gamma_{\lambda}):=\hat{\xi}^{[r_0]}(x, {\cal P}, \Gamma_{\lambda}):=h(B^{{\cal G}({\cal P}\cap \Gamma_{\lambda},1)}(x,r_0)),$$ for $A, B\in {\cal B}(\mathbb{R}^{d} )$ and $r=d(A,B)/3\ge r_0$. 
We will, however, need to modify the definitions of the events $E_2$ and $E_3$ used in the proof of Corollary~\ref{maincorTV}. To this end, we adopt the terminology from~\citet{PP96}, where a connected component is said to be {\it crossing} if the union of balls of radius~$1$ centered at the vertices of the component intersects all $2d$ faces of the boundary of $\Gamma_{\lambda}$.
We then define $E_2$ (resp. $E_3$) as the event that 
$C(\Psi_5 \cap \Gamma_{\lambda}, 1)$ (resp. $C(\Psi_6 \cap \Gamma_{\lambda}, 1)$) is {\it not} crossing. Since ${\cal P}$ stochastically dominates ${\cal P}_{\mu_0}$, 
it follows from~\cite[Proposition~2]{PP96} with $\phi_{\lambda}=0.5\lambda$ that there is a constant $C_6:=C_6(\mu_0)$  such that 
$$
\mathbb{P}(E_2)\vee \mathbb{P}(E_3)\le \exp(-C_6\lambda).
$$

Whenever $C(\Psi_5 \cap \Gamma_{\lambda}, 1)$ and $C(\Psi_6 \cap \Gamma_{\lambda}, 1)$ are crossing, there must be components of size at least $\lambda^{1/d}$ and this, together with the condition $f(s) \geq \sqrt{d}s$ ensures that on $E_2^c\cap E_3^c$
 there are no remote communities in $\Psi_5\cap \Gamma_{\lambda}$ and $\Psi_6\cap \Gamma_{\lambda}$, and consequently no edges are added between remote communities and isolated communities. Therefore, ${\cal G}^{iso}({\cal P}, \lambda)$ and ${\cal G}(\P \cap \Gamma_{\lambda}, 1)$ coincide and thus the graph ball of radius $r$ in the first graph is a subset of the Euclidean ball of radius $r$ in the second graph. Under the event $E=(E_1\cup E_2\cup E_3)^c$, 
when $d(A,B)\ge 3r_0$, we have $S_A=S_A'=\bar{S}_A^{[r]}$, $S_B=S_B'=\bar{S}_B^{[r]}$. Hence, exponential mixing of the score sums in~\eqref{WSM} follows from the same argument for \eqref{maincor1.4.2} in Step~(iv) of the proof of Corollary~\ref{maincorTV}. 
The desired statement then follows directly by applying Theorem~\ref{mainthm1} (a). \qed

\subsection{Interacting diffusions on spatial random graphs}\label{ex.srg}

In this section we obtain quantitative CLTs for statistics of interacting diffusion models on spatial random graphs driven by EDD input.  The results are quantitative counterparts of the qualitative CLTs established in \cite[Section~9]{BYY25} 
and complement the Berry--Esseen rates of normal approximation for such statistics under Poisson input in \cite{TY}.

Before introducing the specific model, we  describe the underlying  {\it spatial random graphs}.  For ease of reading, this description is adapted from \cite{BYY25}, where a more complete treatment is provided and where, for example, the relevant measurability issues are addressed.

Given a locally finite point set ${\cal X}\subset\Gamma_{\lambda}$, let ${\cal G}({\cal X},\sim)=:{\cal G}({\cal X})$ denote a graph with vertex set ${\cal X}$, where edges are determined by a deterministic rule $\sim$ depending on the entire configuration. That is, ${\cal X}\mapsto {\cal G}({\cal X},\sim)$ is a measurable mapping from the space ${\cal N}_{\Gamma_{\lambda}}$ of locally finite subsets of $\Gamma_{\lambda}$ to the space of finite labelled graphs on $\Gamma_{\lambda}$, equipped with the discrete $\sigma$-algebra.

A function $S_{\lambda}:\Gamma_{\lambda}\times{\cal N}_{\Gamma_{\lambda}}\to   [0,\infty)$,
$\lambda\in(0, \infty]$, 
with the convention $\Gamma_{\infty}=\mathbb{R}^d$, is called an {\it interaction range} with respect to the graph operator ${\cal G}(\cdot,\sim)$ if it satisfies:
\begin{itemize}
	\item[(i)] for all ${\cal X}\in{\cal N}_{\Gamma_{\lambda}}$ and $x\in{\cal X}$, the neighborhood of $x$ is determined by ${\cal X}\cap B(x,S_{\lambda}(x,{\cal X}))$, i.e.,
	$y\sim x~\text{in}~{\cal G}({\cal X},\sim)\ \text{iff}\ 
	y\sim x~\text{in}~{\cal G}({\cal X}\cap B(x,S_{\lambda}(x,{\cal X})),\sim)$,
	\item[(ii)] the ball $B(x,S_{\lambda}(x,{\cal X}))$ is a stopping set, in the sense that  
	$$ S_{\lambda}(x,{\cal X})=S_{\lambda}\(x,[{\cal X}\cap B(x,S_{\lambda}(x,{\cal X}))]\cup{\cal Y}\)$$
	for any ${\cal Y}\in {\cal N}_{\mathbb{R}^d\setminus B(x,S_{\lambda}(x,{\cal X}))}$.
\end{itemize}

By construction, $y\sim x$ in ${\cal G}({\cal X},\sim)$ only if $|x-y|\le \min\{S_{\lambda}(x,{\cal X}),S_{\lambda}(y,{\cal X})\}$.

The following stabilization property for the graph ${\cal G}$ is inspired by a similar, though slightly stronger property   in \cite[Definition~8.1 and display~(9.2)]{BYY25}; see Remark~\ref{diff1stPalm2ndPalm} for the differences.

\begin{defi}\label{graphstabilizng}(stabilizing property of the graph operator ${\cal G}(\cdot,\sim)$ with respect to $\P$) 
	We say that ${\cal G}$ stabilizes with respect to the point process $\P$ if there exists a family of interaction ranges $(S_{\lambda})_{\lambda\in(0, \infty]}$ such that
	$$\sup_{\lambda\in(0, \infty]}\sup_{x\in\Gamma_{\lambda}}
	\mathbb{P}_{x}\(S_{\lambda}(x,\P\cap\Gamma_{\lambda})>r\)
	\le \varphi(r),\qquad r\in (0,\infty),
$$
	for some fast decreasing function $\varphi$, i.e., for  $\varphi$ satisfying
$$
 \limsup_{r\to\infty} r^m \varphi(r)=0,\qquad \forall m\ge 1.
$$
\end{defi}

We further assume that the graph operator is translation invariant, that is,
$$
x\sim y~\text{in}~{\cal G}({\cal X},\sim)
\ \text{iff}\ 
x+z\sim y+z~\text{in}~{\cal G}({\cal X}+z,\sim),
\quad \forall x,y\in{\cal X},~z\in\mathbb{R}^d.$$

\begin{defi}\label{admissible}(admissible graphs) 
The graph operator ${\cal G}$ is  {\it admissible} on ${\cal P}$ if  it is  translation invariant and stabilizes with respect to ${\cal P}$ as in Definition~\ref{graphstabilizng}.  
\end{defi}
	
As shown in \cite[Appendix A, Examples A.1--A.4]{BYY25}, if  $\P$ is a homogeneous Poisson point process or a stationary determinantal point process with an exponentially decaying kernel, then admissible graphs on  $\P$ include  the undirected $k$-nearest neighbors graph, the sphere of influence graph, and the Delaunay graph. Throughout this section, we work with such point processes $\P$ that are also EDD, together with the corresponding graph operators ${\cal G}(\cdot,\sim)$.

It is shown in \cite[Section~9]{BYY25} that statistics of interacting diffusions on spatial random graphs satisfy asymptotic normality, provided the underlying point process has fast decay of correlations.  The key to their results is  that the relevant statistics could be expressed as a sum of score functions which satisfy exponential stabilization in the $L^2$ distance.
That work  made no attempt to find rates of normal convergence. 
Here we indicate that our general results  yield rates of normal convergence when the input is a class of EDD point processes   and when the initial diffusion displacements are independent and uniformly bounded.
We  recall the set-up  and refer to \cite[Section~$9$]{BYY25} for full details.

Given a countable  $V \subset \R^d$,  let ${\cal G} := {\cal G}(V)$ be an admissible interaction graph on $V$.
Let $N_v$ be the neighborhood of $v \in V.$
Fix a time horizon $t_0 \in (1,\infty)$. 
 Consider a system of interacting $\R^{d'}$-valued diffusions
$M(v,t):=M^G(v,t), v \in V, \,  t \in [0, t_0],$ defined by
\begin{equation}\label{e:sys_ID_V}
	dM(v,t) = b(t,M(v,t),M(N_v,t))dt + \sigma(t,M(v,t),M(N_v,t))dZ_v(t), 
\end{equation}
where $Z_v(\cdot), v \in V,$ are i.i.d. standard Brownian motions in $\R^{d'}, d' \in \N$, and where $M(N_v,t) = \{M(v',t)\}_{v' \sim v}$, where $v' \sim v$ means that $v'$  is a neighbor of $v$.  Denote by $M(v):= M(v,0) \in \R^{d'}, v \in V,$ the initial states, assumed to be almost surely  uniformly bounded,  i.e.,  $\sup_{v\in V}|M(v,0)| \le L$ a.s. for a constant $L\in(0,\infty).$
These processes take values in the space $\mathbb{C}^{(0)}(\R^{d'})$ of continuous paths in $\R^{d'}$,
with the functions $b$ and $\sigma$ being the {\em drift} and {\em diffusion} coefficients, respectively. Here  $b$ is $\R^{d'}$-valued whereas the diffusion matrix $\sigma$ is $\R^{d'} \times \R^{d'}$-valued; see  \cite[Definition 9.2]{BYY25} for full details.
The particles interact directly only with their (finite) graph neighbors, in contrast to the mean field model, where particles interact with all other particles.

We further  assume  Lipschitz conditions on $b$ and $\sigma$, as spelled out in \citet{LRW23} and \cite[Definition 9.2]{BYY25}. When ${\cal G}$ is admissible, the Lipschitz assumption implies the existence of a strong solution for the  system of interacting diffusions; c.f.  \cite[Theorem 3.1]{LRW23}.

When the processes are defined on the entirety of the graph ${\cal G}$, we denote them simply as $M := M^{{\cal G}}$. The paths (trajectories) of the processes up to time $t$ are denoted by $M[v,t]$. Let $\hat{\P} := \{(x,M(x),Z_x)\}_{x \in \P}$ be a marked point process on $\R^d \times \R^{d'} \times \mathbb{C}^{(0)}(\R^{d'})$, with $M(x), x \in \P,$ denoting the independent initial states.  

To cast the statistics of interacting diffusions in the framework of our general results, we consider, {\em for fixed $t\in (0,\infty)$,} real-valued score functions of the trajectories
\begin{equation}\label{xisysID}
	\xi ((x,M(x),Z_x), \hat{\P}):=  h(M^{{\cal G}(\P)}[x,t])
\end{equation} 
where  $h : \mathbb{C}^{(0)}(\R^{d'})(t) \to \R$ is a Lipschitz($1$) function with respect to the sup-norm, $\|\cdot\|_{\infty}$, on $\mathbb{C}^{(0)}(\R^{d'})(t)$, the space of trajectories up to time $t$. The measurability of $\xi$ with respect to
$\hat{\P}|_{\Gamma_{\lambda}}$ is established in the proof of Theorem 7.8 of \cite{LRW23} and also in Section 9 of \cite{BYY25}. 
Such scores give  rise
to the diffusion statistics 
\begin{equation*}\ignore{\label{sysidfunctional}}
	H_{\lambda}:={H}_{\lambda}^{\xi} := \sum_{x \in \P \cap \Gamma_{\lambda}} \xi ((x,M(x),Z_x), \hat{\P}).
\end{equation*}
Finally, for $k \in \N$ denote by
${\cal G}_x^{k}(\mathcal{X})$ the subgraph of
${\cal G}(\mathcal{X})$ induced by the graph ball
$B^{{\cal G}(\mathcal{X})}(x,k)$ centered at $x\in \mathcal{X}$ with graph radius $k$. This gives rise to the corresponding interacting diffusion processes $\left\{M^{{\cal G}_x^{k}(\mathcal{X})}(y,t)\right\}$ for $t\in [0,t_0]$ on $y\in B^{{\cal G}(\mathcal{X})}(x,k)$, which have the same coefficients $(b,\sigma)$, are driven by the same Brownian motions, and share the same initial conditions $M^{{\cal G}_x^{k}(\mathcal{X})}(y,0)=M(y,0)$ for $y\in B^{{\cal G}(\mathcal{X})}(x,k)$. 
Let $\P_x$ denote the Palm process of $\P$ at $x$, that is, $\P_x\sim{\cal L}_x(\P)$.

\begin{cor} \label{t:sysID} (rates of normal convergence for statistics of interacting diffusions on EDD input) 
	Let $\hat{\P} = \{(x,M(x),Z_x)\}$ be as above, with $M(x), x \in \P$, independent initial states which are a.s. uniformly bounded, and ${\cal G} ={\cal G}(\cdot, \sim)$ an admissible graph on $\P$.  Let $M = M^{{\cal G}(\P)}$ be the system of interacting diffusions as in \eqref{e:sys_ID_V} on ${\cal G}$ with  diffusion coefficients $b,\sigma$ satisfying the Lipschitz assumption of \cite{LRW23}. 
Fix $t_0 \in (0, \infty)$ and assume the score function $\xi$ in \eqref{xisysID} and the corresponding function $h$ satisfy the following moment conditions:
	\begin{equation}\label{moment_diffusion}
	\begin{aligned}
		\sup_{x \in \mathbb{R}^d} \E[ \max(1,|h(M^{{\cal G}(\P_x)}[x,{t_0}])|^p)] &< \infty,\\
		\sup_{x \in \mathbb{R}^d, k\in \mathbb{N}} \mathbb{E}\[\max\(1,\left|h\(M^{{\cal G}_x^{k}(\mathcal{P}_x)}[x,{t_0}]\)\right|\)^{p'}\]&<\infty,
	\end{aligned}
	\end{equation}
	for some positive constants $p,~p'$ such that $\min\{p,p'\}>4$. If $\Var H_{\lambda} = \Omega(\lambda^{\nu})$ for some $\nu > 2/3$, then 
	\begin{equation*} \ignore{\label{ratesdiffusion}}
		d_W\( \frac{H_{\lambda} - \E H_{\lambda}} {\sqrt{\Var H_{\lambda}}}, Z\) = O\(\lambda^{- \kappa}\),
	\end{equation*}
	for all $\kappa< 1.5\nu-1$.
\end{cor}
\noindent {\it Proof.} The proof is adapted from that of \cite[Theorem 9.3]{BYY25} to our setting, which imposes fewer Palm requirements on the graph structure ${\cal G}$. We take the short-range estimator of the score function $\xi$ in \eqref{restr} to be
$$\hat{\xi}^{[r]}((x,M(x),Z_x),\hat{\mathcal{P}})=h\(M^{{\cal G}_x^{\left\lfloor \sqrt{r} \right\rfloor}(\mathcal{P})}[x,t_0]\)\mathbf{1}_{E_{x,r}},
$$
where  $E_{x,r}:=\{S_{\infty}(y,\P)\le\left\lfloor\sqrt{r}\right\rfloor~\text{for all}~y\in \P\cap B\(x,r-\left\lfloor\sqrt{r}\right\rfloor\)\}$. By construction, $\hat{\xi}^{[r]}((x,M(x),Z_x),\hat{\mathcal{P}})$ is determined by $\hat{\mathcal{P}}|_{B(x,r)}$. Moreover, by Remark~\ref{re1.3}~(iii) and the moment assumptions, the power $p_1$ in \eqref{momground} can be chosen arbitrarily large, while $p_2$ and $p_2'$ can be chosen as $p$ and $p'$ in \eqref{moment_diffusion}, respectively. 
Thus, by taking sufficiently large $p_1$, in Corollary~\ref{maincorLp} we take $\ubar{p}:=(p\wedge p')/2+2>4$, and set $p_3:=\ubar{p}/(\ubar{p}-2)<2$.  It suffices to show that $\xi$ satisfies polynomial $L^{p_3}$-stabilization in Definition~\ref{lpstab} with arbitrarily large $\beta$, as discussed in Remark~\ref{re1.6}~(iii). Set
$$\tilde{\xi}^{[r]}\((x,M(x),Z_x),\hat{\mathcal{P}}\)=h\(M^{{\cal G}_x^{\left\lfloor \sqrt{r} \right\rfloor}(\mathcal{P})}[x,t_0]\).$$
By the moment assumption in \eqref{moment_diffusion}, we have 
 $\|\tilde{\xi}^{[r]}\((x,M(x),Z_x),\hat{\mathcal{P}}_x\)\|_{p'}<\infty$ uniformly in $x$ and $r$, where $\hat{\mathcal{P}}_x$ is the Palm process of $\hat{\mathcal{P}}$ at $x$, i.e., $\hat{\mathcal{P}}_x\sim\law_x(\hat{\mathcal{P}})$. By the Minkowski inequality, H\"{o}lder's inequality, the Lipschitz($1$) assumption of $h$ and Lemmas~$9.6$ and~$9.7$ in \cite{BYY25}, whose bounds are stated for second moments and hence are used below after taking square roots, we have
\begin{align}
	&\left\|\xi((x,M(x),Z_x),\hat{\mathcal{P}}_x)-\hat{\xi}^{[r]}((x,M(x),Z_x),\hat{\mathcal{P}}_x)\right\|_{p_3}\nonumber\\
	&\le \left\|\xi((x,M(x),Z_x),\hat{\mathcal{P}}_x)-\tilde{\xi}^{[r]}((x,M(x),Z_x),\hat{\mathcal{P}}_x)\right\|_{p_3}\nonumber\\
	&\quad +\left\|\hat{\xi}^{[r]}((x,M(x),Z_x),\hat{\mathcal{P}}_x)-\tilde{\xi}^{[r]}((x,M(x),Z_x),\hat{\mathcal{P}}_x)\right\|_{p_3}\nonumber\\
	&\le \left\|\left\|M^{{\cal G}(\mathcal{P}_x)}[x,t_0]-M^{{\cal G}_x^{\left\lfloor \sqrt{r} \right\rfloor}(\mathcal{P}_x)}[x,t_0]\right\|_{\infty}\right\|_{p_3}+\left\|\mathbf{1}_{E_{x,r}^c}\(|h(\bm{0}(\cdot))|+\left\|M^{{\cal G}_x^{\left\lfloor \sqrt{r} \right\rfloor}(\mathcal{P}_x)}[x,t_0]\right\|_{\infty}\)\right\|_{p_3}\nonumber\\
	&\le \(\frac{C^{\left\lfloor \sqrt{r} \right\rfloor}}{\left\lfloor \sqrt{r} \right\rfloor!}\)^{1/2}+|h(\bm{0}(\cdot))|\mathbb{P}\(E_{x,r}^c\)^{1/p_3}+\mathbb{P}\(E_{x,r}^c\)^{(2-p_3)/2p_3}\((C')^{1/2}+\(\frac{C^{\left\lfloor \sqrt{r} \right\rfloor}}{\left\lfloor \sqrt{r} \right\rfloor!}\)^{1/2}\),\label{example6.2}
\end{align}
where $\bm{0}(\cdot)$ denotes the zero process on the time interval $[0,t]$, and $C$ and $C'$ are the positive constants appearing in Lemmas~$9.7$ and~$9.6$, respectively. Here, the second inequality follows from the Lipschitz property of $h$, which gives $$\left|h\(M^{{\cal G}_x^{\left\lfloor \sqrt{r} \right\rfloor}(\mathcal{P}_x)}[x,t_0]\)\right|\le |h(\bm{0}(\cdot))|+\left\|M^{{\cal G}_x^{\left\lfloor \sqrt{r} \right\rfloor}(\mathcal{P}_x)}[x,t_0]\right\|_{\infty},$$ and the third inequality follows from Minkowski's inequality, H\"older's inequality, and the fact that 
$$\left\|\left\|M^{{\cal G}_x^{\left\lfloor \sqrt{r} \right\rfloor}(\mathcal{P}_x)}[x,t_0]\right\|_{\infty}\right\|_2\le \left\|\left\|M^{{\cal G}(\mathcal{P}_x)}[x,t_0]-M^{{\cal G}_x^{\left\lfloor \sqrt{r} \right\rfloor}(\mathcal{P}_x)}[x,t_0]\right\|_{\infty}\right\|_{2}+\left\|\left\|M^{{\cal G}(\mathcal{P}_x)}[x,t_0]\right\|_{\infty}\right\|_{2}.$$ The first term on the right-hand side of \eqref{example6.2} decays faster than $r^{-\beta}$ for every $\beta\in(0,\infty)$ as $r$ goes to infinity. Furthermore,
\begin{align}
	\mathbb{P}_x(E_{x,r}^c)\le& \int_{ B\(x,r-\left\lfloor\sqrt{r}\right\rfloor\)}\mathbb{P}_y\( S_\infty(y,\mathcal{P})>\left\lfloor\sqrt{r}\right\rfloor\) \mu\mathrm{d} y\nonumber\\
	\le &\mu \Vol\( B\(x,r-\left\lfloor\sqrt{r}\right\rfloor\)\)C_{\beta}'r^{-\beta'}\nonumber\\
	\le& C_{\beta}''r^{d-\beta'}\label{example6.4}
\end{align}
for some positive constants $C_{\beta}'$ and $C_{\beta}''$ for all $\beta'\in(0,\infty)$, where $\mu$ is the intensity of $\mathcal P$ and the last inequality follows from the stabilizing assumption in Definition~\ref{graphstabilizng}. Consequently, the second and third terms on the right-hand side of \eqref{example6.2} also decay faster than $r^{-\beta}$ for all $\beta\in(0,\infty)$ by taking $\beta'=2\beta p_3/(2-p_3)+d$ in \eqref{example6.4}. Hence the estimates above verify
 polynomial $L^{p_3}$-stabilization in Definition~\ref{lpstab} with arbitrarily large polynomial decay exponent. Together with the moment condition for the short-range estimators verified above, Corollary~\ref{maincorLp} and Remark~\ref{re1.6}~(iii) yield the desired statement. \qed

\begin{re}\label{diff1stPalm2ndPalm}(cf. Remark~\ref{re1.5}~(iv)) 
		In \cite{BYY25}, the definition of a spatial random graph is based on the supremum of conditional probabilities given that $\P$ contains $x$ together with any other finite set $\{x_1,\dots,x_p\}\subset\Gamma_{\lambda}\backslash\{x\}$, for $p\in\mathbb{N}$, which is a potentially restrictive condition as it is based on higher-order Palm distributions. Our results are instead based on the first-order Palm distribution, leading to a slightly broader class of graphs for the same underlying point process $\P$.  For example, for any configuration ${\cal X}\in {\cal N}_{\mathbb{R}^d}$, we define the graph  ${\cal G}({\cal X},\sim)$ as follows: we connect each $x\in {\cal X}$ to its nearest neighbor in ${\cal X}$, 
	$$ N_x:=N_{x}({\cal X}):=\argmin_{y\in {\cal X}\setminus\{x\}}d(x,y),$$ 
	provided that $d(x, N_x)<1/\mathrm{e}$ and $\card({\cal X}\cap B(x,-\ln d(x, N_x)))$ is even. Then, 
	$$ S_{\lambda}(x,{\cal X}):=\[\sup_{y\in \{z\in{\cal X}\cap B(x,1/\mathrm{e});N_z=x\}\cup\{x\}}(-\ln d(y,N_y))\]\wedge d(x,\partial \Gamma_{\lambda})+1/\mathrm{e}$$ 
	defines an interaction range with respect to the graph ${\cal G}(\cdot,\sim)$. Let $\P$ be a homogeneous Poisson point process in $\mathbb{R}^d$ with unit intensity. One readily verifies that 
	\begin{align*}
		&\sup_{\lambda\in{(0, \infty]}}\sup_{x\in \Gamma_{\lambda}}\mathbb{P}_x(S_{\lambda}(x,\P\cap \Gamma_{\lambda})>r)\\
		&\le \sup_{\lambda\in{(0, \infty]}}\sup_{x\in \Gamma_{\lambda}}\mathbb{P}_x\(\sup_{y\in \{z\in \P\cap \Gamma_{\lambda}\cap B(x,1/\mathrm{e});N_z=x\}\cup\{x\}}(-\ln d(y,N_y(\P\cap \Gamma_{\lambda})))+\frac1 e>r\)\\
		&\le \sup_{\lambda\in{(0, \infty]}}\sup_{x\in \Gamma_{\lambda}}\mathbb{P}_x\(d(x,N_x(\P\cap \Gamma_{\lambda})) <\mathrm{e}^{\mathrm{e}^{-1}-r}\)\\
		&\le\sup_{\lambda\in{(0, \infty]}}\sup_{x\in \Gamma_{\lambda}} \mathbb{P}_x\(\card\(\P\cap (B(x,\mathrm{e}^{\mathrm{e}^{-1}-r}))\backslash\{x\}\)\ge 1\)\\
		&=\mathbb{P}\(\card\(\P\cap B\({\bf 0},\mathrm{e}^{\mathrm{e}^{-1}-r}\)\)\ge 1\)
	\end{align*}
	for all $r\ge 1$, where the last equality follows from the translation invariance of the homogeneous Poisson point process and the fact that its reduced Palm distribution coincides with its original distribution. The right-hand side decays exponentially fast. However,
	$$\sup_{\lambda\in{(0, \infty]}}\sup_{x_1,x_2\in \Gamma_{\lambda},x_1\ne x_2}\mathbb{P}_{x_1,x_2}(S_{\lambda}(x_1,\P\cap \Gamma_{\lambda})>r)=1$$
	for all $r\ge 1$, since when $x_2$ is sufficiently close to $x_1$ and $d(x_1,\partial\Gamma_{\lambda})>r,$ the dependence range of whether $x_1$ is connected to its nearest neighbor is always greater than $r$, which shows that the decay condition for the tails of $S_\lambda$ as given by  \cite[Definition~8.1]{BYY25} does not hold when $ p = 2$.  Graphs such as these,  admittedly not common,  illustrate that when we only require assumptions based on one-point Palm distributions then the class of admissible  graphs may be broader than the class of graphs satisfying assumptions based on all Palm distributions.
\end{re}

\subsection{Local $U$-statistics}\label{localustatistic}

The limit theory for local $U$-statistics has attracted considerable attention  over the past decades, see, e.g.,
\citet{RS13,ET14}  and    \cite{BYY19,BYY25} for statistics with Poisson and non-Poisson input.    
Given ${\cal X} \in {\cal N}_{\R^d}$, such statistics have the form 
$
[\sum_{\mathbf{x}\in{\cal X}^{k,\neq}}h_0(\mathbf{x})]/k!,
$ where $h_{0}$ is a real symmetric measurable function on $(\R^d)^{k,\neq}.$
\ignore{\yellow{and they can be represented as sums of score functions
		$$\sum_{x\in {\cal X}}\left\{\frac{1}{k!}\sum_{\mathbf{x}\in ({{\cal X}\backslash\{x\}})^{k-1,\neq}}h_{0}(x,\mathbf{x})\right\}=:\sum_{x\in {\cal X}}\xi_0(x,{\cal X}),$$
		where ${\cal Y}^{l,\neq}$ denotes the set of ordered $l$-tuples consisting of distinct elements in ${\cal Y}$, and $h_{0}(x,(x_1,\dots,x_{k-1})):=h_{0}(x,x_1,\dots,x_{k-1})$ is a real symmetric measurable function on $(\R^d)^{k,\neq}$.}}

Limit theorems for local $U$-statistics may be established under a locally dependent setup as in  \cite{BYY19,BYY25}, meaning that $h_0(x_1,\dots,x_k)$ vanishes whenever there exists a pair $x_i, x_j$ with $i,j\le k$ such that $|x_i - x_j|\ge r$. Related qualitative CLTs for certain $U$-statistics of Gibbs particle processes whose dominating Boolean model is subcritical were obtained in \cite{BHLV20}. Alternatively, such limit theorems may be derived under the assumption that ${\P}$ is a Poisson point process on a general metric space,  see, e.g., \cite{ET14,RS13}.

In this section, we deduce from Corollary~\ref{maincorLp}~(a) and~(b)  rates for the normal approximation error of local $U$-statistics of EDD marked point processes $\hat{\P}$. No independence assumptions are imposed beyond the EDD condition in Definition~\ref{EDD}. Instead, we let the kernel $h$ be a real symmetric measurable function on the space of the marked $k$-tuples $\{\((y_1,m_1),...,(y_k,m_k)\);\ (y_1,\dots,y_k)\in (\R^d)^{k,\neq},\ m_1,\dots,m_k\in\mathbb{M}\}$. We assume that the conditional $L^{p_0}$ norm of $h((x_1,M_{x_1}),\dots,(x_k,M_{x_k}))$ decays either exponentially or polynomially\ignore{(with sufficiently high order)} in the diameter of the argument set, for a given positive $p_0$, where $M_x$ denotes the mark of the point  $x\in \P$. We denote by $H_{\lambda}$ the $U$-statistic generated by $\hat{\P}|_{\Gamma_{\lambda}}$, namely
\begin{equation}\label{repW}
	H_{\lambda}	=\frac{1}{k!}\sum_{\hat{\mathbf{x}}\in (\hat{\P}|_{\Gamma_{\lambda}})^{k,\neq}}h(\hat{\mathbf{x}})=\sum_{x\in \P\cap \Gamma_{\lambda}}\xi((x,M_x),\hP,\Gamma_{\lambda}),
\end{equation}
where 
\begin{equation}\xi((x,m_x), \hatX,\Gamma_{\lambda}):=\frac{1}{k!}\sum_{\hat{\mathbf{x}}\in ({\hatX|_{\Gamma_{\lambda}\backslash\{x\}} })^{k-1,\neq}}h((x,m_x),\hat{\mathbf{x}}).\label{defxiustat01}\end{equation}

\begin{thm}\label{example7}(quantitative CLT for local $U$-statistics)
	Let  $\hP$ be  an EDD marked point process such that the ground point process ${\P}$ is either (i) a stationary Gibbs point process  whose  interaction range is smaller than the critical value of continuum percolation of an associated Poisson point process, or (ii) an $\alpha$-determinantal point process with  kernel $K(x,y):=K_0(x-y)=\overline{K_0(y-x)}$ satisfying condition  \eqref{reEDDdoubleexp}, and $\alpha$ such that $-1/\alpha\in\mathbb{N}$. Let $H_\lambda$ be as in \eqref{repW} and assume  $\var H_{\lambda} =\Omega(\lambda^\nu)$ for some constant $\nu>2/3$. 
	\begin{description}
		\item{(a)} If  the conditional $L^{p_0}$ norm of $h$ decays exponentially fast for some $p_0>3$, i.e., if
		$$\mathbb{E}_{x_1,\dots,x_k}\(|h((x_1,M_{x_1}),\dots,(x_k,M_{x_k}))|^{p_0}\)^{1/p_0}\le C_1\mathrm{e}^{-C_2 \diam(\{x_1,\dots,x_k\})}$$ 
		for some positive constants $C_1$ and $C_2$ for all distinct $x_1,\dots,x_k\in \mathbb{R}^d$, we have
		\begin{equation*}
			d_W\( \frac{H_{\lambda} - \E H_{\lambda}} {\sqrt{\Var H_{\lambda}}}, Z\)= O\( (\log \lambda)^{3d} \lambda^{-1.5\nu+1}\),
		\end{equation*}
		where $Z\sim N(0,1)$.
		\item{(b)} If  the conditional $L^{p_0}$ norm of $h$  decays polynomially fast for some $p_0>3$, i.e., if
		$$\mathbb{E}_{x_1,\dots,x_k}\(|h((x_1,M_{x_1}),\dots,(x_k,M_{x_k}))|^{p_0}\)^{1/p_0}\le C_3\(1\vee \diam(\{x_1,\dots,x_k\})\)^{-C_4}$$ 
		for some positive  constants $C_3$ and $C_4$ for all distinct $x_1,\dots,x_k\in \mathbb{R}^d$, with $$C_4>\frac{d}{3\nu-2}((18-3\nu)\vee 14)+d(k-1)+1,$$ 
		we have
		\begin{equation*}
			d_W\( \frac{H_{\lambda} - \E H_{\lambda}} {\sqrt{\Var H_{\lambda}}}, Z\) = O\(\lambda^{-1.5\nu+1+ \tau(2d,\beta)}\)=o(1),
		\end{equation*}
		for $\beta=C_4-d(k-1)-1$ where $Z\sim N(0,1)$ and where $\tau$ is defined in \eqref{taudef}.
	\end{description}
\end{thm}

\begin{re}\label{re.ex7} 
	\begin{description}
		\item{(i)}
		Theorem~\ref{example7} can be easily extended to the case where the score is a function of $\hP$ instead of $\hP\vert_{\G_{\lambda}}$, i.e., for the case in which
		$H_{\lambda}=\sum_{x\in {\P}\cap\Gamma_{\lambda}}\xi((x,m_x),{\hP}),$ with minor adjustments.
		\item{(ii)} With minor modifications, all results also apply to EDD marked point processes $\hat{\P}$ whose ground process satisfies the bound \eqref{ex.lu.lma-x00} below for sufficiently large $k_0$. The same conclusions remain valid if the above conditional $L^{p_0}$ bounds are assumed to hold uniformly in the essential supremum sense.  For sake of brevity we omit the details.
	\end{description}
\end{re}

Before proceeding with the proof of the theorem, we state moment bounds for the considered point processes.

\begin{lma}\label{ex.lu.lma} 
	For each point process $\P$ in Theorem~\ref{example7} and each $k_0\in (0,\infty)$, there exists a constant $D_{k_0}\in(0,\infty)$, depending only on ${\cal L}(\P)$ and $k_0$, such that for all punctured balls $B^{\circ}(x,r):=B(x,r)\backslash\{x\}$, 
	\begin{equation}\label{ex.lu.lma-x00}
		\sup_{x\in \mathbb{R}^d}\mathbb{E}_{x}(\P(B^{\circ}(x,r))^{k_0})^{1/k_0}\le D_{k_0}(1\vee r)^d, \quad  r\in (0,\infty).       
	\end{equation}
\end{lma} 
\noindent{\it Proof.} When $\P$ is a Gibbs point process whose  interaction range  is smaller than the critical value of continuum percolation of an associated Poisson point process,  
it follows from \cite[Section~4.1]{HOS25} that it is  stochastically dominated by a homogeneous Poisson process $\P_\tau$ with a positive intensity $\tau$. Let $\mu$ be the intensity of $\P$. Fix $0<\delta_1<\delta_2<r$, and define the spherical shell $B_{x,\delta_2}=B(x,r)\setminus B(x,\delta_2)$. Then 
\begin{equation}\label{ex.lu.lma-x01}
	\frac{\mathbb{E}\( \P(B_{x,\delta_2})^{k_0}\P(B(x,\delta_1))\) } {\mathbb{E}\(\P(B(x,\delta_1))\)}\le \frac{\mathbb{E}\(\P_\tau(B_{x,\delta_2})^{k_0}\P_\tau(B(x,\delta_1))\)}{\mu\Vol(B(x,\delta_1))}=\frac{\tau}{\mu}\mathbb{E}\(\P_\tau(B_{x,\delta_2})^{k_0}\). 
\end{equation} 
First taking the limit as $\delta_1\to 0$, and then as $\delta_2\to 0$  in \eqref{ex.lu.lma-x01}, the definition of the Palm distribution in \citet[(10.3)]{kallenberg83} gives
\begin{equation*}
	\mathbb{E}_x \(\P(B^{\circ}(x,r))^{k_0}\) \le \frac{\tau}{\mu}\mathbb{E}\(\P_\tau(B^{\circ}(x,r))^{k_0}\).
\end{equation*} 
This implies inequality \eqref{ex.lu.lma-x00}.

When ${\P}$ is a determinantal point process, condition \eqref{reEDDdoubleexp} together with the Hermitian property ensures that its kernel satisfies Assumptions~1-3, while Assumption~4 is equivalent to the existence of the process given the first three assumptions (cf. the proof of Theorem~\ref{example1}). According to \cite[Theorem~1]{MO21}, for all $x \in \mathbb{R}^d$, there exists a coupling of ${\P}$ and its  reduced Palm version ${\P}_x^{!}$ such that ${\P}_x^{!} \subseteq {\P}$ almost surely. This property, together with Remark~\ref{remark1.1} and Lemma~\ref{moments-ground}, guarantees the desired conclusion. \qed

\noindent{\it Proof of Theorem~\ref{example7}.} We will deduce this from  Corollary~\ref{maincorLp}. The point process ${\P}$ satisfies the EDD condition as shown in \cite[Section~3]{CX23}. By Remarks~\ref{remark1.1} and~\ref{re1.6}~(i), it suffices to show that assumptions \eqref{momscore}, \eqref{momrestrictedscore} hold and that exponential $L^{p_3}$-stabilization for part (a) and polynomial $L^{p_3}$-stabilization for part (b) hold for a family of short-range score functions $(\hat{\xi}^{[r]})_{r\in(0,\infty)}$, for $p_2=p_2' =p_0>3$ and $p_3=3$.  Define $\hat{\xi}^{[r]}$ to be the restriction of $\xi$ to $\hP|_{B(x,r)}$ as in \eqref{restriction}, that is,  $$\hat{\xi}^{[r]}((x,m_x), \hatX,\Gamma_{\lambda})=\frac{1}{k!}\sum_{\hat{\mathbf{x}}\in ({\hatX|_{ B(x,r)^{\circ}\cap \Gamma_{\lambda}} })^{k-1,\neq}}h((x,m_x),\hat{\mathbf{x}}).$$ 

Define for $\mathbf{y}=\(y_1,\dots,y_n\)\in \(\real^d\)^{n,\neq}$ and $x\in \real^d$
$$
d_s(x,\mathbf{y}):=
d_s(x,\{y_1,\dots,y_n\}) := \sup\{d(x,y_j):\ 1\le j\le n\}.
$$
Since $d_s(x,\{y_1,\dots,y_n\}) \le \diam(\{x,y_1,\dots,y_n\})$, the fast-decay assumptions on $h$ in parts~(a) and~(b) imply, respectively, that
\begin{align*}
	\mathbb{E}_{x,y_1\dots,y_{k-1}}\(|h((x,M_{x}),(y_1,M_{y_1})\dots,(y_{k-1},M_{y_{k-1}}))|^{p_0}\)^{1/p_0} &\le C_1 \mathrm{e}^{-C_2 d_s(x,\{y_1,\dots,y_{k-1}\})}\quad\text{and},\\
	\mathbb{E}_{x,y_1\dots,y_{k-1}}\(|h((x,M_{x}),(y_1,M_{y_1})\dots,(y_{k-1},M_{y_{k-1}}))|^{p_0}\)^{1/p_0} &\le C_3 \(1\vee d_s(x,\{y_1,\dots,y_{k-1}\})\)^{-C_4}.
\end{align*}  
We then set 
$$\mathrm{UB}(\varrho)=C_1 \mathrm{e}^{-C_2 \varrho}\quad \text{for (a)}, \qquad\mathrm{UB}(\varrho)=C_3 (1\vee \varrho)^{-C_4}\quad \text{for (b)}.$$

For any marked configuration $\hatX$ with ground configuration $\mathcal{X}$, and for $(x,m_x)\in\hatX$, we decompose the sum in \eqref{defxiustat01} according to the  spherical shell $B(x,j)\setminus B(x,j-1)$ which contains 
among the ground coordinates of $\hat{\mathbf{y}}\in(\hatX)^{k-1,\neq}$ the point  furthest from $x$.
Precisely, for $j\in\mathbb N$, we define
\begin{align*}
	S_j\(x,\hatX\):=&\sum_{\hat{\mathbf{y}} \in (\hatX)^{k-1,\neq},\ d_s(x,\mathbf{y})\in[j-1,j)}h((x,m_x), \hat{\mathbf{y}})\\
	=&\sum_{\mathbf{y}\in ({\cal X}|_{ B^{\circ}(x,j)})^{k-1,\neq}\setminus ({\cal X}|_{ B^{\circ}(x,j-1)})^{k-1,\neq}}h((x,m_x),\hat{\mathbf{y}}),
\end{align*} 
where $\hat{\mathbf{y}}=((y_1,m_{y_1}),\dots,(y_n,m_{y_n}))$ for $\mathbf{y}=(y_1,\dots,y_n)\in \mathcal{X}^{n,\neq}$.
The number of admissible $(k-1)$-tuples contributing to the $j$-th shell term is bounded by the $(k-1)$-st power of the number of ground points in the corresponding shell. Thus,  for any $j \in \N$ we obtain  
via  Lemma~\ref{ex.lu.lma}
\begin{align}
	\mathbb{E}_{x}(|S_j\(x,\hP\vert_{ \Gamma_{\lambda}}\)|^{3})^{1/3}
&
\le \mathbb{E}_{x}(|S_j(x,\hP\vert_{ \Gamma_{\lambda}})|^{p_0})^{1/p_0}  \nonumber\\
	&= \mathbb{E}_{x}\(\mathbb{E}\left(\left|S_j\(x,\hP\vert_{ \Gamma_{\lambda}}\)\right|^{p_0}\middle|\P\cap(B(x,j)\setminus B(x,j-1))\right)\)^{1/p_0} \nonumber\\
	&\le \mathrm{UB}(j-1)^{p_0}\mathbb{E}_{x}\(\P(B(x,j)\setminus B(x,j-1))^{(k-1)p_0}\)^{1/p_0}\nonumber\\
	&\le D_{(k-1)p_0}^{k-1}\mathrm{UB}(j-1)j^{d(k-1)},\label{ex7-x1}
\end{align}
where the penultimate  inequality follows from Minkowski's inequality.

Rewrite $\xi$ and $\hat{\xi}^{[r]}$ via the additive representations
$$\xi((x,M_x),{\hP},\Gamma_{\lambda})=\frac{1}{k!}\sum_{j\in \mathbb{N}} S_j(x,{\hP}|_{ \Gamma_{\lambda}})~\mbox{and}~\hat{\xi}^{[r]}((x,M_x),{\hP},\Gamma_{\lambda})=\frac{1}{k!}\sum_{j\in \mathbb{N}}S_j(x,{\hP}|_{B(x,r)\cap \Gamma_{\lambda}}).$$

We first show that the moment assumptions \eqref{momscore} and \eqref{momrestrictedscore} hold for $p_2=p_2'=p_0>3$. The Minkowski inequality and  \eqref{ex7-x1} imply that in cases~(a) and~(b), we have
\begin{align*}
	\(\mathbb{E}_{x}\left|\xi((x,M_x),{\hP},\Gamma_{\lambda})\right|^{p_0}\)^{1/p_0}&\le \frac{1}{k!}\sum_{j=1}^{\infty}	\(\mathbb{E}_{x}\left|S_j\(x,{\hP}|_{ \Gamma_{\lambda}}\)\right|^{p_0}\)^{1/p_0}\nonumber\\ 
	&\le \frac{D_{(k-1)p_0}^{k-1}}{k!}\sum_{j=1}^{\infty}\mathrm{UB}(j-1)^{d(k-1)}<\infty,\ignore{\label{ex7}}
\end{align*}
where, by assumption,  $C_4>d(k-1)+1$. Moreover, the  assumption \eqref{momrestrictedscore} also holds for $p_2'=p_0$ by replacing $\xi$ and $S_j\(x,{\hP}|_{\Gamma_{\lambda}}\)$ with $\hat{\xi}^{[r]}$ and $S_j\(x,{\hP}|_{ B(x,r)\cap \Gamma_{\lambda}}\)$, respectively.

Next, we show that exponential $L^{p_3}$-stabilization, as in case~(a),  and polynomial $L^{p_3}$-stabilization, as in  case~(b),  hold for $p_3=3$. For any given $r_0>1$, we have 
\begin{align}
	&\xi((x,M_x),{\hP},\Gamma_{\lambda})-\hat{\xi}^{[r_0]}((x,M_x),{\hP},\Gamma_{\lambda})\nonumber\\
	&=\frac{1}{k!}\[\sum_{j=\lceil r_0\rceil+1}^\infty S_j\(x,{\hP}|_{\Gamma_{\lambda}}\)+
	\sum_{\hat{\mathbf{y}} \in \(\hP|_{B^{\circ}(x,\lceil r_0\rceil)}\)^{k-1,\neq}\setminus \(\hP|_{ B^{\circ}(x,r_0)}\)^{k-1,\neq}}h\((x,M_x),\hat{\mathbf{y}}\)\].\label{ex7-x2}
\end{align}
As the first term on the right-hand side in  \eqref{ex7-x2} is the leading term, we can assume, without loss of generality, that $r_0$ is an integer so that the second sum vanishes. Then, applying the Minkowski inequality and using \eqref{ex7-x1} again, we get
\begin{align*}
	&\(\mathbb{E}_{x}\left|\xi((x,M_x),{\hP},\Gamma_{\lambda})-\hat{\xi}^{[r_0]}((x,M_x),{\hP},\Gamma_{\lambda})\right|^{3}\)^{1/3}\nonumber\\
	&\le \frac{1}{k!}\sum_{j=r_0+1}^{\infty}\(\mathbb{E}_{x}\left|S_j(x,{\hP}|_{\Gamma_{\lambda}})\right|^{3}\)^{1/3} \nonumber
	\\ 
	&\le \frac{D_{(k-1)p_0}^{k-1}}{k!}\sum_{j=r_0+1}^{\infty}\mathrm{UB}(j-1)j^{d(k-1)}
\end{align*}
implying  cases (a) and (b) provided  $C_4>d(k-1)+1$. Indeed,
the right-hand side is bounded above by $C_5\mathrm{e}^{-C_6r_0}$ for some positive constants $C_5$ and $C_6$ in case~(a), and by $C_7r_0^{d(k-1)+1-C_4}$ for some positive constant $C_7$ in case~(b). Hence, statements~(a) and~(b) follow  from Corollary~\ref{maincorLp}~(a) and~(b), respectively. \qed

\subsubsection{Subgraph counting in the random connection model}

We study the problem of counting subgraphs within a class of random connection graphs inspired by the {\it age-dependent random connection model} (cf. \citet{HLO25} and references therein), where vertices are equipped with $\iid$ marks and the existence of an edge between two vertices depends on both their Euclidean distance $d$ and their respective marks.

Assume that $\P$ is one of the point processes introduced in Theorem~\ref{example7}, and attach to each $x \in {{\P}}$ an independent mark $M_x \sim U[0,1]$, with the marks being independent of ${{\P}}$. 
Given a function $\psi:(0,\infty) \to [0,1]$ and a symmetric function $g: [0,1] \times [0,1] \to [0,1]$, we
construct a random graph ${\cal G}(\P):={\cal G}(\P, g, \psi)$ with vertex set ${\P}$ by connecting  points $x, y \in {{\P}}$ whenever $g(M_x, M_y) \le \psi(d(x, y))$. 

For example, taking  $\psi$ to be a constant $p \in [0,1]$ leads to the connection rule that we connect vertices
$x,y$ whenever their marks satisfy  $g(M_x, M_y) \le p$,  that is, all pairs of points are connected with the same probability; the probabilities are dependent when a common vertex is involved and thus this model does not precisely coincide  with the age-dependent random connection model. On the other hand, if $\psi(t) = {\bf 1}_{(0 \leq t \leq r)}$ and $g\not\equiv0$, then we obtain a version of the random geometric graph with parameter $r$. Fixing $p \geq 0$, taking $g(s, t) = \max\{s, t\}$ and $\psi(r) = Cr^{-p}, r \in (0, \infty),$  gives rise to the graph defined by the connection rule that we connect $x, y$ whenever 
$$ \max\{U_x, U_y\} \leq  C(d(x,y))^{-p}, $$
which says that we connect points $x, y$ as soon as the balls centered at $x$ and $y$ with respective radii  $C^{1/p}U_x^{-1/p},  C^{1/p}U_y^{-1/p}$  contain both $x,y$. 
When $\P$ is  replaced by the lattice $\Z^d$, there is a regime of $p$-values such that the resulting graph is scale-free; see \citet{JY06}.

Let ${\cal G}_{\lambda}(\P)$  denote the restriction of ${\cal G}(\P)$  to the vertex set ${{\P}}\cap \G_{\lambda}$. Let $H_{\lambda}:=H_{\lambda}(F)$ denote the number of subgraphs of ${\cal G}_{\lambda}(\P)$  that are isomorphic to a given connected simple graph $F$ with $k$ vertices, where {\it connected} means that there exists a path between each pair of vertices in $F$, and {\it simple} means it is an unweighted, undirected graph containing neither self-loops nor multiple edges.
\begin{thm}\label{thm3.16} (quantitative CLT for the subgraph count in a type of age-dependent random connection model)
Assume that the point process ${{\P}}$ satisfies the assumptions in Theorem~\ref{example7}, $\var H_{\lambda} = \Omega(\lambda^\nu)$ for some positive constant $\nu>2/3$, and $g: [0,1] \times [0,1] \to [0,1]$ is a symmetric function which satisfies 
 $g(s, t) \ge \max\{s, t\}$. Let $Z\sim N(0,1)$. 
	\begin{description}
		\item{(a)} If $\psi$ decays exponentially fast, i.e.,  there exist positive constants $C_1$ and $C_2$ such that $\psi(r) \le C_1\mathrm{e}^{-C_2r}$ for all $r \in(0,\infty)$, then  
		\begin{equation*}
			d_W\( \frac{H_{\lambda} - \E H_{\lambda}} {\sqrt{\Var H_{\lambda}}}, Z\) = O\((\log \lambda)^{3d}\lambda^{-1.5\nu+1}\).
		\end{equation*}
		\item{(b)} If there exist constants $C_3 \in (0, \infty)$  and 
		$$C_4 >3 \left(\frac{d}{3\nu-2}((18-3\nu)\vee 14)+ d(k-1)+1\right)$$  
		such that $\psi(r) \le C_3(1\vee r)^{-C_4}$ for all $r \in(0,\infty)$, then
		\begin{equation*}
			d_W\( \frac{H_{\lambda} - \E H_{\lambda}} {\sqrt{\Var H_{\lambda} }},  Z\) =  O\(\lambda^{-1.5\nu+1+\tau(2d,\beta)+\epsilon}\),
		\end{equation*}
		for all $\epsilon \in (0,\infty)$, where $\beta=C_4/3-d(k-1)-1$ and where $\tau$ is defined in \eqref{taudef}. The above bound is $o(1)$ when $\epsilon$ is small enough.
	\end{description}
\end{thm}

\noindent{\it Proof.}  Define $h((x_1,M_{x_1}),\dots,(x_k,M_{x_k}))$ to be the number of subgraphs of ${\cal G}(\P)$  that are isomorphic to $F$, with vertex set $\{x_1,\dots,x_k\}\in{{\P}}^{k,\neq}$ and associated marks $\{M_{x_1},\dots,M_{x_k}\}$. Without loss of generality, assume that the vertex set of $F$ is $\{v_1,\dots,v_k\}$, and write $i \overset{F}{\sim} j$ to indicate that there exists an edge between $v_i$ and $v_j$ in $F$ for $i,j \in \{1,\dots,k\}$. Let $\mathrm{aut}(F)$ denote the number of automorphisms of the graph $F$. If we take
$$ h((x_1,M_{x_1}),\dots,(x_k,M_{x_k}))=\frac{1}{\mathrm{aut}(F)}\sum_{\bm{\pi}\in \textbf{P}\{1,\dots,k\}}\prod_{i\overset{F}{\sim} j}\mathbf{1}_{g(M_{x_{\bm{\pi}(i)}},M_{x_{\bm{\pi}(j)}})\le \psi(d(x_{\bm{\pi}(i)},x_{\bm{\pi}(j)}))},$$
where $\textbf{P}\{1,\dots,k\}$ denotes the set of all permutations of $\{1,\dots,k\}$, then $H_{\lambda}$ can be represented as a local $U$-statistic as in \eqref{repW}.

Next, we upper bound  the $L^p$ norm of $h$,  
$p \in (0,\infty)$.  From the construction, for any $\bm{\pi}\in \textbf{P}\{1,\dots,k\}$, the product 
$$
\prod_{i\overset{F}{\sim} j}\mathbf{1}_{g(M_{x_{\bm{\pi}(i)}},M_{x_{\bm{\pi}(j)}})\le \psi(d(x_{\bm{\pi}(i)},x_{\bm{\pi}(j)}))}$$
is the indicator of the event that an isomorphism of $F$, in which vertex $v_i$ is mapped to $x_{\bm{\pi}(i)}$, $1\le i\le k,$ forms a subgraph of ${\cal G}(\P)$ given that $\{x_1,\dots,x_k\} \subset \P$.
Because $F$ is connected, any such isomorphism must contain a path of the form $(x_{i_1}, \dots, x_{i_m})$ for some $2 \le m \le k$, such that $d\left(x_{i_1},x_{i_m}\right)=\text{diam}\(\{x_1,\dots,x_k\}\)$, i.e., a path connecting the most distant points among $\{x_1,\dots,x_k\}$. Therefore, for any $\bm{\pi}\in \textbf{P}\{1,\dots,k\}$ and $\{x_1,\dots,x_k\}\in {{\P}}^{k,\neq}$,
and recalling $g(s, t) \ge \max\{s, t\}$,
 we estimate the $L^p$-norm of the indicator function as follows:
\begin{align}
	\left\|\prod_{i\overset{F}{\sim} j}\mathbf{1}_{g(M_{x_{\bm{\pi}(i)}},M_{x_{\bm{\pi}(j)}})\le \psi(d(x_{\bm{\pi}(i)},x_{\bm{\pi}(j)}))}\right\|_{p} 
	\le&\left\|\prod_{j=1}^{m-1}\mathbf{1}_{\max\{M_{x_{i_j}},M_{x_{i_{j+1}}}\}\le \psi(d(x_{i_j},x_{i_{j+1}}))}\right\|_{p}\nonumber
	\\\le &\prod_{j=1}^{m-1}\mathbb{P} \(M_{x_{i_j}}\le \psi(d(x_{i_j},x_{i_{j+1}}))\)^{1/p}.\nonumber
\end{align}
In case (a), where $\psi$ decays exponentially fast, we have
\begin{equation}\|h((x_1,M_{x_1}),\dots,(x_k,M_{x_k}))\|_{p}\le \frac{k!}{\mathrm{aut}(F)}(C_1\vee 1)^{(k-1)/{p}}\mathrm{e}^{-C_2\text{diam}\(\{x_1,\dots,x_k\}\)/{p}}.\label{caseaexpoadd1}\end{equation}
When $\psi$ decays polynomially fast as in case (b), we have
\begin{equation}\label{ustatisticdecay}
	\|h((x_1,M_{x_1}),\dots,(x_k,M_{x_k}))\|_{p}\le \frac{k!}{\mathrm{aut}(F)}(C_3\vee 1)^{1/{p}}\(\frac{\text{diam}\(\{x_1,\dots,x_k\}\)}{k-1}\)^{-C_4/{p}},
\end{equation}
noting that the existence of a path $(x_{i_{\bm{\pi},1}}, \dots, x_{i_{\bm{\pi},m}})$ for some $2 \le m \le k$ implies that there is  at least one edge in the path with length at least  $d(x_{i_{\bm{\pi},1}}, x_{i_{\bm{\pi},m}})/(k-1)$. Take $p=3+\epsilon_0$ for some 
 $\epsilon_0 \in (0,\infty)$. Then \eqref{ustatisticdecay} shows that
\begin{equation}\|h((x_1,M_{x_1}),\dots,(x_k,M_{x_k}))\|_{3+\epsilon_0}\le C_5\(\text{diam}\(\{x_1,\dots,x_k\}\)\)^{-C_4/{(3+\epsilon_0)}}, 
\label{caseaexpoadd2}
\end{equation} 
i.e., the $L^{3+\epsilon_0}$-norm of the kernel $h$ decays polynomially fast with exponent $C_4/(3+\epsilon_0)$. Since the marks are independent, the conditional
$L^p$-norm of $h$ given $\{x_1,\dots,x_k\}\subset\P$, coincides with the ordinary $L^p$-norm under Palm
conditioning. 
The assertion in case (a) follows  from
Theorem~\ref{example7} and \eqref{caseaexpoadd1}. For case (b), choose $\epsilon_0>0$ sufficiently small so that
$$\frac{C_4}{3+\epsilon_0}
>
\frac{d}{3\nu-2}\((18-3\nu)\vee 14\)+d(k-1)+1.
$$
Then \eqref{caseaexpoadd2} implies the polynomial-decay condition of
Theorem~\ref{example7}~(b) with decay exponent $C_4/(3+\epsilon_0)$, and hence  the conclusion follows from Theorem~\ref{example7}~(b). \qed

\vskip.3cm

\noindent{\bf Acknowledgements}. We thank Anne Marie Svane for  constructive and insightful comments on an earlier version of this paper. J.  Yukich gratefully acknowledges the kind hospitality and generous financial support from the University of Melbourne,  where the writing of this paper was initiated.


\def\ac{{Academic Press}~}
\def\aap{{Adv. Appl. Prob.}~}
\def\ap{{Ann. Prob.}~}
\def\anap{{Ann. Appl. Prob.}~}
\def\eljp{{\it Electron.\ J.~Prob.\/}~} 
\def\jap{{J. Appl. Prob.}~}
\def\jws{{John Wiley $\&$ Sons}~}
\def\ny{{New York}~}
\def\ptrf{{Prob. Theory Related Fields}~}
\def\sp{{Springer}~}
\def\spa{{Stochastic Processes and their Applications}~}
\def\sv{{Springer-Verlag}~}
\def\tpa{{Theory Prob. Appl.}~}
\def\zw{{Z. Wahrsch. Verw. Gebiete}~}
\def\jtp{{J. Theor. Prob.}~}


\begin{thebibliography}{9}
	
	
	\bibitem[Beardwood, Halton and Hammersley(1959)]{BHH59}Beardwood, J., Halton, J. and Hammersley, J.~(1959). The shortest path through many points. \emph{Math. Proc. Camb. Philos. Soc.}~\textbf{55}, 299--327.
	
	\bibitem[Bene{\v{s}} et al. (2020)]{BHLV20} Bene{\v{s}}, V., Hofer-Temmel, C., Last, G. and Ve{\v{c}}e{\v{r}}a, J.~(2020). Decorrelation of a class of Gibbs particle processes and asymptotic properties of $U$-statistics. \emph{\jap}\textbf{57}, 928--955.
	
	
	\bibitem[Bickel and Breiman(1983)]{BB} Bickel, P. and Breiman, L.~(1983).  Sums of functions of nearest neighbor distances, moment bounds, limit theorems, and a goodness of fit test.  \emph{Ann. Prob.}~\textbf{11}, 184-214.
	
	\bibitem[B{\l}aszczyszyn, Yogeshwaran and Yukich(2019)]{BYY19} B{\l}aszczyszyn, B., Yogeshwaran, D. and Yukich, J. E.~(2019). Limit theory for geometric statistics of point processes having fast decay of correlations. \emph{\ap}\textbf{47}, 835--895.
	
	\bibitem[B{\l}aszczyszyn, Yogeshwaran and Yukich(2026)]{BYY25} B{\l}aszczyszyn, B., Yogeshwaran, D. and Yukich, J. E.~(2026).  Limit theory for  Lipschitz-localized statistics in random geometric models. arXiv:2605.28430.
	

	\bibitem[Bradley(2005)]{B05} Bradley, R. C.~(2005). Basic properties of strong mixing conditions. A survey and some open questions. \emph{Prob. Surveys}~\textbf{2}, 107--144.
	

	\bibitem[Brockwell and Davis(2016)]{BD16} Brockwell, P. J. and Davis, R. A.~(2016). \emph{Introduction to Time Series and Forecasting} (3rd ed.). Springer.
	
	\bibitem[Bulinski and Shashkin(2007)]{BA07} Bulinski, A. and Shashkin, A.  (2007). \emph{Limit Theorems for Associated Random Fields and Related Systems.} Vol. 10. World Scientific.
	
	\bibitem[Chen, Goldstein and Shao(2011)]{CGS11} Chen, L. H. Y., Goldstein, L. and Shao, Q. M.~(2011). {\em Normal Approximation By Stein's Method}. Springer-Verlag.
	
	
	
	\bibitem[Cong and Xia(2023)]{CX23} Cong, T. and Xia, A.~(2023). Convergence rate for geometric statistics of point processes having fast decay of dependence. \emph{Electron. J. Prob.}~\textbf{28}, 1--35.
	
	\bibitem[Cong and Xia(2024)]{CX24} Cong, T. and Xia, A.~(2024). Normal approximation in total variation for statistics in geometric probability. \emph{\aap}\textbf{56}, 106--155.
	
	\bibitem[Davydov(1968)]{D68} Davydov, Y. A.~(1968). Convergence of Distributions Generated by Stationary Stochastic Processes. \emph{Theory Probab. Its Appl.~}\textbf{13}, 691--696.
	
	\bibitem[Dereudre(2019)]{D19}
	Dereudre, D. (2019). Introduction to the theory of Gibbs point processes. In
	\emph{Stochastic Geometry}, Lect. Notes in Math., \textbf{2237}, 181--229. Springer.
	
	\bibitem[Dinh et al.(2025)]{D25} Dinh, T.C., Ghosh, S., Tran, H.S. and Tran, M.H.~(2025).  Gaussian fluctuations for spin systems and point processes: near-optimal rates via quantitative Marcinkiewicz's theorem.  arXiv:2107.08469v3.

\bibitem[Eichelsbacher and Th$\"a$le(2014)]{ET14} Eichelsbacher, P. and Th\"ale, C.~(2014). 
New Berry-Esseen bounds for non-linear functionals of Poisson random measures. \emph{Electron. J. Probab.},
\textbf{19}, 1--25.
	
	
	
	
	\bibitem[Forrester and Lebowitz(2014)]{FL14} Forrester, P. J.  and Lebowitz, J. L.~(2014). Local central limit theorem for determinantal point processes. \emph{J. Stat. Phys.}~\textbf{157}, 60--69.
	
	
	\bibitem[Hamilton(1994)]{H94} Hamilton, J. D.~(1994). {\em Time Series Analysis}. Princeton University Press.
	
	\bibitem[Hern\'{a}ndez Ru\'{\i}z(2025)]{Rta} Hern\'{a}ndez Ru\'{\i}z, L. I.~(2025). Results for convergence rates associated with renewal processes. In \emph{S\'{e}minaire de Probabilit\'{e}s LII,} Lect. Notes in Math.,~\textbf{2363}, 427--450. Springer.
	
	\bibitem[Hirsch, Lachi{\`e}ze-Rey and Owada(2025)]{HLO25} Hirsch, C., Lachi{\`e}ze-Rey, R. and Owada, T.~(2025). Normal approximation for subgraph counts in age-dependent random connection models, arXiv:2505.09318.
	
	\bibitem[Hirsch, Otto and Svane(2025)]{HOS25} Hirsch, C., Otto, M. and Svane, A. M.~(2025). Normal approximation for Gibbs processes via disagreement couplings. \emph{Electron. J. Prob.}~{\bf 30}, 1--56.
	
	\bibitem[Hirsch, Otto and Svane(2026)]{HOS26} Hirsch, C., Otto, M. and Svane, A. M.~(2026). Normal and Poisson approximation for Gibbs point processes with pair potentials. arXiv:2601.18695.
	
	\bibitem[Hough et al.(2009)]{H09} Hough, J. B., Krishnapur, M., Peres, Y. and Vir{\`a}g, B.~(2009). \emph{Zeros of Gaussian analytic functions and determinantal point processes}. American Mathematical Society, Providence, RI.
	
	
	\bibitem[Kallenberg(1983)]{kallenberg83} Kallenberg, O.~(1983). \newblock {\em Random Measures}. \newblock Academic Press, London.
	
	
	
	
	\bibitem[Lachi\`eze-Rey, Schulte and Yukich(2019)]{LSY19} Lachi\`eze-Rey, R., Schulte, M. and Yukich, J. E. (2019). Normal approximation for stabilizing functionals.  \emph{\anap}\textbf{29}, 931--993. 
	
	\bibitem[Lacker, Ramanan and Wu(2023)]{LRW23} Lacker, D., Ramanan, K. and Wu, R.~(2023). Local weak convergence for sparse networks of interacting processes. \emph{\anap}\textbf{33}, 843--888.
	
	\bibitem[Lavancier, M{\o}ller and Rubak(2015)]{LMR15} Lavancier, F., M{\o}ller, J. and Rubak, E.~(2015). Determinantal point process models and statistical inference. \emph{J. R. Stat. Soc., B: Stat. Methodol.}~\textbf{77}, 853--877.
	
	
	\bibitem[Lo and Xia(2024)]{LX24} Lo, T. Y. Y. and Xia, A. (2024). On the rate of normal approximation for Poisson continuum percolation. \emph{Stat. Probab. Lett.}~\textbf{210}, Article~110110.
	
	
	
	\bibitem[Meester and Roy(1996)]{MY96} Meester, R. and Roy, R.~(1996). {\em Continuum Percolation.\/} Oxford University Press. 
	
	\bibitem[Michelen and Perkins(2022)]{MP22} Michelen, M. and Perkins, W.~(2022). Strong spatial mixing for repulsive point processes. \emph{J. Stat. Phys.}~\textbf{189}, Art. No. 9.
	
	\bibitem[M{\o}ller and O'Reilly(2021)]{MO21} M{\o}ller, J. and O'Reilly, E.~(2021). Couplings for determinantal point processes and their reduced Palm distributions with a view to quantifying repulsiveness. \emph{{\jap}}\textbf{58}, 469--483.
	
	
	
	
	
	\bibitem[Penrose and Pisztora(1996)]{PP96} Penrose, M. D. and Pisztora, A.~(1996). Large deviations for discrete and continuous percolation. \emph{\aap}\textbf{28}, 29--52. 
	
	
	
	\bibitem[Reitzner and Schulte(2013)]{RS13} Reitzner, M. and Schulte, M.~(2013). Central limit theorems for $U$-statistics of Poisson point processes. \emph{\ap}\textbf{41}, 3879--3909.
	
	\bibitem[Revuz(1970)]{R70} Revuz, D.~(1970). Mesures associ\'{e}es aux fonctionnelles additives de Markov. I. \emph{Trans. Amer.  Math. Soc.} \textbf{148}, 501--531.
	
	\bibitem[Rio(1993)]{R93} Rio, E.~(1993). Covariance inequalities for strongly mixing processes.
	\emph{Ann. Inst. H. Poincar\'{e} Prob. Statist.} \textbf{29}, 587--597.
	
	
	\bibitem[Rosenthal(1970)]{Rosenthal70} Rosenthal, H. P.~(1970). On the subspaces of $L_p(p > 2)$ spanned by sequences of independent random variables. \emph{Israel J. Math.}~\textbf{8}, 273--303.
	
	
	\bibitem[Ross(2010)]{R10} Ross, S.~(2010). \emph{Introduction to Probability Models (Tenth Edition)}. Academic Press.
	
	
	
	
	\bibitem[Schreiber and Yukich(2013)]{SY13} Schreiber, T.  and Yukich, J. E.~(2013). Limit theorems for geometric functionals of Gibbs point processes. \emph{Ann. Inst. H. Poincar\'{e} Prob. Statist.}~\textbf{49}, 1158--1182.
	
	
	
	\bibitem[Schulte and Yukich(2023)]{SchulteY} Schulte, M. and Yukich, J. E.~(2023).  Rates of multivariate normal approximation for statistics in geometric probability.  \emph{Ann.  Appl.  Prob.}, \textbf{33}, 507--548. 
	
	
	\bibitem[Shirai and Takahashi(2003)]{ST03} Shirai, T.  and Takahashi, Y.~(2003). Random point fields associated with certain Fredholm determinants I: fermion, Poisson and boson point processes. \emph{J. Funct. Anal.}~\textbf{205}, 414--463.
	
	
	
	\bibitem[Soshnikov(2002)]{S02} Soshnikov, A.~(2002). Gaussian limit for determinantal random point fields. \emph{\ap}~\textbf{30}, 171--187. 
	
	\bibitem[Steele(1981)]{Steele81} Steele, J. M.~(1981). Subadditive Euclidean functionals and nonlinear growth in geometric probability. \emph{\ap}\textbf{9}, 365--376. 
	
	\bibitem[Tong(1990)]{T90} Tong, H.~(1990). \emph{Non-linear Time Series: A Dynamical System Approach}.  Oxford University Press.
	
	\bibitem[Trauthwein and Yukich(2026)]{TY} Trauthwein, T. and Yukich, J. E.~(2026).  Second-order Poincar\'{e} inequalities and localization on the Poisson space.  arXiv:2605.23292.
	
	
	
	
	\bibitem[Wiroonsri(2019)]{W19} Wiroonsri, N.~(2019). Normal approximation for associated point processes via Stein's method with applications to determinantal point processes. \emph{J. Math. Anal. Appl.}~\textbf{480}, 123396.
	
	\bibitem[Xia and Yukich(2015)]{XY15} Xia, A. and Yukich, J. E.~(2015). Normal approximation for statistics of Gibbsian input in geometric probability. \emph{\aap}\textbf{47}, 934--972. 
	
	\bibitem[Yukich(2006)]{JY06}  Yukich, J. E.~(2006). Ultra-small scale-free geometric networks,  \emph{J. Appl. Prob.},  \textbf{43},  665--677.
\end{thebibliography}
\end{document}